\documentclass{article}

\usepackage[english]{babel}
\usepackage[utf8]{inputenc}

\usepackage[letterpaper,top=2cm,bottom=2cm,left=3cm,right=3cm,marginparwidth=1.75cm]{geometry}

\usepackage{mathtools}
\usepackage{chngcntr}
\usepackage{indentfirst}
\usepackage{amsmath}
\usepackage{graphicx}
\usepackage{amssymb}
\usepackage{amsthm}
\usepackage[colorlinks=true, allcolors=blue]{hyperref}
\usepackage{mathrsfs}
\usepackage{tikz-cd}
\usepackage{cleveref}

\usepackage[doi = false, isbn = false, url = false, style = alphabetic, maxnames = 100]{biblatex}
    \DeclareFieldFormat{postnote}{#1}
    \AtBeginDocument{%
       \def\MR#1{}
    }

\newtheorem{thm}{Theorem}[subsection]
\newtheorem{lemma}[thm]{Lemma}
\newtheorem{cor}[thm]{Corollary}
\newtheorem{rmk}[thm]{Remark}
\newtheorem{prop}[thm]{Proposition}
\newtheorem{defn}[thm]{Definition}
\newtheorem{conj}[thm]{Conjecture}

\title{Base change fundamental lemma for Bernstein centers of principal series blocks}
\author{Shenghao Li}
\numberwithin{equation}{subsubsection}
\renewcommand{\theequation}{%
  \ifnum\value{subsection}>0 %
    \ifnum\value{subsubsection}>0 %
      \thesubsubsection.\arabic{equation}
    \else
      \thesubsection.\arabic{equation}
    \fi
  \else
    \thesection.\arabic{equation}
  \fi
}
\numberwithin{thm}{subsubsection}
\renewcommand{\thethm}{%
  \ifnum\value{subsection}>0 %
    \ifnum\value{subsubsection}>0 %
      \thesubsubsection.\arabic{thm}
    \else
      \thesubsection.\arabic{thm}
    \fi
  \else
    \thesection.\arabic{thm}
  \fi
  }

\newtheorem{claim}{Claim}

\newtheoremstyle{normalexample}
  {3pt} 
  {3pt} 
  {\normalfont} 
  {} 
  {\bfseries} 
  {.} 
  {.5em} 
  {} 

\theoremstyle{normalexample}
\newtheorem{example}{Case}

\begin{document}
\maketitle

\begin{abstract}
    Let $G$ be an unramified group over a $p$-adic field $F$. This article introduces a base change homomorphism for the Bernstein center of a principal series block, and proves that two functions related by this base change homomorphism are associated. This result provides new evidence for the conjecture on twisted endoscopic transfer of the stable Bernstein center proposed by T. Haines, which will be applied to a general conjecture on test functions for Shimura varieties due to R. Kottwitz and T. Haines.
\end{abstract}

\tableofcontents

\section{Introduction}

    Let $F$ be a $p$-adic field, and $F_r/F$ be the unique degree $r$ unramified extension of $F$ contained in some algebraic closure $\bar{F}$ of $F$. Let $\theta$ denote a generator of the Galois group $\textnormal{Gal}(F_r/F)$. Let $G$ be an unramified connected reductive group over $F$. Use the same symbol $\theta$ to denote the automorphism on $G(F_r)$ induced by $\theta$. For any algebraic group $H$ over $F$, we use $H$ to denote $H(F)$ and $H_r$ to denote $H(F_r)$ when there is no ambiguity.

    Consider the concrete norm map $N_r:G_r\rightarrow G_r$ given by $N_r\delta=\delta\theta(\delta)\cdots \theta^{r-1}(\delta)$. We can therefore define a norm map $\mathscr{N}$, from stable $\theta$-conjugacy classes in $G_r$ to stable conjugacy classes in $G$ (see \cite{kottwitz1982rational} for details). Let $\mathcal{H}(G_r)$ denote the convolution algebra of locally constant and compactly supported $\mathbb{C}$-valued functions on $G_r$. For any function $\phi\in \mathcal{H}(G_r)$, we can define the stable twisted orbital integral $\textnormal{SO}_{\delta\theta}^{G_r}(\phi)$ of $\phi$ for any element $\delta\in G_r$ with semisimple norm. For a more precise definition of stable twisted orbital integrals $\textnormal{SO}_{\delta\theta}^{G_r}$, see \cite{kottwitz1986base}.

    In this article, we mainly consider the matching of certain functions in $\mathcal{H}(G_r)$ and $\mathcal{H}(G)$, respectively. We make the following definition:
    \begin{defn}
        Two functions $\phi\in\mathcal{H}(G_r)$ and $f\in\mathcal{H}(G)$ are associated (or have matching stable orbital integrals) if the following equation holds: for every semisimple element $\gamma\in G$,
        $$\textnormal{SO}_{\gamma}^{G}(f)=\sum_{\delta}\Delta(\gamma,\delta)\textnormal{SO}_{\delta\theta}^{G_r}(\phi),$$
        where the sum is over stable $\theta$-conjugacy classes $\delta\in G_r$ with semisimple norm, and where $\Delta(\gamma,\delta)=1$ if $\mathscr{N}\delta=\gamma$ and $\Delta(\gamma,\delta)=0$ otherwise.
    \end{defn}

    The case of spherical Hecke algebras is studied first. Let $K_r\subset G_r$ and $K\subset G$ denote hyperspecial maximal compact subgroups associated to a hyperspeical vertex in the Bruhat-Tits building $\mathscr{B}(G)$ for $G(F)$. Denote the Hecke algebra of all $K_r$ (resp. $K$) bi-invariant functions in $\mathcal{H}(G_r)$ (resp. $\mathcal{H}(G)$) by $\mathcal{H}(G_r,K_r)$ (resp. $\mathcal{H}(G,K)$). The Satake isomorphism induces a base change homomorphism
    $$b:\mathcal{H}(G_r,K_r)\rightarrow \mathcal{H}(G,K).$$
    We have the following base change fundamental lemma for spherical Hecke algebras, proved in \cite{clozel1990fundamental} and \cite{labesse1990fonctions}:
    \begin{thm}
        If $\phi\in \mathcal{H}(G_r,K_r)$, then $\phi$ and $b\phi$ are associated.
    \end{thm}
    This theorem of Clozel and Labesse contributes a lot in Kottwitz's work on Shimura varieties with good reduction at $p$.

    Decades later, T. Haines proved analogous base change fundamental lemmas for centers of parahoric Hecke algebras (\cite{haines2009base}) and the Bernstein center of depth-zero principal series block (\cite{haines2012base}). The first result plays an important role in the study of Shimura varieties with parahoric level structure at $p$, and the second result is crucial to study Shimura varieties with $\Gamma_1(p)$-level structure at $p$. 

    After that, T. Haines proposed a conjecture on twisted endoscopic transfer of the stable Bernstein center (\cite[Conjecture 6.2.3]{haines2014stable}), which is a generalization of the above three results. Let $\mathfrak{Z}^{\textnormal{st}}(G)$ denote the stable Bernstein center of $G$, and $b_r':\mathfrak{Z}^{\textnormal{st}}(G_r)\rightarrow \mathfrak{Z}^{\textnormal{st}}(G)$ denote the base change homomorphism (see \Cref{stable bernstein center} for more details). We assume that $G$ satisfies LLC+ (see \cite[\S5.2]{haines2014stable} for details). Since $G$ is a quasi-split group, we can naturally embed $\mathfrak{Z}^{\textnormal{st}}(G)$ into the Bernstein center $\mathfrak{Z}(G)$ of $G$ (\cite[Corollary 5.5.2]{haines2014stable}). By viewing $\mathfrak{Z}^{\textnormal{st}}(G_r)$ as a subring of $\mathfrak{Z}(G)$, we have the following conjecture due to T. Haines:
    \begin{conj}[Haines]\label{conjecture}
        If $\phi\in\mathcal{H}(G_r)$ and $f\in\mathcal{H}(G)$ are associated, then $Z_r*\phi$ and $b_r'(Z_r)*\phi$ are associated for any $Z_r\in \mathfrak{Z}^{\textnormal{st}}(G_r)$.
    \end{conj}

    In addition to the three results above, P. Scholze also proved the above conjecture when $G=\textnormal{GL}_n$ (see \cite[Theorem C]{scholze2013local}). The goal of this paper is to generalize Haines's result (\cite{haines2012base}) to general principal series blocks, removing the restriction on depths. For the remainder of this paper, we will refer to a principal series block simply as a principal block. 

    To state the main theorem, we need more notations and definitions. Let $A$ denote a maximal $F$-split torus, and let $T=C_G(A)$ denote the centralizer of $A$ in $G$. Therefore, $T$ is a maximal torus of $G$. Choose a Borel subgroup $B=TU$ of $G$, and an Iwahori subgroup corresponding to $B$. 

    Since we assume $G$ is unramified, this means that $T$ is also unramified. We denote the unique maximal compact subgroup of $T$ (resp. $T_r$) by $^0T$ (resp. $^0T_r$). One can see that $^0T=I\cap T$. Let $^0\chi$ be a character from $^0T$ to $\mathbb{C}^{\times}$. Define $^0\chi_r={^0\chi}\circ N_r$ as a character in $^0T_r$. Let $
    \chi$ denote an arbitrary extension of $^0\chi$ from $^0T$ to $T$. Similarly, define $\chi_r=\chi\circ N_r$ as an extension of $^0\chi_r$ from $^0T_r$ to $T_r$. Let $\mathfrak{s}=[T,\chi]_G$ (resp. $\mathfrak{s}_r=[{T_r},{\chi_r}]_{G_r}$) be an inertial equivalence class and $\mathfrak{R}_{\mathfrak{s}}(G)$ (resp. $\mathfrak{R}_{\mathfrak{s}_r}(G_r)$) denote the corresponding Bernstein block (for more details on Bernstein blocks, see \Cref{type definition} and \Cref{types}). 
    
    Next, we need to construct types for $\mathfrak{s}$ and $\mathfrak{s}_r$. However, Roche's construction (see \cite{roche1998types}) only works for split groups. In \Cref{types}, we generalize this construction to unramified groups:
    \begin{thm}\label{type construction}
        Fix an inertial equivalence class $\mathfrak{s}=[{T},{\chi}]_G$. We can then construct a concrete type $(J_{\mathfrak{s}},\rho_{\mathfrak{s}})$ of   
        $\mathfrak{R}_{\mathfrak{s}}(G)$, with the property that $J_{\mathfrak{s}}$ is an open compact subgroup of $I$ and $\rho$ is a character on $J_{\mathfrak{s}}$. In particular, this construction gives us the same construction as in \cite{roche1998types} if $G$ is split, and the same construction as in \cite[\S3]{haines2012base} if the depth of $\chi$ is zero.
    \end{thm}
    Note that by \cite{fintzen2022tame}, the construction of types works if $p$ is not divided by the cardinality of the absolute Weyl group. For the remainder of this article, we may and we do make this assumption. 

    Suppose that we construct types $(J,\rho)$ and $(J_r,\rho_r)$ for $\mathfrak{s}$ and $\mathfrak{s}_r$ respectively (note that $\rho$ and $\rho_r$ are characters). Define the Hecke algebra $\mathcal{H}(G,\rho)$ as follows:
    $$\mathcal{H}(G,\rho)\coloneqq \{f\in \mathcal{H}(G)|f(j_igj_2)=\rho(j_1)^{-1}f(g)\rho(j_2)^{-1},\forall j_1,j_2\in J,\forall g\in G\}.$$
    Convolution is defined using the Haar measure which gives $J$ volume 1.
    
    Denote the unit element in $\mathcal{H}(G,\rho)$ (resp. $\mathcal{H}(G_r,\rho_r)$) by $e_{\rho}$ (resp. $e_{\rho_r}$) and the center of the corresponding Hecke algebra by $\mathcal{Z}(G,\rho)$ (resp. $\mathcal{Z}(G_r,\rho_r)$). Note that the above concrete construction of types gives us a concrete realization of the Bernstein center of the corresponding Bernstein block. In other words, the Bernstein center of the corresponding Bernstein block is isomorphic to $\mathcal{Z}(G,\rho)$. We have the following corollary:
    \begin{cor}
        There is a base change homomorphism $b_r$ from $\mathcal{Z}(G_r,\rho_r)$ to $\mathcal{Z}(G,\rho)$:
        \begin{equation}
        b_r:
        \mathcal{Z}(G_r,\rho_r) \rightarrow \mathcal{Z}(G,\rho).
    \end{equation}
    \end{cor}
    See \Cref{stable bernstein center} for more details. This base change homomorphism is analogous to the base change homomorphism defined in \cite{haines2012base}. We can now state our main theorem:
    \begin{thm}\label{main theorem}
        For any $\phi\in \mathcal{Z}(G_r,\rho_r)$, the functions $\phi$ and $b_r\phi$ are associated.
    \end{thm}

    This type of theorem plays a role in the pseudostabilization process on the simple Shimura varieties of R. Kottwitz (\cite{kottwitz1992lambda}).
    Moreover, in \Cref{stable bernstein center}, we prove that the base change homomorphism $b_r$ is compatible with the base change homomorphism $b_r'$ defined on the stable Bernstein center. In other word, for any $Z_r\in \mathfrak{Z}^{\textnormal{st}}(G_r)$, we have 
    $$b_r'(Z_r)*e_{\rho}=b_r(Z_r*e_{\rho_r}).$$
    Consequently, it also gives evidence for \Cref{conjecture}.

    \noindent\textbf{Important fact:} to prove the main theorem, we don't need any form of LLC or LLC+. The main theorem is true for any unramified group $G$ over a $p$-adic field $F$ such that $p$ is not divided by the cardinality of the absolute Weyl group. Here we introduce the stable Bernstein center and LLC+ just to make connections with \Cref{conjecture}. 

    There are two main new results of this paper. The first is the construction of types for principal series representations of unramified groups. Roche's work is restricted to split case because some Hecke algebra isomorphisms can only work for split reductive groups. We combine Roche's construction with the construction of Kim and Yu (see \cite{kim2017construction}) and give a concrete construction of types for all principal series blocks. The second is the descent formula. Unlike the depth-zero or parahoric case, the descent formula is more difficult in our case. In parahoric case or depth-zero case, by Bruhat decomposition, the double cosets $P\backslash G/\mathcal{J}$ can be represented by a subset of the finite Weyl group, where $P$ is a standard parabolic subgroup and $\mathcal{J}$ denotes a parahoric subgroup. Consequently, the constant term of a function in the corresponding Hecke algebra is easy to deal with. However, in deeper levels, we don't have decomposition laws in such a clean form. In other words, $P\backslash G/J$ cannot be represented by a subgroup of the finite Weyl group, where $P$ is a standard parabolic and $J$ is the group of the type we construct in \Cref{type construction}. The representative elements of $P\backslash G/J$ can be written in the form of $wi$, where $w$ is an element in the finite Weyl group, and $i$ is an element in the Iwahori subgroup $I$. However, we have the surprising result showing that only the representative elements in the finite Weyl group (i.e. $i=1$) will contribute to the constant term, while others will produce items with vanishing (twisted) orbital integrals. This guarantees that we can still write a descent formula similar to what we have in \cite{haines2009base}. For more details, see \Cref{descent formula section}.

    Similar to \cite{haines2012base}, the proof of \Cref{main theorem} could be divided into two parts: the first part is to reduce $G$ to some simple forms and reduce $\gamma$ to be strongly regular elliptic. The second step is to use global trace formula to prove the theorem for the above case. Note that the reduction step is flawed in \cite{haines2012base} that we cannot reduce to adjoint groups, and we must stop the reduction steps at groups with simply connected derived subgroup and with center an induced torus. Thus, the global trace formula needs to be rewritten. Reducing to such groups instead of to adjoint groups requires some changes in the construction of local data from the simple form of the trace formula. Details may be found in the forthcoming PhD thesis of Weimin Jiang (\cite{Weimin}).

    Our presentation is arranged as follows. In \Cref{basic notation}, we introduce some notations which will be used throughout this article. In \Cref{type definition}, we go over the general theory of Bernstein blocks and types. \Cref{type definition} is the generalization of Roche's work to unramified groups. We define Bernstein center, stable Bernstein center and the base change homomorphism in \Cref{stable bernstein center}. Moreover, some basic properties of them will also be studied in this part. \Cref{descent formula section} is all about descent formulas. This is the most technical section of this article. Unlike the descent formula in \cite{haines2009base}, we will use $K$-averaged constant term and concrete calculations of the commutators of root groups to reduce $\gamma$ to elliptic semisimple elements. In \Cref{section 7}, we further reduce the fundamental lemma to the case where $G_{\textnormal{der}}=G_{\textnormal{sc}}$, the center of $G$ is an induced torus and $\gamma$ is elliptic and strongly regular semisimple. We modify some of the arguments in \cite{haines2012base} to suit our case. \Cref{section 8} gives us an analogues of Labesse's elementary functions adapted to the Bernstein block $\mathfrak{R}_{\mathfrak{s}}(G)$. Finally in \Cref{section 9}, we combine all materials, together with local data introduced in \cite{Weimin}, to finish the proof of \Cref{main theorem}.

    \; \\

    \noindent \textit{Acknowledgments:} I would like to thank my advisor, Thomas Haines, for posing this question and countless helpful conversations and advice. I would like to thank Jeffrey Adler and Kazuma Ohara for valuable advice. This research was partially supported by NSF grants DMS 2200873 and DMS 1801352.

\section{Basic notations}\label{basic notation}

    We will introduce more notations and definitions which will be used throughout the article. Recall that $F$ is a p-adic field of characteristic $0$ and $F_r/F$ is the unique unramified extension of degree $r$. Denote the ring of integers of $F$ by $\mathfrak{O}$, and the prime ideal by $\mathfrak{p}$. Let $k_q=\mathfrak{O}/\mathfrak{p}$ denote the residue field, where $q$ denotes the cardinality of the residue field.

    Assume $G$ is split over $E$, where $E/F$ is a finite unramified extension. Let $h$ be a generator of $\textnormal{Gal}(E/F)$ (sometimes in this article we will also use $h$ to denote functions, but this will be easy to distinguish). Let $\mathfrak{O}_E$ denote the ring of integers of $E$, and $\mathfrak{p}_E$ denote the prime ideal in $\mathfrak{O}_E$. Denote the maximal compact open subgroup of $T(E)$ by $^0T(E)$.

    Let $G'$ denote the derived subgroup of $G$ and $T'=T\cap G'$ the maximal torus of $G'$. Denote the relative Weyl group of $G$ over $F$ and $F_r$ by $W$ and $W_r$. Also, denote the absolute Weyl group by $W_E$. Therefore, our restriction on $p$ can be expressed as $p\nmid |W_E|$, where $|W_E|$ denotes the cardinality of $W_E$. Note that $(W_r)^{\theta}=(W_E)^{h}=W.$

    Use $\Phi_F(G,T)$ and $\Phi_E(G,T)$ to denote the relative root system and the absolutely root system, respectively. Use the same symbol to denote the automorphism on $\Phi_E(G,T)$ induced by $h$. Let $\Phi_E^+(G,T)$ denote the set of positive roots determined by $B$. Since $G$ is unramified, $\Phi_E^+(G,T)$ is stable under $h$. We may fix root group isomorphism $u_{\alpha}:\mathbb{G}_a\rightarrow U_{\alpha}$ for $\alpha\in \Phi_E(G,T)$ such that the conditions in \cite[\S2]{roche1998types} hold. The Galois action on the root groups is as follows:
    $$h(u_{\alpha}(a))=u_{h(\alpha)}(x_{\alpha}h(a))$$
    for any $a\in E$ and $\alpha\in \Phi_E(G,T)$, where $x_{\alpha}$ is an element in $\mathfrak{O}_E^{\times}$ only determined by the $h$-orbit of $\alpha$. Let $\mathfrak{g}_{\alpha}$ denote the one-dimensional eigenspace of $\mathfrak{g}_E$ corresponding to $\alpha$.

    Let $H$ denote any group. For $h\in H$ and a subset $S\subset H$, set $^hS\coloneqq hSh^{-1}$. Let $H'\subset H$ denote a subgroup, and $\sigma$ a representation on $H'$. We set $^h\sigma$ be the representation on $^hH'$ such that $^h\sigma(h')=\sigma(h^{-1}h'h)$ for any $h'\in {^hH'}$. 

\section{Bernstein blocks and types}\label{type definition}

    In this section, we recall the Bernstein decomposition of $\mathfrak{R}(G)$ (the category of smooth complex representations of $G$) and the definition of types and some consequences. We follow the content in \cite[\S 7]{roche1998types}.

    A cuspidal pair is a pair $(M,\sigma)$, consisting of a Levi subgroup $M\subset G$ and an irreducible supercuspidal representation of $M$. We say two pairs, $(M_1,\sigma_1)$ and $(M_2,\sigma_2)$, are equivalent if there exists a $g\in G$ such that $M_2={^gM_1}$ and $\sigma_2\simeq {^g\sigma_1}$. Further, we say they are inertial equivalent if there exists a $g\in G$ and an unramified character $\nu$ of $M_2$ (see \cite[\S 3.3.1]{haines2014stable} for the definition of unramified characters) such that 
    $$M_2={^gM_1}\;\;\;\;\textnormal{and} \;\;\;\; ^g\sigma_1\simeq \sigma_2\otimes \nu.$$
    Denote the equivalent class of $(M,\sigma)$ by $(M,\sigma)_G$ and the inertial equivalent class of $(M,\sigma)$ by $\mathfrak{s}=[M,\sigma]_G$ (we omit the subscript when there is no ambiguity). In addition, use $\mathfrak{B}(G)$ to denote the set of all inertial equivalence classes in $G$.

    Let $\textnormal{Irr}(G)$ denote the set of equivalent classes of smooth irreducible representations of $G$. We have the following well-defined map (\cite{zelevinsky1977induced}):
    $$\textnormal{Irr}(G)\rightarrow \{\textnormal{equivalence classes of cuspidal pairs}\}$$
    $$\pi\rightarrow \textnormal{class of } (M,\sigma) \textnormal{ if } \pi\simeq \textnormal{a } G\textnormal{-subquotient of } i_P^G\sigma$$
    where $P$ is any parabolic subgroup of $G$ with Levi factor $M$ and $i_P^G$ is the normalized parabolic induction functor. The image of $\pi$ under this map is called the supercuspidal support of $\pi$. We use $\mathfrak{J}(\pi)$ to denote the inertial equivalence class of its supercuspidal support. For a fixed $\mathfrak{s}\in \mathfrak{B}(G)$, we denote by $\mathfrak{R}_{\mathfrak{s}}(G)$ the full subcategory of $\mathfrak{R}(G)$ defined as follows:
    \begin{defn}
        Let $(\pi,V)\in \mathfrak{R}(G)$. Then $(\pi,V)\in \mathfrak{R}_{\mathfrak{s}}(G)$ if and only if every irreducible $G$-subquotient $\pi_0$ of $\pi$ satisfies $\mathfrak{J}(\pi_0)=\mathfrak{s}$.
    \end{defn}

    And we have the following theorem, known as Bernstein decomposition:
    \begin{thm}[{\cite[Theorem 7.2]{roche1998types}}]
        The category $\mathfrak{R}(G)$ is the direct product
        $$\mathfrak{R}(G)=\prod_{\mathfrak{s}\in \mathfrak{B}(G)}\mathfrak{R}_{\mathfrak{s}}(G)$$
        of the subcategories $\mathfrak{R}_{\mathfrak{s}}(G)$ as $\mathfrak{s}$ ranges through $\mathfrak{B}(G)$.

        More concretely, for any $(\pi,V)\in \mathfrak{R}(G)$, and any $\mathfrak{s}\in \mathfrak{B}(G)$, $V$ has a unique maximal $G$-stable subspace $V_{\mathfrak{s}}\in \mathfrak{R}_{\mathfrak{s}}(G)$ and
        $$V=\bigoplus_{\mathfrak{s}\in \mathfrak{B}(G)}V_{\mathfrak{s}}.$$
    \end{thm}

    Next we are going to recall the definition of types:

    \begin{defn}
        Fix an $\mathfrak{s}\in \mathfrak{B}(G)$. An $\mathfrak{s}$-type in $G$ is a pair $(K,\rho)$, where $K$ is a compact open subgroup of $G$ and $\rho$ is an irreducible smooth representation of $K$, with the following property: for an irreducible smooth representation $\pi$ of $G$, $\mathfrak{J}(\pi)=\mathfrak{s}$ if and only if $\rho\subset \pi|_K$.
    \end{defn}
    \begin{rmk}
        There is also a definition of $\mathfrak{s}$-type for $\mathfrak{s}\subset \mathfrak{B}(G)$, which is not a singleton. However, in this paper, we will only discuss the case where $\mathfrak{s}$ is a singleton.
    \end{rmk}
    \begin{thm}[{\cite[Theorem 7.5]{roche1998types}}]
        Let $\mathfrak{R}_{\rho}$ denote the full subcategory of $\mathfrak{R}(G)$ consisting of all $(\pi,V)\in\mathfrak{R}(G)$ such that $V=V[\rho]$, the space generated by $V^{\rho}$ (the $\rho$-isotypic vectors). Then we have the following statements:
        \begin{itemize}
            \item[(1)] The categories $\mathfrak{R}_{\rho}(G)$ and $\mathfrak{R}_{\mathfrak{s}}(G)$ are equal as subcategories of $\mathfrak{R}(G)$. In particular, $\mathfrak{R}_{\rho}(G)$ is closed under subquotients.
            \item[(2)] Then functor
            $$V\rightarrow V^{\rho}:\mathfrak{R}(G)\rightarrow \mathcal{H}(G,\rho)\textnormal{-mod}$$
            is an equivalence of categories.
        \end{itemize}
    \end{thm}

    By the above theorem, we can identify the Bernstein center of the Bernstein block $\mathfrak{R}_{\mathfrak{s}}(G)$ corresponding to the inertial equivalence class $\mathfrak{s}$ by $\mathcal{Z}(G,\rho)$, the center of the Hecke algebra $\mathcal{H}(G,
    \rho)$.

\section{The construction of types for principal blocks}\label{types}

    In this section, we are going to construct types for principal series blocks of unramified groups. Here, a principal series block means a Bernstein block $\mathfrak{R}_{\mathfrak{s}}(G)$ corresponding to an inertial equivalence class $[M,\sigma]_G\in\mathfrak{B}(G)$ such that $M=T$ and $\chi$ is a character from $T$ to $\mathbb{C}^{\times}$. Since we only consider the inertial equivalence class of a pair $(T,\chi)$, we will use the symbol $({^0T},{^0\chi})_G$ to denote the inertial equivalence class $[T,\chi]_G$, where $^0\chi=\chi|_{^0T}$. For short, we will call such a block by a principal block and the corresponding inertial equivalence class by a principal class. In all future sections, fix a principal class $\mathfrak{s}=({^0T},{^0\chi})_G$ and a principal block $\mathfrak{R}_{\mathfrak{s}}(G)$.

\subsection{Main result}\label{main result}
    As in \cite{roche1998types}, we have the construction of types for principal blocks of split reductive groups. However, to generalize this result to unramified groups, we cannot use the same method. Instead, we need to combine Roche's construction with Kim-Yu's construction (see \cite{kim2017construction}).

    In previous sections, we already fixed a Borel subgroup $B=TU$. Therefore, we also fix a set of positive roots $\Phi_{E}(G,T)^+$ in the absolute root system $\Phi_{E}(G,T)$. Denote the opposite Borel subgroup by $\bar{B}=T\bar{U}$. Let $^0\chi_E={^0\chi}\circ N_E$ be a character defined on $^0T(E)$, where $N_E$ denotes the norm map from $^0T(E)$ to $^0T$. Therefore, $\mathfrak{s}_E\coloneqq ({^0T(E)},{^0\chi_E})_{G(E)}$ is also a principal class on $\mathfrak{B}(G(E))$.

    The conductor of $\alpha$, $\textnormal{cond}(\alpha)$, is the smallest positive integer $l$ such that $^0\chi_E\circ \alpha^{\vee}(1+\mathfrak{p}_E^l)=1$. We define a concave function $f_{\mathfrak{s}}:\Phi_E(G,T)\rightarrow \mathbb{N}$ as follows:
    \begin{equation}
    f_{\mathfrak{s}}(\alpha)=
       \begin{cases}
       [\textnormal{cond}(\alpha)/2] & \alpha\in \Phi_{E}(G,T)^+ \\
         [(\textnormal{cond}(\alpha)+1)/2] & \alpha\in \Phi_{E}(G,T)^-.
    \end{cases}
    \end{equation}
    Here $[x]$ denotes the largest integer $\leq x$. For the proof of $f_{\mathfrak{s}}$ being a concave function, see \cite[\S3]{roche1998types}.

    Define $U_{E,f_{\mathfrak{s}}}=\langle u_{\alpha}(\mathfrak{p}_E^{f_{\mathfrak{s}}(\alpha)}):\alpha\in \phi_E(G,T)\rangle$ and $J_{E,{\mathfrak{s}}}=\langle U_{f_{\mathfrak{s}}},{^0T(E)}\rangle$. Note that by \cite[Lemma 3.2]{roche1998types}, the group $U_{E,f_{\mathfrak{s}}}$ and $J_{E,{\mathfrak{s}}}$ has Iwahori decomposition with respect to any parabolic subgroup $P\subset G$. Since the group $G$ is quasi split, we can deduce that the group $U_{f_{\mathfrak{s}}}=U_{E,f_{\mathfrak{s}}}\cap G(F)=(U_{E,f_{\mathfrak{s}}})^h$ and $J_{{\mathfrak{s}}}=J_{E,{\mathfrak{s}}}\cap G(F)=(J_{E,{\mathfrak{s}}})^h$ also has Iwahori decomposition over and parabolic subgroup $P\subset G$. Note that $J_{E,{\mathfrak{s}}}\cap T(E)={^0T(E)}$ and $J_{{\mathfrak{s}}}\cap T={^0T}$. The main result is as follows:
    \begin{thm}\label{thm:main 4}
    The pair $(J_{\mathfrak{s}},\rho_{\mathfrak{s}})$ is an ${\mathfrak{s}}$-type, where $J_{\mathfrak{s}}$ is defined above, and $\rho_{\mathfrak{s}}$ is the character on $J_{\mathfrak{s}}$ such that $\rho_{\mathfrak{s}}|_{J_{\mathfrak{s}}\cap U}$ and $\rho_{\mathfrak{s}}|_{J_{\mathfrak{s}}\cap \bar{U}}$ are trivial, and $\rho_{\mathfrak{s}}|_{J_{\mathfrak{s}}\cap T}={^0\chi}$.
    \end{thm}
    \begin{rmk}
        This is a generalization of \cite[Theorem 7.7]{roche1998types}. If the group $G$ is split, then this theorem will show that the type constructed in \cite{roche1998types} and \cite{kim2017construction} are actually the same.
    \end{rmk}

\subsection{Construction of types following Kim-Yu}\label{concrete construction of types}

    In this subsection, we are going to recall the construction of types in \cite{kim2017construction}. Let $d\in \mathbb{Z}_{\geq1}$ and let $\vec{G}=(G=G^1\supsetneq G^2\supsetneq\dots\supsetneq G^d)$ be a sequence of tamely ramified twisted Levi subgroups of $G$. Let $M^d$ be a Levi subgroup of $G^d$ and $A_{M^d}$ the split component of the center of $M^d$. Denote the centralizer of $A_{M^d}$ in $G^i$ by $M^i$ for $1\leq i\leq d$. According to \cite[2.4 Lemma(a),(b)]{kim2017construction}, $M^i$ is a Levi subgroup of $G^i$ and a tamely ramified twisted Levi subgroup of $M^1$.
    \begin{defn}[{\cite[\S7.2]{kim2017construction}}]
        A $G$-datum is a $5$-tuple $\Sigma:=((\vec{G},M^d),(y,\{\iota\}),\vec{r},(K_{M^d},\rho_{M^d}),\vec{\phi})$ satisfying the following:
        \begin{description}
            \item[\textbf{D1}] $\vec{G}=(G=G^1\supsetneq G^2\supsetneq\dots\supsetneq G^d) $ is a sequence of twisted Levi subgroups of $G$ that split over a tamely ramified extension of $F$ for some $d\in \mathbb{Z}_{\geq1}$, and $M^d$ is a Levi subgroup of $G^d$. Let $\vec{M}=(M^1\supset M^2\supset\dots\supset M^d)$ be constructed as above.
            \item[\textbf{D2}] $\vec{r}=(r_1,r_2,\dots,r_d)$ is a sequence of real numbers satisfying $r_1\geq r_2>\dots>r_d>0$.
            \item[\textbf{D3}] $y$ is a point of $\mathcal{B}(M^d,F)$, and $\{\iota\}$ is a commutative diagram
            \[
            \begin{tikzcd}
            \mathcal{B}(G^d,F) \arrow[r]  & \mathcal{B}(G^{d-1},F)   \arrow[r]& \dots\arrow[r] & \mathcal{B}(G^1,F) \\
            \mathcal{B}(M^d,F) \arrow[r] \arrow[u]& \mathcal{B}(M^{d-1},F) \arrow[r] \arrow[u]& \dots \arrow[r] & \mathcal{B}(M^1,F)\arrow[u]
            \end{tikzcd}
            \]
            of admissible embeddings of buildings that is $\vec{s}$-generic relative to $y$ in the sense of \cite[3.5 Definition]{kim2017construction}, where $\vec{s}=(r_2/2,\dots,r_d/2,0)$. We identify a point in $\mathcal{B}(M^d,F)$ with its image via the embedding $\{\iota\}$.
            \item[\textbf{D4}] $K_{M^d}$ is a compact open subgroup of $M^d(F)_y$ (the stabalizer of $y$) containing $M^d(F)_{y,0}$ and $\rho_{M^d}$ is an irreducible smooth representation of $K_{M^d}$ such that $((G^d,M^d),(y,\iota:\mathcal{B}(M^d,F)\rightarrow\mathcal{B}(G^d,F)),(K_{M^d},\rho_{M^d}))$ is a depth zero $G^d$-datum in the sense of \cite[\S7.1]{kim2017construction}.
            \item[\textbf{D5}] $\vec{\phi}=(\phi_1,\dots,\phi_d)$ is a sequence of characters, where $\phi_i$ is a character of $G^i(F)$. If $r_1=r_2$, we set $\phi_1=1$; otherwise $\phi_d$ is of depth $r_d$. We also require that $\phi_i$ is $G^{i-1}$-generic of depth $r_i$ in the sense of \cite[Definition 3.5.2]{fintzensupercuspidal} for all $2\leq i\leq d$.
        \end{description}
    \end{defn}
    \begin{rmk}
        For \textbf{D5}, the original definition should be $G^{i-1}$-generic of depth $r_i$ relative to the point $y\in \mathcal{B}(G^i,F)$. However, by \cite[Definition 3.5.2]{fintzensupercuspidal}, the choice of $y$ is not important; it could be any point of the building. In fact, in this section, we will simply choose the origin of the apartment $\mathcal{A}(G^i,T,F)$.
    \end{rmk}
       
    Following \cite[7.4 and 7.5 Theorem]{kim2017construction}, the pair $\mathcal{J}=(K, \rho)$ will be a $G$-cover of $\mathcal{J}_M=(K_M,\rho_M)$, where $K=K_{M^d}G^d_{y,0}G^{d-1}_{y,r_d/2}\dots G^1_{y,r2/2}$ and $K_M=K_{M^d}M^{d-1}_{y,r_d/2}\dots M^1_{y,r_2/2}$, and $\rho$ (resp. $\rho_1^M$) is an irreducible representation of $K$ (resp. $K_M)$. The concrete construction of $\rho_M$ is complicated, but for principal blocks, it is easy to describe and we will discuss it in the next paragraph.

    In order to obtain a type for principal blocks, we need to require $M^1=T$. Therefore, the twisted Levi sequence $\vec{M}$ could only be the constant sequence $(T,\dots,T)$. In order to obtain a singleton, the group $K_{M^d}$ will just be the group $T_{y,0}={^0T}$. Therefore, the group $K_M$ will be the same as $^0T$, and the irreducible representation $\rho_M$ is just $^0\chi=\rho_{M^d}\otimes(\phi_d|_{^0T})\otimes\dots\otimes (\phi_1|_{^0T})$. Since $\mathcal{J}=(K, \rho)$ is a $G$-cover of $\mathcal{J}_M=(K_M,\rho_M)$, we can say $\rho$ is the character on $K$ such that $\rho|_{K\cap U}$ and $\rho|_{K\cap \bar{U}}$ are trivial, and $$\rho|_{K\cap T}=\rho|_{^0 T}={^0\chi}.$$ In conclusion, $(K,\rho)$ will be an $\mathfrak{s}=({^0T},{^0\chi})$-type.

\subsection{Construction of G-datum}
    
    To prove \Cref{thm:main 4}, we need to construct a $G$-datum. We will first define the twisted Levi sequence $(\vec{G},M^d)$, and then the character sequence $\vec{\phi}$ and the real number sequence $\vec{r}$. Lastly, we will define $(y,\{\iota\})$ and $(K_{M^d},\rho_{M^d})$, the depth zero datum and the generic embeddings of buildings. 
    
    Only in this section, for any reductive group $H$, we will use the symbol $H_r$ (resp. $H_{r+})$ to denote the Moy-Prasad filtration subgroup of depth $r$ (resp. $r+$) of $H$.

\subsubsection{The twisted Levi sequence $(\vec{G},M^d)$}\label{construction of twisted Levi}

    We start with a given $\mathfrak{s}$-type, where $\mathfrak{s}=({^0T},{^0\chi})$. Let $\chi$ denote an arbitrary character on $T$ which is an extension of the character $^0\chi$. Since $p$ does not divide the order of the absolute Weyl group of $G$, we can identify $\mathfrak{g}$ with $\mathfrak{g}^*$ by a bilinear form $B(\;,\;)$ defined on $\mathfrak{g}$ (see \cite[Proposition 4.1]{adler2000intertwining}). Assume that the depth of the character $^0\chi$ is $r>0$ (we will deal with the depth zero case later). It is trivial on $T_{r+}$ but non-trivial on $T_r$. Fix an additive character $\psi$ on $F$ which is trivial on $\mathfrak{p}$ but not trivial on $\mathfrak{O}$. By the isomorphism $T_r/T_{r+}\simeq \mathfrak{t}_r/\mathfrak{t}_{r+}$, we can view the character $^0\chi|_{T_r/T_{r+}}$ as an additive character on $\mathfrak{t}_r/\mathfrak{t}_{r+}$. By composing with the character $\psi$, we can regard the character $^0\chi|_{T_r/T_{r+}}$ as an element $a_0\in \mathfrak{t}_{-r}/\mathfrak{t}_{-r+}$ (see \cite[Lemma 4.3]{roche1998types}). We define an element $a\in\mathfrak{t}_l\backslash\mathfrak{t}_{l+}$ to be a good element if for any absolute root $\alpha\in \Phi_E(G,T)$, we have that either $\mathrm{d}\alpha(a)=0$ or $\textnormal{val}(\mathrm{d}\alpha(a))=l$ (see \cite[Definition 5.2]{adler2000intertwining}). According to \cite[Theorem 3.3]{fintzen2021tame}, by our restriction on $p$, there exists a good element for any coset of $\mathfrak{t}_{-r}/\mathfrak{t}_{-r+}$. We choose a good element $a\in a_0$.

    According to \cite[Theorem 4.15]{roche1998types}, the centralizer $C_G(a)$ of $a$ in the group $G$ is a twisted Levi subgroup of $G$. If $a\in \mathfrak{z}(\mathfrak{g})$, then this character must be trivial on $T'_r$ (note that this is an equivalent condition: if the character is trivial on $T'_r$, then the good element must lie in $\mathfrak{z}(\mathfrak{g})$). Therefore, we can choose a character $\chi_1'$ on $T$, such that $\chi_1'|_{T_r}=\chi|_{T_r}$ and $\chi_1'|{T'}=1$. Let $\chi'=\chi\otimes \chi_1'^{-1}$; we get a new character on $T$ with smaller depth. Keep doing this until the good element we choose is not in $\mathfrak{z}(\mathfrak{g})$. Denote the multiplication of all the characters we select in the above steps by $\chi_{1}$. We need the following lemma:
    \begin{lemma}\label{G=G'T}
        Let $G'$ denote the derived subgroup of $G$. Any element in $G(F)$ could be written as a product of an element in $G'(F)$ and an element in $T(F)$. In other words, $G(F)=G'(F)\cdot T(F)$.
    \end{lemma}
    \begin{proof}
        By choosing a $z$-extension (see \Cref{z-extension} below), we can assume that $G'$ is simply connected. Let $T'=T\cap G'$ be a maximal torus of $G'$. Since $G'$ is simply connected, $X_*(T')$ has a basis of the simple positive coroots permuted by $\textnormal{Gal}(\bar{F}/F)$. Therefore, $T'$ is an induced torus. Let $D=G/G'$, the cocenter of $G$. Then we have two exact sequences, the first one included in the second one:
        $$1\longrightarrow T'(F)\longrightarrow T(F)\longrightarrow D(F)\longrightarrow 1$$
        $$1\longrightarrow G'(F)\longrightarrow G(F)\longrightarrow D(F)\longrightarrow 1.$$
        The first sequence is exact by the fact that $T'$ is an induced torus, which implies that $H^1(F,T')=0$ by Hilbert's Theorem 90. The second sequence is exact by Kneser's theorem, which implies that $H^1(F,G')=0$. This means that $g\in G(F)$ can be pushed to $D(F)$ and its image can be lifted to $T(F)$. Therefore, $G(F)=G'(F)\cdot T(F)$.

    \end{proof}
    
    By the above lemma, there is a unique extension of $\chi_1$ on $T$ to a character $\phi_1$ on $G$. Let $\rho_1=\chi\otimes \chi_1^{-1}$ be the new character in $T$ with a smaller depth. Assume that the final good element that we pick is $a_1$. Therefore, the centralizer $C_G(a_1)$ will be a proper twisted Levi subgroup of $G$.

    Keep doing this. Let $\rho_i=\chi\otimes\chi_1^{-1}\otimes\dots\otimes\chi_{i}^{-1}$. We stop at $d\in\mathbb{Z}_{\geq1}$ if $\rho_d$ is a depth zero character on the torus $T$. Therefore, we get a strictly decreasing sequence of twisted Levi subgroups. Let $B_i(\;,\;)$ denote the bilnear form defined on the Lie algebra $\mathfrak{g}^i$ of $G^i$ (see \cite[Proposition 4.1]{adler2000intertwining}). All groups in the sequence contain the maximal torus $T$. Denote the sequence by 
    $$\vec{G}=(G^1=G,G^2,\dots,G^d)$$
    and the sequence of good elements we pick $\vec{a}=(a_1,a_2,\dots,a_d)$. In addition, we pick $M^d\subset G^d$ as the maximal torus (note that $G^d$ is not necessarily the maximal torus). The absolute root system of the group $G^i$ will be defined as follows:
    $$\Phi_{E}(G^i,T)=\{\alpha\in \Phi_{E}(G^{i-1},T)\;|\;\mathrm{d}\alpha(a_{i-1})=0\}.$$
    Denote the derived subgroup of $G^i$ by $(G^i)'$, and the maximal torus of $(G^i)'$ by $(T^i)'=T\cap G^i$.

\subsubsection{The character sequence}

    We will follow everything in the above section. We have a sequence of characters $\vec{\phi}=(\phi_1,\dots,\phi_d)$ on $\vec{G}$, and a sequence of real numbers $\vec{r}=(r_1,\dots,r_d)$, where $r_i=\text{depth}(\phi_i)$. From the construction, we can see that the depth of $\phi_i$ is the same as the depth of $\chi_i$. Note that for the first character, if $^0\chi$ is not trivial on $T'_{r}$, then the good element representing $^0\chi$ will not lie in $\mathfrak{z}(\mathfrak{g})$, so we don't need to multiply it with other characters. In this case, we will simply define $\phi_1=1$ and $r_1=r$.

    To show that except $\phi_1$, all the other $\phi_i$'s cannot be trivial, we need the following lemma:

    \begin{lemma}\label{depth lemma}
        Let $l_i$ denote the depth of $\phi_i$. Then the good element representing it in the Lie algebra $\mathfrak{g}^{i+1}$ of $G^{i+1}$ must be in $\mathfrak{z}(\mathfrak{g}^{i+1})$. In addition, $l_i=r_{i+1}$ for $1\leq i\leq d-1$.
    \end{lemma}
    \begin{proof}
        For the first part, we only need to show that the character $\rho_i$ is trivial on $T'_{i+1,l_i}$, where $T'_{i+1}=T'\cap G^{i+1}$. In fact, from the construction in \Cref{construction of twisted Levi}, we can see that if we regard the good element $a_i\in g^i\simeq (g^i)^*$ as an element in $(\mathfrak{g}^{i+1})^*$ by restriction, then it not only represents the character $\rho_i|_{T_{r_i}/T_{r_i+}}$, but also represents the character $\phi_{i+1}|_{G^{i+1}_{r_{i+1}}/G^{i+1}_{r_{i+1}+}}$, by the isomorphism $G^{i+1}_{r_{i+1}}/G^{i+1}_{r_{i+1}+}\simeq \mathfrak{g}^{i+1}_{r_{i+1}}/\mathfrak{g}^{i+1}_{r_{i+1}+}$ (this is because $B_i(\mathfrak{t},\mathfrak{g}_{\alpha})=0$ for any $\alpha\in \phi_E(G^i,T)$). By writing $a_i$ as the sum of an element in $\mathfrak{t}_i'$ and an element in $\mathfrak{z}(g_i)$ and checking the definition of the bilinear form $B_i(\;,\;)$ directly, we have the following formula:
        \begin{equation}
            B_{i}(a_{i},\mathrm{d}\alpha^{\vee}(1))=B_{i}(\gamma^{\vee},\gamma^{\vee})^{-1}\sum_{\beta \in \Phi_{E}(G^{i},T)}\; \mathrm{d}\beta(a_{i})(\beta,\alpha^{\vee})
        \end{equation}
        for a short coroot $\gamma^{\vee}\in \Phi_{E}(G^{i},T)^{\vee}$. As a well-known fact,
        $$\sum_{\beta\in \Phi_E(G^i,T)}\beta(\beta,\alpha^{\vee})=k_{\alpha}\alpha$$
        for some scalar $k_{\alpha}$. Pairing it with $\alpha^{\vee}$, we get $k_{\alpha}=B_i(\alpha^{\vee},\alpha^{\vee})/2$. Taking the derivative, we have the following equation:
        \begin{equation}\label{bilinear}
            B_{i}(a_{i},\mathrm{d}\alpha^{\vee}(1))=m_{\alpha}\cdot \mathrm{d}\alpha(a_i),
        \end{equation}
        where $m_{\alpha}=B_i(\alpha^{\vee},\alpha^{\vee})/2B_i(\gamma^{\vee},\gamma^{\vee})$. Therefore, $m_{\alpha}$ is a ratio of squares of root lengths in $\Phi_E(G^{i},T)^{\vee}$. Since $p$ does not divide the order of the absolute Weyl group, we can see that $m_{\alpha}$ is a unit in $\mathfrak{O}$. Recall that the character $\rho_{i}|_{T_{l_i}}$ is denoted by the good element $a_{i}$ in $\mathfrak{t}_{-l_i}/\mathfrak{t}_{-l_i+}$. And recall that the absolute root system of $G^{i+1}$ consists of all the absolute roots $\alpha$ of $G^{i}$ such that $\mathrm{d}\alpha(a_{i})=0$. Therefore, the Lie algebra $\mathfrak{t}_{i+1}'$ of $T'_{i+1}$ is generated by the following sets:
        $$\{\alpha^{\vee}(1)\;|\;\alpha\in \Phi(G^{i+1},T)\}=\{\alpha^{\vee}(1)\;|\;\alpha\in \Phi(G^{i},T),\;\mathrm{d}\alpha(a_{i})=0\}$$

        From \Cref{bilinear} we can see that $B_{i+1}(a_{i},\alpha^{\vee}(1))=0$ for all $\alpha\in \Phi(G^{i+1},T)$. This implies that $B_{i+1}(a_{i},\mathfrak{t}'_{i+1,l_i})=0$. So $\rho_i$ is trivial on $T'_{i+1,l_i}$. Therefore, $\chi_{i+1}$ cannot be trivial. The proof is complete.

        For the second part, it is directly from the construction of $\chi_{i+1}$. Therefore, we have $\textnormal{depth}(\rho_i)=\textnormal{depth}(\chi_{i+1})=\textnormal{depth}(\phi_{i+1})$.
    \end{proof}

    From the above lemma, we can see that the sequence $\vec{r}$ is a strictly decreasing sequence from $i\geq 2$. And if $r_1=r_2$, from the definition we can get that $\phi_1=1$.

    The next thing we need to prove is that the character $\phi_i$ is $G^{i-1}$-generic of depth $r_i$. In fact, good elements and generic characters are 'dual concepts' to each other (see \cite[Aside 3.13]{fintzen2021types}). To be more precise, we have the following theorem:
    \begin{thm}
        The character $\phi_i$ is $G^{i-1}$-generic of depth $r_i$.
    \end{thm}
    \begin{proof}
        By \cite[Lemma 2.3]{fintzen2022tame}, we can simply let $y$ be the origin of the apartment $\mathcal{A}(G,T,F)$ and omit it from the presentation.
        
        First, from the construction of , we can see that the character $\phi_i|_{T_{r_i}/T_{r_i+}}=\rho_{i-1}|_{T_{r_i}/T_{r_i+}}$ could be represented by the good element $a_{i-1}\in \mathfrak{g}^{i-1}\simeq (\mathfrak{g}^{i-1})^*$. If we consider $a_{i-1}$ as an element in $(\mathfrak{g}^{i})^*$ by restriction, then it belongs to $(\mathfrak{g}^{i})^*_{-r_i}\backslash (\mathfrak{g}^{i})^*_{-r_i+}$. Also, from \Cref{depth lemma}, we can see that $a_{i-1}(\mathrm{d}\alpha^{\vee}(1))=B_{i-1}(a_{i-1},\mathrm{d}\alpha^{\vee}(1))=0$. By the definition of good elements, for any $\alpha\in \Phi_{E}(G^{i-1},T)\backslash \Phi_{E}(G^{i},T)$, we have $\text{val}(\mathrm{d}\alpha(a_{i-1}))=-r_i$. Therefore, for any $\alpha\in \Phi_{E}(G^{i-1},T)\backslash \Phi_{E}(G^{i},T)$, by \Cref{bilinear},
        $$\text{val}(a_{i-1}(\mathrm{d\alpha^{\vee}(1))=}\text{val}(B_{i-1}(a_{i-1},d\alpha^{\vee}(1))=\text{val}(m_{\alpha}\alpha(a_{i-1}))=-r_i,$$
        which means that $\phi_i$ is $G^{i-1}$-generic of depth $r_i$.
    \end{proof}

\subsubsection{The depth zero datum and the generic embeddings of buildings}

    We already define the twisted Levi sequence $\vec{G}$, $M^d=T$ and the character sequence $\vec{\phi}$ together with the real number sequence $\vec{r}$. Let $K_{M^d}={^0T}$ be the unique maximal open compact subgroup of the torus $T$, and let $\rho_{M^d}=\rho_d=(\phi_1^{-1}\otimes\dots\otimes \phi_d^{-1})|_{^0T}\otimes {^0\chi}$, which is a depth zero character on $^0T$. Let $\vec{s}=(r_2/2,r_3/2,\dots,r_{d}/2,0)$. We can see that the sequence $\vec{M}=(M^1,\dots,M^d)$ is just a constant sequence which consists of the maximal torus $T$. We pick an arbitrary point $y\in \mathcal{B}(T,F)$, and map it to a point in $ \mathcal{A}(G^d,T,F)$, which we will also denote by $y$, such that $0<\langle\alpha,y\rangle<1/2$ for any $\alpha\in \phi_{E}(G,T)^+$. This embedding is $\vec{s}$-generic, and the corresponding embedding $\iota:\mathcal{B}(T,F)\rightarrow\mathcal{B}(G^d,F)$ is 0-generic.

    Combining all the three parts above, we construct a $G$-datum, and the corresponding type is for the principal block $\mathfrak{s}=({^0T},{^0\chi})$.

\subsection{The concrete description of the type}

    In this subsection, we are going to give a concrete description of the type $(K,\rho)$ that we just construct, and finish the proof of \Cref{thm:main 4}.

    To start with, we want to give a detailed description of the sequence $\vec{r}$. We have the following theorem:
    \begin{thm}
        If $\alpha\in \Phi_{E}(G^{i-1},T)\backslash \Phi_{E}(G^{i},T)$ for an integer $i$ such that $2\leq i\leq d$, then $\textnormal{cond}(\alpha)=r_i+1$. If $\alpha\in \Phi_{E}(G^{d},T)$, then $\textnormal{cond}(\alpha)=1$.
    \end{thm}
    \begin{proof}
        Let us start with the special case, $\alpha\in \Phi_{E}(G^{d},T)$. First, from the construction of the character sequence, we know that $(\chi_1\otimes\dots\otimes \chi_{d})$ is trivial on $(T^d)'$. This gives us the fact that $^0\chi|_{^0(T^d)'}=\rho_{M^d}|_{^0(T^d)'}$. In addition, $\alpha^{\vee}(1+\mathfrak{p}_E)\subset ^0(T^d)'(E)$ and the norm map induces a surjection from $^0(T^d)'(E)$ to $^0(T^d)'(F)$ (see the proof of [\cite{haines2012base}, Lemma 8.2.1]). All these imply that $^0\chi_E\;\circ\alpha^{\vee}|_{1+\mathfrak{p}_E}=\rho_{M^d}\circ N_E\circ\alpha^{\vee}|_{1+\mathfrak{p}_E}$. However, $\rho_{M^d}$ is a depth zero character, which means that it is trivial on $T_{0+}=T(F)_1$. We conclude that $\textnormal{cond}(\alpha)=1$.

        Next, we consider the case of $\alpha\in \Phi_{E}(G^{i-1},T)\backslash \Phi_{E}(G^{i},T)$. For any $s\geq i-1$, $N_E\circ \alpha^{\vee}\subset (T^s)'$, so $\chi_s\circ N_E\circ \alpha^{\vee}$ will be trivial. Therefore, $\rho_i\circ N_E\circ\alpha^{\vee}={^0\chi_E}\circ\alpha^{\vee}$. By \Cref{depth lemma}, the depth of $\rho_{i-1}$ is $r_i$, so ${^0\chi_E}\circ\alpha^{\vee}(1+\mathfrak{p}_E^{r_i+1})=1$. The character $\rho_i|_{T_{r_i}/T_{r_i+}}$ could be represented by the good element $a_{i-1}$. Recall that $\textnormal{val}(B_{i-1}(a_{i-1},\alpha^{\vee}(1))=-r_i$. Therefore, there exists an element $u\in \mathfrak{p}^{r_i}$ such that $\psi\circ B_{i-1}(a_{i-1},\alpha^{\vee}(u))\neq 0$. By the isomorphism $\alpha^{\vee}(1+\mathfrak{p}_E^{r_i})/\alpha^{\vee}(1+\mathfrak{p}_E^{r_i+1})\simeq \alpha^{\vee}(\mathfrak{p}_E^{r_i})/\alpha^{\vee}(\mathfrak{p}_E^{r_i+1})$, and the fact that the trace map $\textnormal{T}:k_E\rightarrow k_F$ is surjective, we can find an element $v\in 1+\mathfrak{p}_E^{r_i}$ such that $\rho_i\circ N    \circ\alpha^{\vee}(v)\neq 1$. Therefore, $\textnormal{cond}(\alpha)=r_i+1$. This completes the proof.
    \end{proof}

    As in \Cref{concrete construction of types}, the group $K=K_{M^d}G^d_{y,0}G^{d-1}_{y,r_d/2}\dots G^1_{y,r2/2}=G^d_{y,0}G^{d-1}_{y,r_d/2}\dots G^1_{y,r2/2}$. By the condition \textbf{D3}, we can see that the depth of each root appearing in $K$ is exactly $f_{\mathfrak{s}}(\alpha)$. By \cite[Lemma 2.10]{yu2001construction}, we can see that $K=(G^d(E)_{y,0}G^{d-1}(E)_{y,r_d/2}\dots G^1(E)_{y,r2/2})^h=(J_{E,\mathfrak{s}})^h=J_{\mathfrak{s}}$. The proof of \Cref{thm:main 4} is complete.
    \begin{rmk}
        In fact, the condition for \cite[Lemma 2.10]{yu2001construction} is a little bit stronger than here. In that lemma, we require the last term of the the sequence to be nonzero, but in our case it is zero. This will make no difference if the field extension is unramified. For the unramified field extension $E/F$, we have the fact that $(G^d(E)_{y,0})^h=G^d(F)_{y,0}$. In other words, the Galois fixed point of an Iwahori subgroup in $G^d(E)$ will also be an Iwahori subgroup of $G^d(F)$. By using this, we can directly copy the proof of loc. cit. so we omit the details here.
    \end{rmk}
    \begin{cor}\label{support of hecke algebra}
        The support of the Hecke algebra $\mathcal{H}(G,\rho_{\mathfrak{s}})$ is $J_{\mathfrak{s}}\widetilde{W}_{^0\chi}J_{\mathfrak{s}}$, where $\widetilde{W}$ denotes the extended affine Weyl group of $G$, and $\widetilde{W}_{^0\chi}$ denotes all the elements inside of $\widetilde{W}$ which fix the character $^0\chi$.
    \end{cor}
    \begin{proof}
        This is directly from \cite[Corollary 8.2]{kim2017construction} or \cite[Theorem 4.15]{roche1998types}.
    \end{proof}

\section{Bernstein center, stable Bernstein center and base change homomorphism}\label{stable bernstein center}

    In this section, we are going to recall the concept of Bernstein center, stable Bernstein center and base change homomorphism $b_r$. Most of our presentations and notations are taken from \cite{haines2014stable} and \cite{haines2012base}.

    We follow all previous notations. Assume $G$ is any connected reductive group unramified over $F$. For any algebraic group $H$, denote its $F_r$-points by $H_r$.

\subsection{The variety structure on a principal block}

    Let $\mathfrak{s}=({^0T},{^0\chi})_G$ be a principal class of $\mathfrak{B}(G)$, and let $\chi$ be any extension of $^0\chi$ to a character $\chi:T\rightarrow \mathbb{C^{\times}}$. Let $\mathfrak{X}_{\mathfrak{s}}=\{(T,\xi)_G\}$ denote the set of all $G$-equivalent classes inside $\mathfrak{s}$. Then according to \cite[\S3.3]{haines2014stable}, we have the following bijection:
    \begin{align*}
        \hat{A}/W_{^0\chi} & \xrightarrow{\sim} \mathfrak{X}_{\mathfrak{s}}\\
        \eta & \rightarrow (T,\chi\eta)_G
    \end{align*}
    where $\hat{A}=\textnormal{Hom}(X_*(A),\mathbb{C}^{\times})$ and $\eta\in \hat{A}$ is viewed as an unramified character on $T(F)$ (cf. \cite[\S 3.2]{haines2012base}). Therefore, the set $\mathfrak{X}_{\mathfrak{s}}$ could be viewed as an affine algebraic variety. Note that the structure of the affine variety depends only on the character $^0\chi$.

    As in \cite[\S7]{roche1998types}, for the type $(J_{\mathfrak{s}},\rho_{\mathfrak{s}})$ constructed in \Cref{types}, the principal block $\mathfrak{R}_{\mathfrak{s}}(G)$ is naturally equivalent to the category of $\mathcal{H}(G,\rho_{\mathfrak{s}})$-modules by the isomorphism
    $$\mathfrak{R}_{\mathfrak{s}}(G)\rightarrow \mathcal{H}(G,\rho_{\mathfrak{s}})\textnormal{-modules}:V\rightarrow V^{\rho_{\mathfrak{s}}},$$
    where $V^{\rho_{\mathfrak{s}}}$ denotes the $\rho_{\mathfrak{s}}$-isotypic subspace of $V$. According to \cite[\S3.4]{haines2012base}, we have the isomorphism
    \begin{equation}
        \beta:\mathbb{C}[\mathfrak{X}_{\mathfrak{s}}]\xrightarrow{\sim}\mathcal{Z}(G,\rho_{\mathfrak{s}})
    \end{equation}
    in the following way: for any extension $\xi:T\rightarrow\mathbb{C}^{\times}$ of the form $\xi=\chi\eta$ with $\eta$ and unramified character of $T$, consider the subspace $i_B^G(\xi)^{\rho_{\mathfrak{s}}}$ of $i_B^G(\xi)$. Then for an element $z\in \mathcal{Z}(G,\rho_{\mathfrak{s}})$, it acts on the subspace $i_B^G(\xi)^{\rho_{\mathfrak{s}}}$ by the scalar $\beta^{-1}(z)(\xi)$. It is easy to see that the function $\beta^{-1}(z)$ is well defined because the action of $z$ on the subspace $i_B^G(\xi)^{\rho_{\mathfrak{s}}}$ is the same as the action of $z$ on $i_B^G(^w\xi)^{\rho_{\mathfrak{s}}}$ for any element $w\in W$.

    To end this subsection, let us give a concrete action of $\mathcal{H}(G,\rho)$ on $i_B^G(\xi)^{\rho_{\mathfrak{s}}}$. We denote the action by $\bullet$. Then for any $f\in\mathcal{H}(G,\rho)$, $h\in i_B^G(\xi)^{\rho_{\mathfrak{s}}}$ and $g_0\in G$,we have
    \begin{align*}
            (f\bullet h)(g_0) & = \int_G \; h(g)f^{\vee}(g^{-1}g_0) \; dg \\
            & = \int_G \; h(g)f(g_0^{-1}g) \; dg,
    \end{align*}
    where $f^{\vee}(g):=f(g^{-1})$. Note that by the above discussion, we have
    $$f\bullet h(g_0)=\beta(f)(\xi)\cdot h(g_0)$$
    for any $f\in \mathcal{Z}(G,\rho)$.
\subsection{Bernstein center}

    The Bernstein center $\mathfrak{Z}(G)$ is defined to be the ring of endomorphisms of the identity functor on the category of smooth representations $\mathfrak{R}(G)$. Let $\mathcal{D}(G)_{ec}^G$ denote the set of $G$-invariant essentially compact distributions on $C_c^{\infty}(G)$. By \cite[Equation 3.2.2]{haines2014stable}, we have isomorphisms
    \begin{equation}
        \mathfrak{Z}(G) \;\widetilde{\leftarrow} \;\lim\limits_{\longleftarrow} \mathcal{Z}(G,J)\;\widetilde{\rightarrow}\;\mathcal{D}(G)^G_{ec},
    \end{equation}
    where $J$ is an open compact subgroup of $G(F)$, $\mathcal{Z}(G,J)$ denotes the center of the ring $\mathcal{H}(G,J)$ of $J$-bi-invariant compactly supported functions, and the inverse limit is taken over all compact open subgroups $J$.

    Similar to the above subsection, for each inertial class $\mathfrak{s}=[M,\sigma]_G$, we endow the set $\mathfrak{X}_{\mathfrak{s}}=\{(L,\tau)_G\;|\;(L,\tau)\sim (M,\sigma)\}$ with a variety structure (see \cite[\S 3.3.1]{haines2014stable} for details). Define
    $$\mathfrak{X}_G=\coprod_{\mathfrak{s}}\mathfrak{X}_{\mathfrak{s}}.$$
    We shall see that $\mathfrak{X}_G$ has a natural structure of an algebraic variety, and the $\mathfrak{X}_{\mathfrak{s}}$'s form the connected components of that variety. By \cite[\S 3.3.2]{haines2014stable}, an element $Z\in \mathfrak{Z}(G)$ determines a regular function on $\mathfrak{X}_G$: for a point $(M,\sigma)_G \in \mathfrak{X}_{\mathfrak{s}}$, $Z$ acts on $i_P^G(\sigma)$ by a scalar $Z(\sigma)$ and the function $(M,\sigma)_G\rightarrow Z(\sigma)$ is a regular function on $\mathfrak{X}_G$. In fact, we have an isomorphism
    \begin{equation}
        \mathfrak{Z}(G)\;\widetilde{\rightarrow}\; \mathbb{C}[\mathfrak{X}_G].
    \end{equation}

    By direct computation and checking the definition, we have the following proposition:
    \begin{prop}\label{property of Bernstein center}
        Let $Z\in \mathfrak{Z}(G)$, and suppose $Z$ acts on $\pi=i_B^G(\chi)$ by a scalar $Z(\chi)$ for a character $\chi$ on $T(F)$. Let $(J_{\mathfrak{s}},\rho_{\mathfrak{s}})$ be the type constructed in \Cref{types} associated with the inertial equivalence class $\mathfrak{s}=[T,\chi]_G$. Denote the unit element in $\mathcal{H}(G,\rho_{\mathfrak{s}})$ by $e_{\rho_{\mathfrak{s}}}$. Then we have the following statements:
        \begin{itemize}
            \item[(a)] $Z*e_{\rho_{\mathfrak{s}}}$ is in $\mathcal{Z}(G,\rho_{\mathfrak{s}})$ and $Z*e_{\rho_{\mathfrak{s}}}$ acts on $\pi^{\rho_{\mathfrak{s}}}$ by the scalar $Z(\chi)$.
            \item[(b)] More generally, assume that $\pi$ is just a smooth finite-length representation such that $Z$ acts on it by a scalar $Z(\pi)$. Then for every $f\in \mathcal{H}(G)$, $\textnormal{tr}(Z*f\;|\;\pi)=Z(\pi)\textnormal{tr}(f\;|\;\pi).$
        \end{itemize}
    \end{prop}
    Let $\mathfrak{B}(G)_{\textnormal{ps}}$ denote the set of all principal series blocks and $\mathfrak{X}_G^{\textnormal{ps}}=\coprod_{\mathfrak{s}\in \mathfrak{B}(G)_{\textnormal{ps}}}\mathfrak{X}_{\mathfrak{s}}$. Denote the ring of regular functions on $\mathfrak{X}_G^{\textnormal{ps}}$ by $\mathfrak{Z}(G)_{\textnormal{ps}}$. For any $Z\in \mathfrak{Z}(G)_{\textnormal{ps}}$, by extending it trivially to all non-principal blocks, we can regard it as an element in $\mathfrak{Z}(G)$. Therefore, we can regard $\mathfrak{Z}(G)_{\textnormal{ps}}$ as a subset of $\mathfrak{Z}(G)$. Note that this is only a subset of $\mathfrak{Z}(G)$, not a subring. However, this subset is closed under multiplication and addition of $\mathfrak{Z}(G)$. Therefore, we can endow $\mathfrak{Z}(G)_{\textnormal{ps}}$ with a ring structure. Since $\mathfrak{Z}(G)_{\textnormal{ps}}$ is a subset of $\mathfrak{Z}(G)$, it has a convolution action on $\mathcal{H}(G)$.

    By \Cref{property of Bernstein center}(a), we have the following corollary:
    \begin{cor}\label{surjectivity of convolution}

        Any element in $\mathcal{Z}(G,\rho_{\mathfrak{s}})$ can be written in the form of $Z*e_{\rho_{\mathfrak{s}}}$ for some $Z\in \mathfrak{Z}(G)_{\textnormal{ps}}$.
    \end{cor}

    Next, we are going to recall the definition of constant term homomorphisms. Fix a standard Levi subgroup $M$, a standard parabolic $P=MN$ and an inertial equivalence class $\mathfrak{s}_M=[M',\sigma]_M\subset \mathfrak{B}(M)$. Let $\mathfrak{s}=[M',\sigma]_G$ be the corresponding inertial equivalence class in $G$. There is a canonical surjective morphism of varieties
    \begin{align}
        c_M^{G*}:\mathfrak{X}_{\mathfrak{s}_M} & \rightarrow \mathfrak{X}_{\mathfrak{s}} \notag \\
        (L,\tau)_M & \rightarrow (L,\tau)_G.
    \end{align}
   If we restrict the morphism $c_M^G$ to pairs of the form $(T,\xi)_M$, where $T$ denotes the maximal torus and $\xi$ is a complex character on $T$, then we get a morphism from $\mathfrak{X}_M^{\textnormal{ps}}$ to $\mathfrak{X}_G^{\textnormal{ps}}$, for which we use the same symbol $c_M^{G*}$ to denote. This induces an algebra homomorphism
    \begin{equation}
        c_M^G:\mathfrak{Z}(G)_{\textnormal{ps}}\rightarrow \mathfrak{Z}(M)_{\textnormal{ps}}.
    \end{equation}
    We call $c_M^G$ the constant term homomorphism, and for any $Z\in \mathfrak{Z}(G){\textnormal{ps}}$, we denote the image by $(Z)_P$.

\subsection{Stable Bernstein center}

\subsubsection{The local Langlands correspondence}

    We will recall the general form of the conjectural local Langlands correspondence (LLC) for a connected reductive $p$-adic group here. Let $F$ be a $p$-adic field and $\bar{F}$ denote the algebraic closure of $F$. Let $W_F\subset \Gamma_F:=\textnormal{Gal}(\bar{F}/F)$ denote the Weil group of $F$, and $I_F$ denote the inertia subgroup of $\Gamma_F$. We fix a geometric Frobenius element $\Phi\in W_F$. Let $W_F':=W_F\ltimes \mathbb{C}$ denote the Weil-Deligne group.

    A Langlands parameter is an admissible homomorphism $\varphi:W_F'\rightarrow {^LG}$, where $^LG=\hat{G}\rtimes W_F$ denotes the Langlands dual group. Denote the set of $\hat{G}$-conjugacy classes of admissible homomorphisms $\varphi:W_F'\rightarrow {^LG}$ by $\Phi(G/F)$ and let $\Pi(G/F)=\mathfrak{R}(G)_{\textnormal{irred}}$ denote the isomorphism classes of irreducible smooth representations of $G(F)$.

    \begin{conj}[LLC]
        There is a finite to one surjective map $\Pi(G/F)\rightarrow \Phi(G/F)$, which satisfies the desiderata of \cite[\S 10]{borel1979automorphic}.
    \end{conj}

\subsubsection{Stable Bernstein center and its relation with Bernstein center}

    In this subsection, we will assume LLC+ holds for the group $G$ (see \cite[Definition 5.2.1]{haines2014stable} for details). For any irreducible smooth representation $\pi$ of $G$, we write $\varphi_{\pi}$ the attached parameter by LLC.
    
    We term a $\widehat{G}$-conjugacy class of an admissible homomorphism
    $$\lambda:W_F\rightarrow {^LG}$$
    an infinitesimal character. Denote the $\widehat{G}$-conjugacy class of $\lambda$ by $(\lambda)_{\widehat{G}}$.

    Following \cite[\S 5.3]{haines2014stable}, we can regard the set of all infinitesimal characters as an affine variety, and we denote it by $\mathfrak{Y}_G$. Let $\mathfrak{Z}^{\textnormal{st}}(G)$ denote the ring of regular functions on the affine variety $\mathfrak{Y}$. We call this ring the stable Bernstein center of $G/F$.

    We have the following proposition from \cite[Proposition 5.5.1 and Corollary 5.5.2]{haines2014stable}:
    \begin{prop}\label{stable bernstein center and bernstein center}
        Assume LLC+ holds for $G$. The map $(M,\sigma)_G\rightarrow (\varphi_{\sigma}|_{W_F})_{\widehat{G}}$ defines a quasi-finite morphism of affine algebraic varieties
        $$f:\mathfrak{X}_G\rightarrow \mathfrak{Y}_G.$$
        It is surjective if $G/F$ is quasi-split.

        Moreover, $f$ induces a $\mathbb{C}$-algebra homomorphism $\mathfrak{Z}^{\textnormal{st}}(G)\rightarrow \mathfrak{Z}(G)$. It is injective if $G/F$ is quasi-split.
    \end{prop}

    To finish this part, we are going to introduce one way to construct distributions. Let $(\tau,V)$ be a finite-dimensional algebraic representation of $^LG$ on a complex vector space. Given a geometric Frobenius element $\Phi\in W_F$ and an admissible homomorphism $\lambda:W_F\rightarrow {^LG}$, we may define the semisimple trace
    \begin{equation*}
        \textnormal{tr}^{\textnormal{ss}}(\lambda(\Phi),V):=\textnormal{tr}(\tau\lambda(\Phi),V^{\tau\lambda(I_F)}).
    \end{equation*}
    \begin{prop}[{\cite[Proposition 5.7.1]{haines2014stable}}]\label{geometric stable Bernstein center}
        The map $\lambda\rightarrow \textnormal{tr}^{\textnormal{ss}}(\lambda(\Phi),V)$ defines a regular function on the variety $\mathfrak{Y}$ hence defines an element $Z_V\in \mathfrak{Z}^{\textnormal{st}}(G)$ by
        \begin{equation*}
            Z_V((\lambda)_{\widehat{G}})=\textnormal{tr}^{\textnormal{ss}}(\lambda(\Phi),V).
        \end{equation*}
        We use the same symbol $Z_V$ to denote the corresponding element in $\mathfrak{Z}(G)$ given via $\mathfrak{Z}^{\textnormal{st}}(G)\rightarrow \mathfrak{Z}(G)$. The latter has the property \begin{equation}
            Z_V(\pi)=\textnormal{tr}^{\textnormal{ss}}(\varphi_{\pi}(\Phi),V)
        \end{equation}
        for every $\pi\in \Pi(G/F)$, where $Z_V(\pi)$ stands for $Z_V((M,\sigma)_G)$ if $(M,\sigma)_G$ is the supercuspidal support of $\pi$.
    \end{prop}
    \begin{rmk}
        In fact, to state and prove \Cref{main theorem}, we don't need LLC+ and the stable Bernstein center. The main result is true without LLC+. We recall the definition of the stable Bernstein center and its relation with the Bernstein center to connect this result with the general conjecture proposed by Haines.
    \end{rmk}

\subsection{Base change homomorphism}

\subsubsection{Base change homomorphism of principal blocks}\label{base change homomorphism}

    Recall that in the above sections, we regard the set $\mathfrak{X}_{\mathfrak{s}}$ as an affine variety. Then, for any principal class $\mathfrak{s}=({^0T},{^0\chi})_G$, we can have a corresponding principal class $\mathfrak{s}_r=({^0T_r},{^0\chi}\circ N_r)_{G_r}$ on $G_r$, where $N_r$ denotes the norm map from $T_r$ to $T$. Therefore, we can define a morphism $N_r^*:\mathfrak{X}_{\mathfrak{s}}\rightarrow\mathfrak{X}_{\mathfrak{s}_r}$ by mapping $(T,\xi)_G$ to $(T_r,\xi\circ N_r)_{G_r}$. This induces an algebra homomorphism, which is also denoted by $N_r$:
    \begin{equation}
        N_r:\mathbb{C}[\mathfrak{X}_{\mathfrak{s}_r}]\rightarrow\mathbb{C}[\mathfrak{X}_{\mathfrak{s}}].
    \end{equation}
    \begin{defn}[{\cite[Definition 4.1.1]{haines2012base}}]
        The base change homomorphism $b_r:\mathcal{Z}(G,\rho_{\mathfrak{s}_r})\rightarrow \mathcal{Z}(G,\rho_{\mathfrak{s}})$ is the unique map which makes the following diagram commute:
         \[
            \begin{tikzcd}
            \mathbb{C}[\mathfrak{X}_{\mathfrak{s}_r}] \arrow[r,"\beta"] \arrow[d,"N_r"] & \mathcal{Z}(G_r,\rho_{\mathfrak{s}_r})   \arrow[d,"b_r"]\\
            \mathbb{C}[\mathfrak{X}_{\mathfrak{s}}] \arrow[r,"\beta"] & \mathcal{Z}(G,\rho_{\mathfrak{s}})
            \end{tikzcd}
        \]
    \end{defn}

    Note that the morphism $N_r^*$ can be extended to a morphism from $\mathfrak{X}_G^{\textnormal{ps}}$ to $\mathfrak{X}_{G_r}^{\textnormal{ps}}$. This induces an algebra homomorphism from $\mathfrak{Z}(G_r)_{\textnormal{ps}}$ to $\mathfrak{Z}(G)_{\textnormal{ps}}$, for which we also use $b_r$ to denote:
    \begin{equation}
        b_r:\mathfrak{Z}(G_r)_{\textnormal{ps}}\rightarrow \mathfrak{Z}(G)_{\textnormal{ps}}.
    \end{equation}
    By the definition of the base change homomorphism, we have
    \begin{equation}\label{matching of base change}
        b_r(Z_r*e_{\rho_{\mathfrak{s}_r}})=b_r(Z_r)*e_{\rho_{\mathfrak{s}}}
    \end{equation}
    for any $Z_r\in \mathfrak{Z}(G_r)_{\textnormal{ps}}$.

    Following \cite[\S 5.1]{haines2012base}, we have the following commutative diagram:
        
            \begin{equation}\label{constant term and base change commutes}
            \begin{tikzcd}
            \mathfrak{Z}(G_r)_{\textnormal{ps}} \arrow[r,"c_{M_r}^{G_r}"] \arrow[d,"b_r"] & \mathfrak{Z}(M_r)_{\textnormal{ps}}   \arrow[d,"b_r"]\\
            \mathfrak{Z}(G)_{\textnormal{ps}} \arrow[r,"c_M^G"] & \mathfrak{Z}(M)_{\textnormal{ps}}
            \end{tikzcd}
            \end{equation}

\subsubsection{Base change homomorphism of the stable Bernstein center}

    Again, let $F_r/F$ be a finite unramified field extension. Then $W_{F_r}\subset W_F$ and $I_{F_r}=I_F$. Let $\mathfrak{Y}^{G/F}$ (resp. $\mathfrak{Y}^{G/F_r}$) denote the variety of infinitesimal characters associated to $G$ (resp. $G_r$).
    \begin{prop}[{\cite[Proposition 5.4.1]{haines2014stable}}]
        The map $(\lambda)_{\widehat{G}}\rightarrow (\lambda|_{W_{F_r}})_{\widehat{G}}$ determines a morphism of algebraic varieties $\mathfrak{Y}^{G/F_r}\rightarrow \mathfrak{Y}^{G/F}$.
    \end{prop}
    \begin{defn}
        We call the corresponding map $b_r':\mathfrak{Z}^{\textnormal{st}}(G_r)\rightarrow \mathfrak{Z}^{\textnormal{st}}(G)$ the base change homomorphism for the stable Bernstein center.
    \end{defn}

    By \Cref{stable bernstein center and bernstein center}, if we assume LLC+ holds for $G$, then for any element $Z\in \mathfrak{Z}^{\textnormal{st}}(G)$, we may regard it as an element in the Bernstein center $\mathfrak{Z}(G)$, which we also denote by $Z$. Therefore, $\mathfrak{Z}^{\textnormal{st}}(G)$ can also act on $\mathcal{H}(G)$ be convolution. We have the following lemma:
    \begin{lemma}
        For any $Z_r\in \mathfrak{Z}^{\textnormal{st}}(G_r)$, we have
        $$b_r(Z_r*e_{\rho_{\mathfrak{s}_r}})=b_r'(Z_r)*e_{\rho_{\mathfrak{s}}}.$$
    \end{lemma}
    \begin{proof}
        By \cite[Lemma 8.1.3]{kazhdan2012endoscopic}, we have the following commutative diagram:
        \[
            \begin{tikzcd}
            \textnormal{Hom}_{\textnormal{conts}}(T(F),\mathbb{C}^{\times}) \arrow[r] \arrow[d,"N_r"] & H^1_{\textnormal{conts}}(W_F,{^LT})   \arrow[d,"\textnormal{Res}"]\\
            \textnormal{Hom}_{\textnormal{conts}}(T(F_r),\mathbb{C}^{\times}) \arrow[r] & H^1_{\textnormal{conts}}(W_{F_r},{^LT_r})
            \end{tikzcd}
        \]
        This immediately gives us the lemma 
    \end{proof}

\section{Descent formulas}\label{descent formula section}

     In this section, we are going to develop another form of descent formula, to reduce the case to
     elliptic semisimple elements.

     Fix a principal class $\mathfrak{s}=({^0T},{^0\chi})$. In all future sections, we will omit the subscript $\mathfrak{s}$. In addition, we will use the symbols $(J,\rho)$ and $(J_r,\rho_r)$ instead of $(J_{\mathfrak{s}},\rho_{\mathfrak{s}})$ and $(J_{\mathfrak{s}_r},\rho_{\mathfrak{s}_r})$. Also, by \Cref{surjectivity of convolution} and \Cref{matching of base change}, we will only study elements of the form $Z_r*e_{\rho_r}$ for some $Z_r\in \mathfrak{Z}(G_r)_{\textnormal{ps}}$.

\subsection{Preparations for descent formulas}

    Fix once and for all Haar measures $\mathrm{d}g^i,\mathrm{d}g^j,\mathrm{d}g_r^i, \mathrm{d}g_r^j$ on $G,G_r$ respectively, such that $\textnormal{vol}_{\mathrm{d}g^i}(I)=\textnormal{vol}_{\mathrm{d}g^j}(J)=\textnormal{vol}_{\mathrm{d}g_r^i}(I_r)=\textnormal{vol}_{\mathrm{d}g_r^j}(J_r)=1$. 

    Assume that $\delta\in G_r$ is a $\theta$-semisimple element such that $\gamma=\mathcal{N}\delta$ is non-elliptic. Denote the $F$-split component of the center of the centralizer $G_{\gamma}^{\circ}$ by $S$. Let $M=\textnormal{Cent}_G(S)$, an $F$-Levi subgroup. Since $\gamma$ is non-elliptic, $M$ is a proper Levi subgroup of $G$, and $\gamma$ is elliptic in $M$. We can further assume that $M$ is standard and choose a standard $F$-parabolic subgroup $P=MN$ with $M$ as the Levi factor (otherwise we can replace $\gamma$ by its conjugation). By \cite[Lemma 4.2.1]{haines2009base}, $\gamma$ will be the norm of some elements in $M_r$. Therefore, we can assume $\delta\in M_r$. In conclusion, we will assume that $\gamma\in M$ is a non-elliptic semisimple element, such that $\gamma$ is elliptic in $M$ and $\gamma=\mathcal{N}\delta$ for some $\delta\in M_r$.

    The twisted centralizer $G_{\delta\theta}$ of $\delta\theta$ is an inner form of $G_{\gamma}$ whose group of $F$-points is 
    $$G_{\delta\theta}(F)=\{ g\in G_r\;|\; g^{-1}\delta\theta(g)=\delta\}.$$
    We have $M_{\delta\theta}^{\circ}=G_{\delta\theta}^{\circ}$ for dimension reasons. We choose compatible measures on the inner forms $G_{\delta\theta}^{\circ}$ and $G_{\gamma}^{\circ}$, and use them to form the quotient measures $d\bar{g}$ in the (twisted) orbital integrals of $\phi\in C_c^{\infty}(G_r)$. By definition, we have
    \begin{equation}
        \textnormal{TO}_{\delta\theta}^{G_r}(\phi)=\int_{G_{\delta\theta}^{\circ}\backslash G_r}\; \phi(g^{-1}\delta\theta(g))\;d\bar{g}.
    \end{equation}

    We also need to define another type of the same principal block:
    \begin{lemma}
        Let $\rho_r^I:=\textnormal{Ind}_{J_r}^{I_r}\rho_r$. Then $(I_r, \rho_r^I)$ is also a type for $\mathfrak{s}_r$.
    \end{lemma}
    \begin{proof}
        From Frobenius reciprocity, we know that
        $$\mathrm{Hom}_{J_r}(\rho_r,\sigma)=\mathrm{Hom}_{I_r}(\rho_r^I,\sigma)$$
        for any representations $\sigma$ of $G_r$. So we only need to prove that $\rho_r^I$ is irreducible. By \Cref{support of hecke algebra}, we can see that the intertwiners of $(J_r,\rho_r)$ inside of $I_r$ is just $J_r$ itself. Then by Mackey's theory, $\rho_r^I$ is an irreducible representation of $I_r$.
    \end{proof}

    Following \cite{haineshecke}, we define $e_{\rho_r^I}\in \mathcal{H}(G_r)$ by
    $$e_{\rho_r^I}(x)=
    \begin{cases}
        \mathrm{dim}(\rho_r^I)\;\mathrm{tr}(\rho_r^I(x^{-1})), & x\in I_r \\
        0, & x\notin I_r.
    \end{cases}$$
    For any representation $(\pi,V)$ of $G_r$, let $V^{\rho_r^I}$ denote the $\rho_r^I$-isotypical component of $V$. Then we have $V^{\rho_r^I}=e_{\rho_r^I}V$. 

    Next, we compare the twisted orbital integrals of $e_{\rho_r^I}$ with the twisted orbital integrals of $e_{\rho_r}$ (for the untwisted version, just take $r=1$). We have the following proposition:
    \begin{prop}\label{orbital integral of e}
        For any $\theta$-semisimple element $\delta\in G(F_r)$, we have
        $$\textnormal{TO}_{\delta\theta}^{G_r}(e_{\rho_r^I})=[I:J]\;\textnormal{TO}_{\delta\theta}^{G_r}(e_{\rho_r}),$$
        where in the first integral, we use the Haar measure $\mathrm{d}g_r^i$, and in the second integral, we use the Haar measure $\mathrm{d}g_r^j$.
    \end{prop}
    \begin{proof}
        Fix a $\theta$-stable representation $(\pi,V)$ of $G_r$ with intertwiner $I_{\theta}$. By the twisted version of the Kazhdan density theorem (see \cite{kottwitz2000distributions}), we only need to prove that $\textnormal{tr}(e_{\rho_r^I}I_{\theta}|V)=[I:J]\;\textnormal{tr}(e_{\rho_r}I_{\theta}|V)$, where on the left hand side, we use the Haar measure $\mathrm{d}g_r^i$ and on the right hand side, we use the Haar measure $\mathrm{d}g_r^j$. With the Haar measures given, the two functions induce projections from $V$ to $V^{\rho_r^I}$ and $V^{\rho_r}$, respectively. Together with the fact that $\rho_r^I$ is also $\theta$-stable, we can get
        $$\textnormal{tr}(e_{\rho_r^I}I_{\theta}|V)=\textnormal{tr}(I_{\theta}|V^{\rho_r^I}),$$
        $$\textnormal{tr}(e_{\rho_r}I_{\theta}|V)=\textnormal{tr}(I_{\theta}|V^{\rho_r}).$$
        So we only need to prove $\textnormal{tr}(I_{\theta}|V^{\rho_r^I})=[I:J]\;\textnormal{tr}(I_{\theta}|V^{\rho_r})$. We choose a basis for $\mathrm{Hom}_{J_r}(\rho_r,V)$, denoted by $\{\phi_1,\cdots,\phi_n\}$. By Frobenius reciprocity, if we denote the corresponding element of $\phi_i$ in $\mathrm{Hom}_{I_r}(\rho_r^I,V)$ by $\phi_i'$, then the set $\{\phi_1',\cdots,\phi_n'\}$ will be a basis for $\mathrm{Hom}_{J_r}(\rho_r^I,V)$.
        
        Here, we recall the explicit form of Frobenius reciprocity: fix a set of representatives of the left coset for $J_r\backslash I_r$, denoted by $\{i_1,\cdots,i_l\}$. We define the functions $f_k$ for $1\leq k\leq l$ as follows:
        $$f_k(x)=
            \begin{cases}
            \rho_r(j), & \text{if } x \in J_r i_k \text{ and } x = j i_k \text{ for } j \in J_r, \\
            0, & \text{otherwise}.
            \end{cases}$$
        Then the explicit form of Frobenius reciprocity is given as follows:
        $$\phi'_h(f_k)=\sum_{x\in I_r/J_r}\pi(x)\phi_h(f_k(x^{-1})).$$
        The sets $\{\phi_h(1)\}_{1\leq h\leq n}$ and $\{\phi'_h(f_k)\}_{1\leq h\leq n,1\leq k\leq l}$ form bases for $V^{\rho_r}$ and $V^{\rho_r^I}$, respectively. Assume that $I_{\theta}(\phi_h(1))=\sum_{1\leq q\leq n}a_{hq}\phi_q(1)$ for some complex numbers $a_{hq}, 1\leq h,q\leq n$. Then we have
        \begin{align*}
        I_{\theta}(\phi'_h(f_k)) & = I_{\theta}(\sum_{x\in I_r/J_r}\pi(x)\phi_h(f_k(x^{-1})))\\ & =\sum_{x\in I_r/J_r}\pi(\theta(x))I_{\theta}(\phi_h(f_k(x^{-1}))) \\ & =\sum_{x\in I_r/J_r}\pi(\theta(x))I_{\theta}(\phi_h({^\theta f}_k((\theta(x))^{-1}))) \\ & =\sum_{x\in I_r/J_r}\pi(\theta(x))\sum_{1\leq q\leq n}a_{hq}\cdot \phi_h({^\theta f}_k((\theta(x))^{-1})) \\ & =\sum_{1\leq q \leq n}a_{hq}\cdot\phi'_q({^\theta f}_k),
        \end{align*}
        where $^\theta f_k(x)=f_k(\theta^{-1}(x))$. We can see that only those $k$ such that $\theta(J_ri_k)=J_ri_k$ could contribute to the trace. By \cite[Proposition 2.2]{yu2001construction}, we know $H^1(\textnormal{Gal}(F_r/F), J_r)=1$. Therefore, $(J_r\backslash I_r)^{\theta}=J\backslash I$, and we may pick $i_k$ such that $\theta(i_k)=i_k$ for all such $k$. Thus, for all such $k$, $I_{\theta}(\phi'_h(f_k))=\sum_{1\leq q \leq n}a_{hq}\cdot\phi'_q({^\theta f}_k)=\sum_{1\leq q \leq n}a_{hq}\cdot\phi'_q(f_k)$. Therefore, $\textnormal{tr}(I_{\theta}|V^{\rho_r^I})=[I:J]\;\sum_{1\leq h\leq n} a_{hh}=[I:J]\;\textnormal{tr}(I_{\theta}|V^{\rho_r})$. The proof is complete.
    \end{proof}

\subsection{Descent formulas}
    Choose $K_r$ to be a maximal compact subgroup of $G_r$ in good position with respect to $P$ and stable by $\theta$. For any $\phi \in C_c^{\infty}(G_r)$, we define the constant term $\phi_P$ as follows:
    $$\phi_{P_r}(m)=\delta_{P_r}^{1/2}(m)\int_{N_r}\int_{K_r}\phi(k^{-1}mn\theta(k))\;\mathrm{d}k\;\mathrm{d}n$$
    where the Haar measure $\mathrm{d}k$ is chosen such that $\textnormal{vol}_{\mathrm{d}k}(K_r)=1$ and the Haar measure $\mathrm{d}n$ is chosen such that $\textnormal{vol}_{\mathrm{d}n}(K_r\cap N_r)=1$.
    Note that $\phi_{P_r}\in C_c^{\infty}(G_r)$. We have the following fundamental descent formula, from \cite[Proposition 4.4.9]{laumon1996cohomology}:
    \begin{prop}
        Suppose $M$ is an $F$-Levi subgroup of $G$ and $P=MN$ is an $F$-rational parabolic subgroup with Levi factor $M$. Suppose that we are given a $\theta$-semisimple $\delta\in G_r$ with norm $\gamma=\mathcal{N}\delta\in G$, with the property that $G_{\delta\theta}^{\circ}\subset M$. Then
        \begin{equation}
            \textnormal{TO}_{\delta\theta}^{G_r}(\phi)=|D_{G/M}(\gamma)|^{-1/2}\textnormal{TO}_{\delta\theta}^{M_r}(\phi_{P_r})
        \end{equation}
        where $D_{G/M}(m)=\textnormal{det}(1-\mathrm{Ad}(m^{-1});\mathrm{Lie}(G(F))/\mathrm{Lie}(M(F))).$
    \end{prop}

    Next, we will prove several important properties of the function $\phi_{P_r}$. Let $(\sigma,V_r)$ be a smooth $\theta$-stable finite-length representation of $M_r$. We fix an intertwiner $I_{\theta}^M:\sigma\rightarrow \sigma^{\theta}$ (note that $I_{\theta}^M$ is uniquely determined up to a non-zero scalar). We define the corresponding intertwiner $I_{\theta}:i_{P_r}^{G_r}(\sigma)\rightarrow(i_{P_r}^{G_r}(\sigma))^{\theta}$ as follows:
    $$(I_{\theta}\Phi)(g)=I_{\theta}^M(\Phi(\theta^{-1}g)),$$
    for any $\Phi\in i_{P_r}^{G_r}(\sigma)$ and $g\in G_r$.
    We have the following lemma:
    \begin{lemma}\label{trace and induced representation}
        Let $\sigma$ denote a smooth $\theta$-stable finite-length representation of $M_r$, with intertwiner $I_{\theta}^M$. Then for any $\phi\in C_c^{\infty}(G_r)$, we have
        $$\textnormal{tr}(\phi_{P_r}I_{\theta}^M|\;\sigma)=\textnormal{tr}(\phi I_{\theta}\;|\;i_{P_r}^{G_r}(\sigma)).$$
    \end{lemma}
    \begin{proof}
        See \cite[Theorem 2]{van1972computation} for the untwisted version. The twisted version is proved in the same way.
    \end{proof}

    Starting from here, we will write out the descent formula for the function $e_{\rho_r}$. We start with the functions $e_{\rho^I_r}$ and $e_{\rho_r}$. Instead of computing orbital integrals, we compute the traces. Let $(\sigma, W)$ be any irreducible representation of $M_r$. Then from \Cref{trace and induced representation}, we get
    \begin{align}\label{trace descent equation}
        \textnormal{tr}(e_{\rho_r^I} I_{\theta}\;|\;i_{P_r}^{G_r}(\sigma)) & =\textnormal{tr}((e_{\rho_r^I})_P I_{\theta}^M\;|\;\sigma) \notag \\  & =\textnormal{tr}(I_{\theta}\;|\;i_{P_r}^{G_r}(\sigma)^{\rho_r^I}).
    \end{align}

    Let $^PW_r$ denote the minimal length representatives of $W_{M_r}\backslash W_r$. We have the refined Iwasawa decomposition:
    \begin{equation}\label{refined Iwasawa decomposition}
        G_r=\coprod_{w\in {^PW_r}}P_rwI_r.
    \end{equation}
    By Mackey Theory, we have
    $$i_{P_r}^{G_r}(\sigma)|_{I_r}=\bigoplus_{w\in {^PW_r}}i_{I_r\cap {^{w^{-1}}P_r}}^{I_r}({^{w^{-1}}\sigma}).$$
    By writing out this isomorphism explicitly, we can see that $I_{\theta}$ will map a function in $i_{I_r\cap {^{w^{-1}}P_r}}^{I_r}({^{w^{-1}}\sigma})$ to a function in $i_{I_r\cap {^{\theta(w)^{-1}}P_r}}^{I_r}({^{\theta(w)^{-1}}\sigma})$. So the only $w$ which contributes to the trace comes from $({^PW_r})^{\theta}$. By [\cite{haines2009base}, Lemma 4.5.2(c)], $({^PW_r})^{\theta}={^PW}$. Therefore, we may rewrite \Cref{trace descent equation} as
    \begin{equation}
        \textnormal{tr}(I_{\theta}\;|\;i_{P_r}^{G_r}(\sigma)^{\rho_r^I})=\bigoplus_{w\in {^PW}}\textnormal{tr}(I_{\theta}\;|\;i_{I_r\cap {^{w^{-1}}P_r}}^{I_r}({^{w^{-1}}\sigma})^{\rho_r^I}).
    \end{equation}
    By the proof of \Cref{orbital integral of e}, we can simplify the above trace as follows:
    \begin{equation}\label{equation in trace}
        \bigoplus_{w\in {^PW}}\textnormal{tr}(I_{\theta}\;|\;i_{I_r\cap {^{w^{-1}}P_r}}^{I_r}({^{w^{-1}}\sigma})^{\rho_r^I})=[I:J]\bigoplus_{w\in {^PW}}\textnormal{tr}(I_{\theta}\;|\;i_{I_r\cap {^{w^{-1}}P_r}}^{I_r}({^{w^{-1}}\sigma})^{\rho_r}).
    \end{equation}
    Again by Mackey Theory, 
    \begin{align*}
        i_{I_r\cap {^{w^{-1}}P_r}}^{I_r}({^{w^{-1}}\sigma})|_{J_r} & =\bigoplus_{u\in I_r\cap {^{w^{-1}}P_r}\backslash I_r/J_r}i_{J_r\cap {^{u^{-1}}(I_r\cap {^{w^{-1}}P_r}})}^{J_r} ({^{u^{-1}w^{-1}}\sigma}) \\ & =\bigoplus_{u\in I_r\cap {^{w^{-1}}P_r}\backslash I_r/J_r}i_{J_r\cap {^{u^{-1}w^{-1}}P_r}}^{J_r} ({^{u^{-1}w^{-1}}\sigma}).
    \end{align*}

    Note that $I_r\cap {^{w^{-1}}P_r}\backslash I_r=I_r\cap {^{w^{-1}}\bar{N}_r}$ and $I_r\cap {^{w^{-1}}P_r}\backslash I_r/J_r=(I_r\cap {^{w^{-1}}\bar{N}_r})/(J_r\cap {^{w^{-1}}\bar{N}_r})$. By Frobenius reciprocity, we have 
    $$\mathrm{Hom}_{J_r}(\rho_r,i_{J_r\cap {^{u^{-1}w^{-1}}P_r}}^{J_r} ({^{u^{-1}w^{-1}}\sigma}))=\mathrm{Hom}_{J_r\cap {^{u^{-1}w^{-1}}P_r}}(\rho_r|_{J_r\cap {^{u^{-1}w^{-1}}P_r}},{^{u^{-1}w^{-1}}\sigma}|_{J_r\cap {^{u^{-1}w^{-1}}P_r}}).$$
    We pause here to note that $\sigma$ is inflated to a representation of $P_r$ on which $N_r$ acts trivially. From this it follows that the only $u$ which contributes to the above trace must satisfy the following property:
    \begin{equation}\label{condition for u}
        \rho_r|_{J_r\cap {^{u^{-1}w^{-1}}N_r}}=1.
    \end{equation}

    The following is an important lemma for a principal class $\mathfrak{s}$:
    \begin{lemma}\label{axiom for u}
        For any $w\in {^PW}$, the only $u\in (I_r\cap {^{w^{-1}}\bar{N}_r})/(J_r\cap {^{w^{-1}}\bar{N}_r})$ satisfying \Cref{condition for u} is the identity element.
    \end{lemma}
    We will postpone the proof to the next subsection.

    Now, with the above lemma, we may rewrite \Cref{equation in trace} as follows:
    \begin{equation}
        [I:J]\bigoplus_{w\in {^PW}}\textnormal{tr}(I_{\theta}\;|\;i_{I_r\cap {^{w^{-1}}P_r}}^{I_r}({^{w^{-1}}\sigma})^{\rho_r})=[I:J]\bigoplus_{w\in {^PW}}\textnormal{tr}(I_{\theta}\;|\;i_{J_r\cap {^{w^{-1}}P_r}}^{J_r}({^{w^{-1}}\sigma})^{\rho_r}).
    \end{equation}
    By writing out the Frobenius reciprocity explicitly, we have
    $$\textnormal{tr}(I_{\theta}\;|\;i_{J_r\cap {^{w^{-1}}P_r}}^{J_r}({^{w^{-1}}\sigma})^{\rho_r})=\textnormal{tr}(I_{\theta}^M\;|\;{^{w^{-1}}\sigma}^{\rho_r|_{J_r\cap {^{w^{-1}}P_r}}})=\textnormal{tr}(I_{\theta}^M\;|\;{^{w^{-1}}\sigma}^{\rho_r|_{J_r\cap {^{w^{-1}}M_r}}}).$$
    Conjugating this by $w$, and finally we get the following equation
    \begin{align*}
        \textnormal{tr}(e_{\rho_r}I_{\theta}\;|\;i_{P_r}^{G_r}(\sigma) & = \frac{1}{[I:J]}\textnormal{tr}(e_{\rho_r^I} I_{\theta}\;|\;i_{P_r}^{G_r}(\sigma)) \\ & = \bigoplus_{w\in {^PW}}\textnormal{tr}(I_{\theta}\;|\;i_{J_r\cap {^{w^{-1}}P_r}}^{J_r}({^{w^{-1}}\sigma})^{\rho_r}) \\ & =\bigoplus_{w\in {^PW}}\textnormal{tr}(I_{\theta}^M\;|\;\sigma^{^w\rho_r|_{^wJ_r\cap M_r}}).
    \end{align*}
    Since $w\in ({^PW_r})^{\theta}$, which is the minimal length representatives of $W_{M_r}\backslash W_r$, we have that for any positive absolute root $\alpha$ of $M_r$, $w^{-1}(\alpha)$ is still positive. Therefore, $({^wJ_r\cap M_r}, {^w\rho}_r|_{^wJ_r\cap M_r})$ will be a type for the inertial equivalence class $^w\mathfrak{s}_{M,r}=[{^0T}_r,{^w(}{^0\chi}_r)]_{M_r}$. Let $({^wJ_{M_r}},{^w\rho}_{M_r})=({^wJ_r\cap M_r},{^w\rho}_r|_{^wJ_r\cap M_r})$, and $e_{^w\rho_{M_r}}$ be the unit element in the corresponding Hecke algebra. Then we have
    $$\textnormal{tr}(I_{\theta}^M\;|\;\sigma^{^w\rho_{M_r}})=\textnormal{tr}(e_{^w\rho_{M_r}}I_{\theta}^M\;|\; \sigma).$$
    Combining everything above, we have the following proposition:
    \begin{prop}
        For any principal block satisfying \Cref{axiom for u}, and for any irreducible representation $(\sigma,W)$ of $M_r$, we have the following equation:
        \begin{equation}
            \textnormal{tr}(e_{\rho_r}I_{\theta}\;|\;i_{P_r}^{G_r}(\sigma))=\textnormal{tr}((e_{\rho_r})_{P_r}I_{\theta}^M\;|\;\sigma)=\bigoplus_{w\in {^PW}}\textnormal{tr}(e_{^w\rho_{M_r}}I_{\theta}^M\;|\;\sigma).
        \end{equation}
        By the twisted version of the Kazhdan density theorem (see \cite{kottwitz2000distributions}), we have 
        \begin{equation}\label{descent formula for unit element}
            \textnormal{TO}_{\delta \theta}^{G_r}(e_{\rho_r})=|D_{G/M}(\gamma)|^{-1/2}\sum_{w\in {^PW}}\textnormal{TO}_{\delta\theta}^{M_r}({e}_{^w\rho_{M_r}}).
        \end{equation}
    \end{prop}

    \begin{rmk}
        Note that this proposition helps us bypass the technical result \cite[Lemma 4.5.4]{haines2009base}. 
    \end{rmk}

    Next, we consider the descent formula for functions of the form $Z_r*e_{\rho_r}$ for $Z_r\in \mathfrak{Z}_{\textnormal{ps}}(G_r)$. For any $\phi,\phi'\in C_c^{\infty}(G_r)$, we write $\phi\equiv \phi'$ if the twisted orbital integrals of $\phi-\phi'$ at all $\theta$-semisimple elements vanish. We have the following lemma:
    \begin{lemma}
        For any $Z_r\in \mathfrak{Z}(G_r)_{\textnormal{ps}}$ and any $\phi\in C_c^{\infty}(G_r)$, 
        $$(Z_r*\phi)_{P_r}\equiv (Z_r)_{P_r}*\phi_{P_r}.$$
    \end{lemma}
    \begin{proof}
        By the twisted version of the Kazhdan density theorem (see \cite{kottwitz2000distributions}), we only need to check the following equation
        $$\textnormal{tr}((Z_r*\phi)_{P_r}I_{\theta}^M\;|\;\sigma)=\textnormal{tr}(((Z_r)_{P_r}*\phi_{P_r})I_{\theta}^M\;|\;\sigma)$$
        for any smooth irreducible $\theta$-stable representation $\sigma$. By \Cref{trace and induced representation} and \Cref{property of Bernstein center}, we have:
        \begin{align*}
            \textnormal{tr}((Z_r*\phi)_{P_r}I_{\theta}^M\;|\;\sigma) & =\textnormal{tr}((Z_r*\phi)I_{\theta}\;|\;i_{P_r}^{G_r}(\sigma))\\
            & =Z_r(i_{P_r}^{G_r}(\sigma))\textnormal{tr}(\phi I_{\theta}\;|\;i_{P_r}^{G_r}(\sigma)) \\
            & = (Z_r)_{P_r}(\sigma)\textnormal{tr}(\phi_{P_r}I_{\theta}^M\;|\;\sigma) \\
            & = \textnormal{tr}(((Z_r)_{P_r}*\phi_{P_r})I_{\theta}^M\;|\;\sigma).
        \end{align*}
    \end{proof}
    
    Combining the above lemma with \Cref{descent formula for unit element}, we get the following equation:
    \begin{equation}\label{twisted descent formula for general function}
        \textnormal{TO}_{\delta \theta}^{G_r}(Z_r*e_{\rho_r})=|D_{G/M}(\gamma)|^{-1/2}\sum_{w\in {^PW}}\textnormal{TO}_{\delta\theta}^{M_r}((Z_r)_{P_r}*{e}_{^w\rho_{M_r}}).
    \end{equation}
    By taking $r=1$ and \Cref{constant term and base change commutes}, we have the following equation:
    \begin{equation}\label{descent formula for general function}
        \textnormal{O}_{\gamma}^{G}(b_r(Z_r)*e_{\rho})=|D_{G/M}(\gamma)|^{-1/2}\sum_{w\in {^PW}}\textnormal{O}_{\gamma}^{M}(b_r((Z_r)_{P_r})*{e}_{^w\rho_{M}}).
    \end{equation}

    Comparing the above two equations, we can see that to prove $Z_r*e_{\rho_r}$ and $b_r(Z_r)*e_{\rho}$ are associated, one only needs to prove the base change fundamental lemma for $e_{^w\rho_{M_r}}$ and $e_{^w\rho_M}$. Thus, we may reduce the case to elliptic elements.

\subsection{Proof of \Cref{axiom for u}}

    In this subsection, we will prove \Cref{axiom for u} for all unramified groups and all principal blocks. We will use definitions and notations from previous sections freely. 

\subsubsection{Proof for split groups}

    We start with split groups. Assume $G$ is split. First, consider the case where $P=B=TU$. Under this assumption, $w$ could be an arbitrary element in the Weyl group $W$. For any representative $u$ of $(I\cap {^{w^{-1}}\bar{U}})/(J\cap {^{w^{-1}}\bar{U}})$, we can write $u$ in the form of 
    $$u=\prod_{i=1}^nu_{\alpha_i}(a_i),$$
    with each root $\alpha_i\in {^{w^{-1}}\Phi^-}$, and the roots are arranged in order of non-increasing height relative to the simple roots determined by $B$ when $i$ increases. We may also assume that each $u_{\alpha_i}(a_i)$ is in $I$, which means that $a_i\in \mathfrak{O}$ when $\alpha_i$ is positive and $a_i\in \mathfrak{p}$ when $\alpha_i$ is negative. Let $b_{\alpha}=\textnormal{cond}(\alpha)$ and $c_i=\textnormal{val}(a_i)$. Note that $f(\alpha)+f(-\alpha)=b_{\alpha}$. Since we assume $u\notin J\cap {^{w^{-1}}\bar{U}}$, we can find $i$ such that $c_i<f(\alpha_i)$. Therefore, $\max_i\{b_{\alpha_i}-1-c_i\}>0$. Let $c_u=\max_i\{b_{\alpha_i}-1-c_i\}$. We choose a root $\alpha_k\in \{\alpha_1,\alpha_2,\cdots, \alpha_n\}$ such that $b_{\alpha_k}-c_k-1=c_u$, and the sum of coefficients is smallest among all possible choices (if there is more than one choice, just pick the one on the rightmost side). For any positive integer $c$, define the following two sets:
    $$\mathcal{U}_c\coloneqq\{u_{\alpha}(a_{\alpha})|\alpha\in\Phi, \textnormal{val}(a_{\alpha})\geq \max\{c,f(\alpha)\}\};$$
    $$\mathcal{T}_c\coloneqq\{\alpha^{\vee}(1+b)|\alpha\in \Phi,\textnormal{val}(b)\geq\max\{\textnormal{cond}(\alpha),c\}\}.$$
    Denote all elements in $\mathcal{U}_{c}$ coming from $B$-positive roots (resp. $B$-negative roots) by $\mathcal{U}_c^+$ (resp. $\mathcal{U}_c^-$). It is easy from the definition that both $\mathcal{U}_c$ and $\mathcal{T}_c$ are in $\textnormal{ker}(\rho)$, whenever $c\geq c_u$. We have the following lemma which will be used throughout the whole subsection:
    \begin{lemma}\label{commutator lemma}
        We have the following statements:
        \begin{itemize}
            \item[(a)]\label{a} For any $u_{\alpha}(a_{\alpha})\in \mathcal{U}_c$ and any element $u_{\alpha_p}(a_p)$ such that $c\geq c_u$ and $\alpha\neq -\alpha_p$, the commutator $[u_{\alpha}(a_{\alpha})^{-1},u_{\alpha_p}(a_{p})^{-1}]$ can be written as a product of elements in $\mathcal{U}_c$. Moreover, if $\alpha$ and $\alpha_p$ are of the same sign, the terms appearing in the product should also be of the same sign. \\
            \item[(b)]\label{b} For any $u_{\alpha}(a_{\alpha})\in \mathcal{U}_{c_u}$ and any element $u_{\alpha_p}(a_p)$ such that $b_{\alpha_p}-c_p-1<c_u$ and $\alpha=-\alpha_p$, the commutator $[u_{\alpha}(a_{\alpha})^{-1},u_{\alpha_p}(a_{p})^{-1}]$ can be written as a product of elements in $\mathcal{U}_{c_u}$ and $\mathcal{T}_{c_u}$. Moreover, if we have $u_{\alpha}(a_{\alpha})\in \mathcal{U}_{c}$ for $c>c_u$, we can remove the restriction $b_{\alpha_p}-c_p-1<c_u$ and the commutator $[u_{\alpha_i}(a_i)^{-1},u_{\alpha}(a_{\alpha})]$ can be written as a product of elements in $\mathcal{U}_{c}$ and $\mathcal{T}_{c}$. \\
            \item[(c)]\label{c} For any $t=\alpha^{\vee}(1+b)\in \mathcal{T}_c$ with $c\geq c_u$ and any element $u_{\alpha_p}(a_p)$, the commutator $[t^{-1},u_{\alpha_i}(a_i)^{-1}]$ belongs to $\mathcal{U}_c$. 
        \end{itemize}
    \end{lemma}
        In the proof, we will use the following three commutator relations (the first one is from [\cite{roche1998types}, $\S 2$], and the other two can be checked easily):
        \begin{itemize}
            \item[(1)] For any roots $\alpha,\beta\in \Phi$ and $a_{\alpha},a_{\beta}\in F$ such that $\alpha\neq -\beta$, we have
            \begin{equation}\label{1}
            [u_{\alpha}(a_{\alpha}),u_{\beta}(a_{\beta})]=\prod_{\substack{i,j>0 \\ i\alpha+j\beta\in \Phi}}u_{i\alpha+j\beta}(C_{\alpha,\beta;i,j}a_{\alpha}^ia_{\beta}^j)
        \end{equation}
        for some integers $C_{\alpha,\beta;i,j}$.
        \item[(2)] For any root $\alpha\in \Phi$ and any $a_{\alpha},a_{-\alpha}\in F$, we have
        \begin{equation}\label{2}
            [u_{\alpha}(a_{\alpha}),u_{-\alpha}(a_{-\alpha})]=u_{\alpha}(\frac{-a_{\alpha}^2 a_{-\alpha}}{1-a_{\alpha}a_{-\alpha}})(-\alpha^{\vee})((1-a_{\alpha}a_{-\alpha}))u_{-\alpha}(\frac{-a_{\alpha}a_{-\alpha}^2}{1-a_{\alpha}a_{-\alpha}}).
        \end{equation}
        \item[(3)] For any element $t=\alpha^{\vee}(1+b)$, where $\alpha\in \Phi$ and $b\in \mathfrak{p}$, and any root $\beta\in \Phi$ and $a_{\beta}\in F$, we have
        \begin{equation}\label{3}
            [t,u_{\beta}(a_{\beta})]=u_{\beta}(((1+b)^{(\alpha^{\vee},\beta)}-1)a_{\beta}).
        \end{equation}
        \end{itemize}
        Now, let us begin the proof:
        \begin{proof}
            For part (a), by \Cref{1}, we can write $[u_{\alpha_p}(a_p)^{-1},u_{\alpha}(a_{\alpha})^{-1}]$ as a product of elements of the form $u_{i\alpha+j\alpha_p}(C_{i,j;\alpha,\alpha_p}a_{\alpha}^ia_p^j)$ for positive integers $i,j$ such that $i\alpha+j\alpha_p\in \Phi$. Fix $i,j$ and denote $i\alpha+j\alpha_p$ by $\alpha'$ and $C_{i,j;\alpha,\alpha_p}a_{\alpha}^ia_p^j$ by $a_{\alpha'}$. We have the following computation on valuations:
            $$\textnormal{val}(a_{\alpha'})\geq \textnormal{val}(a_\alpha)+\textnormal{val}(a_p)\geq \max\{b_{\alpha_p}-1,c,f(\alpha)\}.$$
            Note that $b_{\alpha_p}-1\leq f(\alpha_p)$ only if $\alpha_p$ is negative and $b_{\alpha_p}=1$, which shows that $f(\alpha_p)=1$. But if $\alpha_p$ is negative, we require that $\textnormal{val}(a_p)\geq 1$. Thus, $\textnormal{val}(a_\alpha)+\textnormal{val}(a_p)\geq f(\alpha_p)$ and we get $\textnormal{val}(a_{\alpha'})\geq \max\{f(\alpha_p),c,f(\alpha)\}$. By definition, we must have $b_{\alpha'}\leq \max\{b_{\alpha},b_{\alpha_p}\}$. If $\alpha$ and $\alpha_p$ are of the same sign, then $\alpha'$ must be of the same sign, and $f(\alpha')\leq \max\{f(\alpha),f(\alpha_p)\}$. Thus, $\textnormal{val}(a_{\alpha'})\geq \max\{c,f(\alpha')\}$, which shows that $u_{\alpha'}(a_{\alpha'})\in \mathcal{U}_c$. If $\alpha_p$ is positive and $\alpha$ is negative, we get $f(\alpha)'\leq \max\{f(-\alpha_p),f(\alpha)\}$, Notice that if $b_{\alpha_p}-1< f(-\alpha_p)$, we must have $b_{\alpha_p}=1$ and $\textnormal{val}(a_p)=0$. Under this condition, $f(-\alpha_p)=1$, and obviously $f(\alpha)'\leq \max\{f(-\alpha_p),f(\alpha)\}=f(\alpha)$. This implies that $\textnormal{val}(a_{\alpha'})\geq f(\alpha')$ and $u_{\alpha'}(a_{\alpha'})\in \mathcal{U}_c$. If $\alpha$ is positive and $\alpha_p$ is negative, we have $f(\alpha')\leq \max\{f(\alpha_p),f(\alpha)+1\}$. Notice that $\textnormal{val}(a_p)\geq 1$, so $\textnormal{val}(a_p)+\textnormal{val}(a_{\alpha})\geq f(\alpha)+1$. Therefore, $\textnormal{val}(a_{\alpha'})\geq \max\{f(\alpha_p),c,f(\alpha)+1\}\geq \max\{c,f(\alpha')\}$, which shows that $u_{\alpha'}(a_{\alpha'})\in \mathcal{U}_c$. In conclusion, $u_{\alpha'}(a_{\alpha'})\in \mathcal{U}_c$ and the proof of part (a) is complete.

            For part (b), we apply \Cref{2} directly. By the assumption, we have $\textnormal{val}(a_{\alpha}a_p)\geq c+c_p\geq \max\{c,b_{\alpha_p}\}$. Therefore, the element $t=(-\alpha^{\vee})(1-a_{\alpha}a_p)$ belongs to $\mathcal{T}_c$. From the computations in the preceding paragraph, we have $\textnormal{val}(a_{\alpha}^2a_p)\geq \max\{c,f(\alpha)\}$ and $\textnormal{val}(a_{\alpha}a_p^2)\geq \max\{c,f(\alpha_p)\}$. This finishes the proof of part (b).

            For part (c), we use \Cref{3} directly. Note that $$\textnormal{val}(((1+b)^{(\alpha^{\vee},\alpha_p)}-1)a_{p})\geq \textnormal{val}(b)+\textnormal{val}(a_p)\geq c+c_p\geq \max\{c,f(\alpha_p)\}.$$
            This finishes the proof of part (c).
        \end{proof}

    We go back to the proof of \Cref{axiom for u}. Recall that we already choose a root $\alpha_k$. We make the following claim:
     \begin{claim}\label{claim 1}
        We may assume $c_k<f(\alpha_k)$.
    \end{claim}
    \begin{proof}
        If $c_k\geq f(\alpha_k)$, we will prove that we could remove this term from the expression of $u$. Note that $c_u= b_{\alpha_k}-f(\alpha_k)-1=f(-\alpha_k)-1\leq f(\alpha_k)\leq c_k $. Therefore, $u_{\alpha_k}(a_k)\in \mathcal{U}_{c_u}$.
        
        From our definition of $\alpha_k$, we can see that for any $i>p$, $b_{\alpha_i}-c_i-1<c_u$. Then by \Cref{commutator lemma}, we can write the commutator $[u_{\alpha_k}(a_k)^{-1}, (\prod_{i=k+1}^nu_{\alpha_i}(a_i))^{-1}]$ as a product in $\mathcal{U}_{c_u}$ and $\mathcal{T}_{c_u}$, which is in $\textnormal{ker}(\rho)$. Therefore, for any $u'\in \bar{U}$ such that $u^{-1}u'u\in J$, we have 
        $$\rho(u^{-1}u'u)=\rho((\prod_{1\leq i\leq n, i\neq k}u_{\alpha_i}(a_i))^{-1}u'(\prod_{1\leq i\leq n, i\neq k}u_{\alpha_i}(a_i)).$$
        We successfully removed the term $u_{\alpha_k}(a_k)$, and the claim is proved.
        
    \end{proof}
    
    From now on, we assume that $u$ satisfies \Cref{claim 1}. We will prove, by induction on $n$, that for all such $u$'s, with $u\notin J\cap {^{w^{-1}}\bar{U}}$, we can choose an element of the form $u'=u_{-\alpha_k}(b)\in {^{w^{-1}}U}$, with $b\in\mathfrak{p}^{c_u}$, such that $\rho(u^{-1}u'u)\neq 1$. Note that 
    \begin{equation}\label{valuation for b}
    c_u\geq b_{\alpha_k}-1-(f(\alpha_k)-1)=b_{\alpha_k}-f(\alpha_k)=f(-\alpha_k).
    \end{equation}
    This implies that $u'\in\mathcal{U}_{c_u}$.

    For $n=1$, we can write $u=u_{\alpha_1}(a_1)$, with $c_1<f(\alpha_1)$. By the definition of $\textnormal{cond}(\alpha)$, we can find an element $b\in \mathfrak{p}^{b_{\alpha_1}-1-c_1}=\mathfrak{p}^{c_u}$ such that $^0\chi\circ \alpha_1^{\vee}((1+a_1b)^{-1})\neq 1$. Take $u'=u_{-\alpha_1}(b)$. Then we have the following equation:
    \begin{equation}\label{conjugation relation}
        u^{-1}u'u=u_{\alpha_1}(-a_1)u_{-\alpha_1}(b)u_{\alpha_1}(a_1)=u_{\alpha_1}(\frac{-a_1^2b}{1+a_1b})\alpha_1^{\vee}((1+a_1b)^{-1})u_{-\alpha_1}(\frac{b}{1+ab}).
    \end{equation}
    \Cref{valuation for b} tells us that $\textnormal{val}(b)\geq f(-\alpha_1)$. Therefore, $u_{\alpha_1}(\frac{-a_1^2b}{1+a_1b})$ and $u_{-\alpha_1}(\frac{b}{1+ab})$ are both in $\mathcal{U}_{c_u}$, and $\rho(u^{-1}u'u)=\rho(\alpha_1^{\vee}((1+a_1b)^{-1}))={^0\chi\circ \alpha_1^{\vee}((1+a_1b)^{-1})\neq 1}$.

    For $n-1$, suppose that the statement is correct. Then for $n$, let $u=\prod_{i=1}^nu_{\alpha_i}(a_i)$ and $u_0=\prod_{i=2}^nu_{\alpha_i}(a_i)$. We have three cases:
    \begin{example}
        $\alpha_k=\alpha_1$: for this case, we can see that $c_u> b_{\alpha_i}-1-c_i$ for any $i>1$. As the computation for the case $n=1$, we can find $u'=u_{-\alpha_1}(b)$, with $b\in \mathfrak{p}^{c_u}$, such that $\rho(u_{\alpha_1}(-a_1)u'u_{\alpha_1}(a_1))\neq 1$. Notice that $u_{\alpha_1}(\frac{-a_1^2b}{1+a_1b})$ and $u_{-\alpha_1}(\frac{b}{1+ab})$ are both in $\mathcal{U}_{c_u}$. 
        
        By \Cref{commutator lemma}, we have
        $$\rho(u^{-1}u'u)=\rho(u_0^{-1}\alpha_1^{\vee}((1+a_1b)^{-1})u_0).$$
        We have $u_{\alpha_2}(a_2)^{-1}\alpha_1^{\vee}((1+a_1b)^{-1})u_{\alpha_2}(a_2)=\alpha_1^{\vee}((1+a_1b)^{-1})\cdot [\alpha_1^{\vee}((1+a_1b)^{-1})^{-1},u_{\alpha_2}(a_2)^{-1}].$
        Note that as in the proof of \Cref{commutator lemma}(c), $\textnormal{val}(a_1a_2b)\geq \textnormal{val}(b)+\textnormal{val}(a_2)\geq c_u-1+c_2\geq b_{\alpha_2}-1$. This shows that $[\alpha_1^{\vee}((1+a_1b)^{-1})^{-1},u_{\alpha_2}(a_2)^{-1}]$ is also in $\mathcal{U}_{c_u}$. By \Cref{commutator lemma} again, we can get
        $$\rho(u_0^{-1}u'u_0)=\rho((\prod_{i=3}^nu_{\alpha_i}(a_i))^{-1}u'\prod_{i=3}^nu_{\alpha_i}(a_i)).$$
        For the same reason, we can remove all conjugation actions by $u_{\alpha_i}(a_i)$ for $2\leq i\leq n$. Therefore, we have 
        $$\rho(u^{-1}u'u)=\rho(u_{\alpha_1}(-a_1)u'u_{\alpha_1}(a_1))\neq 1.$$ 
        The statement holds for $n$.
        
    \end{example}

    \begin{example}
        $\alpha_k\neq \alpha_1$ and $\alpha_k$ is positive: under this case, the element $u_0$ also satisfies \Cref{claim 1} and $c_u=c_{u_0}$. By induction hypothesis, we can choose $u'=u_{-\alpha_k}(b)$, with $b\in\mathfrak{p}^{c_u}=\mathfrak{p}^{c_{u_0}}$, such that $\rho(u_0^{-1}u'u_0)\neq 1$. Note that the $\alpha_k$ we choose for $u_0$ is the same as the $\alpha_k$ we choose for $u$.
        
        Consider the element $\tilde{u}=[u'^{-1}, u_{\alpha_1}(a_1)^{-1}]=[u_{\alpha_k}(-b), u_{\alpha_1}(-a_1)]$. By \Cref{commutator lemma}(a), $\tilde{u}$ can be written as a product of elements in $\mathcal{U}_{c_u}$. If $\tilde{u}=1$, we have $u^{-1}u'u=u_0^{-1}u'u_0$, so the statement holds for $n$. Otherwise, we can find positive integers $x,y$ such that $\alpha'\coloneqq x\alpha_i-y\alpha_k$ is a root. Let $u_{\alpha'}(a_{\alpha'})\coloneqq u_{\alpha'}(C_{\alpha_1,-\alpha_k;x,y}a_1^xb^y)$. If $y=1$, $\alpha'$ will also be a positive root, so $u_{\alpha'}(a_{\alpha'})\in \mathcal{U}_{c_u}^+$. Assume that $\alpha_l\in \Phi^+$ and $\alpha_{l+1}\in \Phi^-$ for a positive integer $l>k$. By \Cref{commutator lemma}(a), the commutator $[(\prod_{i=2}^lu_{\alpha_i}(a_i))^{-1},u_{\alpha'}(a_{\alpha'})^{-1}]$ can be written as a product of elements in $\mathcal{U}_{c_u}^+$. By our choice of $\alpha_k$, for any $i\geq l+1$, we have $c_u>b_{\alpha_i}-1-c_i$. Combining these with \Cref{commutator lemma} again, we conclude that the commutator $[u_0^{-1},u_{\alpha'}(a_{\alpha'})^{-1}]$ is a product of elements in $\mathcal{U}_{c_u}$ and $\mathcal{T}_{c_u}$. 
        
        If $y\geq 2$, $u_{\alpha'}(a_{\alpha'})\in \mathcal{U}_{2c_u}$. By \Cref{commutator lemma}, we can write $[u_{\alpha'}(a_{\alpha'})^{-1},u_0^{-1}]$ as a product of elements in $\mathcal{U}_{2c_u}\subset \mathcal{U}_{c_u}$ and $\mathcal{T}_{2c_u}\subset \mathcal{T}_{c_u}$. We conclude that the commutator $[\tilde{u}^{-1},u_0^{-1}]$ is a product of elements in $\mathcal{U}_{c_u}$. Therefore, $\rho(u^{-1}u'u)=\rho(u_0^{-1}u'u_0)\neq1$, and the statement holds for $n$.
        
    \end{example}

    \begin{example}
        $\alpha_k\neq \alpha_1$ and $\alpha_k$ is negative: under this case, the element $u_0$ also satisfies \Cref{claim 1} and $c_u=c_{u_0}$. By induction hypothesis, we can choose $u'=u_{-\alpha_k}(b)$, with $b\in\mathfrak{p}^{c_u}=\mathfrak{p}^{c_{u_0}}$, such that $\rho(u_0^{-1}u'u_0)\neq 1$, and the $\alpha_k$ we choose for $u$ is the same as the $\alpha_k$ we choose for $u_0$. Note that for any $i>k$, $c_u> b_{\alpha_i}-1-c_i$.

        We will use the same notations from the preceding case. Let $\tilde{u}=[u_{\alpha_1}(a_1)^{-1},u'^{-1}]$. If $\tilde{u}=1$, we can get $u^{-1}u'u=u_0^{-1}u'u_0$, and the statement holds for $n$. Otherwise, we choose positive integers $x,y$ such that $\alpha'\coloneqq x\alpha_i-y\alpha_k$ is a root. Let $u_{\alpha'}(a_{\alpha'})\coloneqq u_{\alpha'}(C_{\alpha_1,-\alpha_k;x,y}a_1^xb^y)$. Note that $-\alpha_k$ is positive, and by \Cref{commutator lemma}(a), $u_{\alpha'}(a_{\alpha'})\in \mathcal{U}_{c_u}^+$. Assume that the roots $\alpha_1,\cdots,\alpha_l$ are positive, and $\alpha_{l+1},\cdots,\alpha_n$ are negative. Since $\alpha_k$ is negative, we have $k>l$. Let $u_+=\prod_{i=2}^lu_{\alpha_i}(a_i)$ and $u_-=\prod_{i=l+1}^nu_{\alpha_i}(a_i)$. Again, by \Cref{commutator lemma}(a), $[u_+^{-1},u_{\alpha'}(a_{\alpha'})^{-1}]$ can be written as a product of elements in $\mathcal{U}_{c_u}^+$. We pick one term in the product, and denote it by $u_{\alpha_+}(a_{\alpha_+})$. If the sum of coefficients of $\alpha_+$ is strictly greater than the sum of coefficients of $-\alpha_k$, we get the fact that $\alpha_i\neq -\alpha_+$ for $l+1\leq i\leq k$. If $u_{\alpha_+}(a_{\alpha_+})$ commutes with $\prod_{i=l+1}^ku_{\alpha_i}(a_i)$, we have the following equation:
        $$[u_-^{-1},u_{\alpha_+}(a_{\alpha_+})^{-1}]=[(\prod_{i=k+1}^nu_{\alpha_i})^{-1},u_{\alpha_+}(a_{\alpha_+})^{-1}].$$
        Note that for $i\geq k+1$, $c_u+c_i\geq b_{\alpha_i}$. \Cref{commutator lemma} tells us that $[u_{\alpha_+}(a_{\alpha_+})^{-1},(\prod_{i=k+1}^nu_{\alpha_i})^{-1}]$ is a product of elements in $\mathcal{U}_{c_u}$. If $u_{\alpha_+}(a_{\alpha_+})$ does not commute with $\prod_{i=l+1}^ku_{\alpha_i}(a_i)$, denote the smallest $i$ such that $u_{\alpha_i}(a_i)$ does not commute with $u_{\alpha_+}(a_{\alpha_+})$ by $m$. Then we can see that the commutator of $u_{\alpha_m}(a_m)$ and $u_{\alpha_+}(a_{\alpha_+})$ can be written as a product of elements in $\mathcal{U}_{c_u+1}$ (this is because $c_i\geq 1$). By \Cref{commutator lemma}, the commutator of an element in $\mathcal{U}_{c_u+1}$ and $u_{\alpha_i}(a_i)$ can be written as a product of element in $\mathcal{U}_{c_u+1}$ and $\mathcal{T}_{c_u+1}$, and the commutator of an element in $\mathcal{T}_{c_u+1}$ and $u_{\alpha_i}(a_i)$ is in $\mathcal{U}_{c_u+1}$. Therefore, $[u_-^{-1},u_{\alpha_+}(a_{\alpha_+})^{-1}]\in \textnormal{ker}(\rho)$, and we have $\rho(u^{-1}u'u)=\rho(u_0^{-1}u'u_0)\neq 1$. The statement holds for $n$.

        If the sum of coefficients of $\alpha_+$ is equal to the sum of coefficients of $-\alpha_k$, we can get that $\alpha_1$ is negative. Together with the fact that $u'$ and $u_{\alpha_1}(a_1)$ do not commute, we can reduce it to the case above, and the statement holds for $n$
    \end{example}

    In conclusion, for any $u\in (I\cap {^{w^{-1}}\bar{U}})\backslash(J\cap {^{w^{-1}}\bar{U}})$, we find an element $u'=u_{-\alpha_k}(b)\in {^{w{-1}}U}$ such that $\rho(u^{-1}u'u)\neq 1$. The proof of \Cref{axiom for u} for split groups and $P=B$ is complete.

    If $P$ is not $B$, then we could still write $u$ as $u=\prod_{i=1}^nu_{\alpha_i}(a_i)$, with each $u_{\alpha_i}(a_i)\in {^{w^{-1}}\bar{N}}$. The $u'$ we pick belongs to $^{w^{-1}}N$. So \Cref{axiom for u} is also true for general parabolic subgroups $P$. The proof for split groups is complete.

\subsubsection{Proof for unramified groups}
    
    Next, we prove \Cref{axiom for u} for unramified groups. We still first work with $P=B=TU$. Therefore, $w$ could be any element in the Weyl group $W$. Denote the relative root system by $\Phi_F$ and absolute root system by $\Phi_E$ and let $h$ be a generator of $\textnormal{Gal}(E/F)$. Define $\textnormal{Res}$ to be the restriction map from $\Phi_E\rightarrow \Phi_F$. Let $ \lvert \textnormal{Gal}(E/F)\rvert=e $. Let $(J_E,\rho_E)$ denote the type we construct for the principal block $\mathfrak{s}_E=(T(E),\chi_E)_{G(E)}$.
    
    For any $u\in (I\cap {^{w^{-1}}\bar{U}})\backslash(J\cap {^{w^{-1}}\bar{U}})$, we need to find an element $u'\in {^{w^{-1}}U}$ such that $\rho(u^{-1}u'u)\neq 1$. From the preceding section, if we regard $u$ as an element in $I_E\cap {^{w^{-1}}\bar{U}(E)}$ and write $u=\prod_{i=1}^nu_{\alpha_i}(a_i)$ satisfying \Cref{claim 1} with $\alpha_i\in \Phi_E$, we can find $u_E'=u_{-\alpha_k}(b)\in {^{w^{-1}}U}_E$, with $\alpha_k\in \Phi_E$ and $b\in \mathfrak{p}_E^{c_u}$, such that $\rho_E(u^{-1}u'_Eu)\neq 1$. Denote $\textnormal{Res}(-\alpha_k)$ by $-\alpha_{k,F}$.

    Let $N_E(g)=\prod_{i=1}^eh^i(g)$ for any $g\in G(E)$. Define
    $$U_{\alpha}^j=\prod_{\beta\in \Phi_E,\textnormal{Res}(\beta)=\alpha}u_{\beta}(\mathfrak{p}_E^j)$$
    for a positive integer $j$ and $\alpha\in \Phi_F$. Since $-j\alpha_{k,F}$ cannot be a relative root for any $j\geq 3$, by \Cref{1}, we can write $u^0:=h(N_E(u_E'))\cdot N_E(u_E')^{-1}$ as a product of elements of the form $u_{\beta}(a_{\beta})$, where $\beta\in \Phi_E$, $\textnormal{Res}(\beta)=-2\alpha_{k,F}$ and $a_{\beta}\in \mathfrak{p}_E^{2c_u}$. Therefore, $h(N_E(u_E'))\cdot N_E(u_E')^{-1}\in U_{-2\alpha_{k,F}}^{2c_u}$. For the same reason, the group $U_{-2\alpha_{k,F}}^{2c_u}$ is abelian. We need the following lemma:
    \begin{lemma}\label{cohomology statement}
        The first Galois cohomology of $U_{-2\alpha_{k,F}}^{2c_u}$ is trivial, i.e. 
        $$H^1(\textnormal{Gal}(E/F),U_{-2\alpha_{k,F}}^{2c_u})=1.$$
    \end{lemma}
    \begin{proof}
        We only need to consider one Galois orbit. Pick $\beta_1\in \Phi_E$ such that $\textnormal{Res}(\beta_1)=-2\alpha_{k,F}$. Denote the stabilizer of $\beta_1$ in $\textnormal{Gal}(E/F)$ by $H$. Let $\lvert H \rvert=e_1$ and $e_2=e/e_1$. Therefore, $h^{e_2}$ is a generator of $H$. Denote it by $h_1$. We have the fact that $U_{-2\alpha_{k,F}}^{2c_u}$ is isomorphic to $\textnormal{Ind}_H^{\textnormal{Gal}(E/F)}u_{\beta_1}(\mathfrak{p}_E^{2c_u})$ as $\textnormal{Gal}(E/F)$-modules. By Shapiro's Lemma, 
        $$H^1(\textnormal{Gal}(E/F),U_{-2\alpha_{k,F}}^{2c_u})=H^1(H,u_{\beta_1}(\mathfrak{p}_E^{2c_u})).$$

        Let $E'=E^H$. The action of $H$ on $u_{\beta_1}(\mathfrak{p}_E^{2c_u})$ is as follows: for any $a\in \mathfrak{p}_E^{2c_u}$, $h_1u_{\beta_1}(a)=u_{\beta_1}(d_{\beta_1}h_1(a))$, for some constant $d_{\beta_1}\in \mathfrak{O}_E^{\times}$. Since $h_1^{e_1}$ acts as the identity on $u_{\beta_1}(\mathfrak{p}_E^{2c_u})$, we get $\prod_{i=1}^{e_1}h_1^i(d_{\beta_1})=1$. Note that this is the norm map from $E$ to $E'$. By Hilbert's Theorem 90, we can find and element $d\in \mathfrak{O}_{E'}^{\times}$, such that $d_{\beta_1}=h_1(d)d^{-1}$. The map $a\rightarrow da$ induces an $H$-module isomorphism from $u_{\beta_1}(\mathfrak{p}_E^{2c_u})$ to $\mathfrak{p}_E^{2c_u}$. Again by Hilbert's Theorem 90, $H^1(H,\mathfrak{p}_E^{2c_u})=1$. This finishes the proof. 
    \end{proof}

        Note that $N_E(u^0)=1$. From the above lemma, we can find $u_1\in U_{-2\alpha_{k,F}}^{2c_u}$ such that $u^0=u_1h(u_1)^{-1}$. Therefore, $h(N_E(u_E')u_1)=N_E(u_E')u_1$, which shows that $N_E(u_E')u_1\in ({^{w^{-1}}U}_E)^{\textnormal{Gal}(E/F)}={^{w^{-1}}U}$. Note that $u_1\in \mathcal{U}_{2c_u}$. Let $u'=N_E(u_E')u_1$ and $u_2=u^{-1}u_1u=u_1[u_1^{-1},u^{-1}]$. We have 
        \begin{equation}
            u^{-1}u'u=u^{-1}N_E(u'_E)uu_2=N_E(u^{-1}u'_Eu)u_2.
        \end{equation}
        By \Cref{commutator lemma}, we can write $u_2$ as a product of elements in $\mathcal{U}_{2c_u}\subset \mathcal{U}_{c_u}$ and $\mathcal{T}_{2c_u}\subset \mathcal{T}_{c_u}$. Further, $u^{-1}u'_Eu$ could be written in the form of $u^-tu^+$, with $t\in T(\mathfrak{O}_E)$, $u^-\in J_E\cap \bar{U}(E)$ and $u^+\in J_E\cap U(E)$. One can see that $J_E\cap \bar{U}(E)$ and $J_E\cap U(E)$ are both $\textnormal{Gal}(E/F)$-stable, and normalized by $T(\mathfrak{O}_E)$. Therefore, we have 
        $$N_E(u^{-1}u'_Eu)=N_E(u^-tu^+)=\prod_{i=1}^eh^i(u^-tu^+)=N_E(t)\cdot u_3,$$
        where $u_3$ can be written as a product of elements in $J_E\cap \bar{U}(E)$ and $J_E\cap U(E)$. The formula $\rho_E(u^{-1}u_E'u)\neq1$ implies that $^0\chi(N_E(t))\neq 1$. With all the discussion above, we can write $u^{-1}u'u=N_E(t)\cdot u_3 u_2$. Since both $u^{-1}u'u$ and $N_E(t)$ are $\textnormal{Gal}(E/F)$-invariant, $u_3u_2$ is also $\textnormal{Gal}(E/F)$-invariant.

        Let $U_f=\langle u_{\alpha}(\mathfrak{p}_E^{f(\alpha)}):\alpha\in \Phi_E\rangle$. Further, let $U_f^-=U_f\cap \bar{U}(E)$ and $U_f^+=U_f\cap U(E)$. By \cite[Lemma 3.2]{roche1998types}, we have
        $$U_f=U_f^-{^0T_f}U_f^+,$$
        where $^0T_f=\prod_{\alpha\in \Phi_E}\alpha^{\vee}(1+\mathfrak{p}_E^{b_\alpha})$. It is easy to see that $U_f\in \textnormal{ker}(\rho_E)$. Also, since the above equation is a bijection, we have $(U_f)^h=(U_f^-)^h({^0T_f})^h(U_f^+)^h$. 
        
        Note that $\mathcal{U}_{c_u}$ and $\mathcal{T}_{c_u}$ are both contained in $U_f$. This implies that $u_3u_2\in U_f$. We can write $u_3u_2$ as $u_f^-t_fu_f^+$, where $u_f^-\in U_f^-$, $t_f\in {^0T_f}$ and $u_f^+\in U_f^+$. Since $u_3u_2$ is $\textnormal{Gal}(E/F)$-invariant, we can further require that $u_f^-\in (U_f^-)^h$, $t_f\in ({^0T_f})^h$ and $u_f^+\in (U_f^+)^h$. Therefore, $\rho(u^{-1}u'u)=\rho(N_E(t)\cdot t_f)$. We need the following lemma:
        \begin{lemma}
            The norm map $N_E:{^0T_f}\rightarrow ({^0T_f})^h$ is surjective.
        \end{lemma}
        \begin{proof}
            First, note that $f$ is $\textnormal{Gal}(E/F)$-stable and $^0T_f$ is abelian. The Galois action on $^0T_f$ is as follows:
            $$h(\alpha^{\vee}(1+a))=(h(\alpha))^\vee(1+h(a)).$$ Therefore, we only need to consider one Galois orbit of roots.

            Fix a root $\alpha\in\Phi_E$. Denote the stabilizer of $\alpha$ in $\textnormal{Gal}(E/F)$ by $H_{\alpha}$. Let $\lvert H_{\alpha}\rvert=e_{3}$, and $e_{4}=e/e_{3}$. Therefore, $h_{\alpha}\coloneqq h^{e_{4}}$ is a generator of $H_{\alpha}$. Let $E_{\alpha}=E^{H_{\alpha}}$. Let $t_{\alpha}=\prod_{i=1}^{e_{4}}(h^i\alpha)^{\vee}((1+x_i))$ be $\textnormal{Gal}(E/F)$-invariant. This requires that $h(x_i)=x_{i+1}$ (define $x_{e_{4}+1}=x_1)$. Therefore, $h^{e_{4}}(x_1)=x_1$, which shows that $x_1\in 1+\mathfrak{p}_{E_{\alpha}}^{b_{\alpha}}$. Since the norm map from $1+\mathfrak{p}_E^{b_{\alpha}}$ to $1+\mathfrak{p}_{E_{\alpha}}^{b_{\alpha}}$ is surjective, we can choose an element $y\in \mathfrak{p}_E^{b_{\alpha}}$ such that $\prod_{i=1}^{e_3}h_{\alpha}^i(y)=1+x_1$. Therefore, $t_{\alpha}=N_E(\alpha^{\vee}(1+y))$. This finishes the proof.
        \end{proof}

        By the above lemma, we can find $t_f'\in {^0T_f}$ such that $N_E(t_f')=t_f$. Therefore, $\rho(t_f)=\rho(N_E(t_f'))=\rho_E(t_f')=1$, and we have $\rho(u^{-1}u'u)=\rho(N_E(t)\cdot t_f)=\rho(N_E(t))\neq 1$. This proves \Cref{axiom for u} for unramified groups when $P=B$. For a general standard parabolic subgroup $P=MN$, we could still write $u$ as $u=\prod_{i=1}^nu_{\alpha_i}(a_i)$, with each $u_{\alpha_i}(a_i)\in {^{w^{-1}}\bar{N}(E)}$. The $u'$ we find from the proof above belongs to $({^{w^{-1}}N(E)})^h={^{w^{-1}}N}$. This finishes the proof of \Cref{axiom for u}.

\section{Vanishing statement and reduction steps}\label{section 7}

\subsection{Vanishing statement}

\begin{lemma}\label{vanishing statement}
    Let $\phi\in \mathcal{Z}(G_r,\rho_r)$ and let $f=b_r\phi\in \mathcal{Z}(G,\rho)$. If $\gamma$ is not a norm from $G_r$, then $\textnormal{O}_{\gamma}(f)=0$.
\end{lemma}

    We basically follow the proof in \cite[\S 7.1]{haines2012base}. \Cref{twisted descent formula for general function} and \Cref{descent formula for general function} help us reduce to the case where $\gamma$ is an elliptic element in $G$. For the canonical map $p:G_{sc}\rightarrow G$ and the abelian group 
    $$\mathbf{H}_{ab}^0(F,G)=G(F)/p(G_{sc}(F)),$$
    by \cite[Proposition 2.5.3]{labesse1999cohomologie}, an elliptic element is a norm from $G_r$ if and only if its image in $\mathbf{H}_{ab}^0(F,G)$ is a norm. Now the following lemma finishes the proof:
    \begin{lemma}
        Let $f=b_r\phi\in \mathcal{Z}(G,\rho)$. Let $x\in G(F)$ be any element such that, for some character $\eta$ on the group $\mathbf{H}_{ab}^0(F,G)$ which is trivial on the norms, we have $\eta(x)\neq 1$. Then $f(x)=0$.
    \end{lemma}
    \begin{proof}
        By pulling back, we can view $\eta$ as a character on $G(F)$. We can easily see that $\eta$ is actually an unramified character. This means that $\eta$ is trivial on $J$. Moreover, the restriction of $\eta$ on the torus $T(F)$ is also an unramified character. Therefore, $f\eta\in \mathcal{H}(G,\rho)$.

        Next, we will examine the action of $f\eta$ on $i_B^G(\chi)^{\rho}$, where $\chi$ is an arbitrary extension of $^0\chi$. We pick an arbitrary function $h\in i_B^G(\chi)^{\rho}$. We can see that
        \begin{align*}
            (f\eta)\bullet h(g_0) & = \int_G \; h(g)(f\eta)^{\vee}(g^{-1}g_0) \; dg \\
            & = \int_G \; h(g)(f\eta)(g_0^{-1}g) \; dg \\
            & = \eta(g_0)^{-1}\;(f\bullet (\eta h))(g_0)
        \end{align*}
        It is easy to see that the function $\eta h$ belongs to the space $i_B^G(\eta\chi)^{\rho}$. Since $f\in \mathcal{Z}(G,\rho)$, we have
        $$(f\bullet(\eta h))(g_0)=\beta(f)(\eta\chi)\cdot (\eta h)(g_0).$$
        Note that $\beta(f)(\eta\chi)$ is a constant independent of $g_0$. Thus we can get 
        $$(f\eta)\bullet h(g_0)=\beta(f)(\eta\chi)\cdot \eta(g_0)^{-1}(\eta h)(g_0)=\beta(f)(\eta\chi)\cdot h(g_0).$$
        Therefore, $f\eta$ acts on $i_B^G(\chi)^{\rho}$ by the same scalar as $f$ acts on $i_B^G(\eta\chi)^{\rho}$. This implies that $f\eta\in \mathcal{Z}(G,\rho)$. However, this scalar is the same as $\phi$ acts on $i_{B_r}^{G_r}(\eta_r\chi_r)^{\rho_r}=i_{B_r}^{G_r}(\chi_r)^{\rho_r}$ (because $\eta_r$ is trivial). This is the same as the scalar of $f$ acting on $i_B^G(\chi)^{\rho}$. Therefore, we have $f\eta=f$, and in particular
        $$f(x)(\eta(x)-1)=0.$$
        Since $\eta(x)-1\neq 0$, we have $f(x)=0$. 
        
    \end{proof}

\subsection{Reduction steps}\label{reduction steps}

    In this section, we plan to develop several reduction steps. We should claim that, in this section, we are going to deal with all functions inside of $\mathcal{Z}(G_r,\rho_r)$, not just functions of the form $Z_r*e_{\rho_r}$.

    By \Cref{vanishing statement}, we may assume that $\gamma=\mathscr{N}(\delta)$ for some $\theta$-semisimple element inside $G_r$. We mainly follow \cite[\S5.4]{haines2009base}, with some modifications:

    \begin{itemize}
        \item[(a)] We may assume that $G_{\textnormal{der}}=G_{\textnormal{sc}}$. The theory is generalized as follows:

    Recall that $G$ is split over $E$, where $E/F$ is a finite unramified extension. Choose a finite unramified extension $F'/F$ such that $F'$ contains $F_r$ and $E$. Consider a $z$-extension of $F$-groups
    \begin{equation}\label{z-extension}
    1 \longrightarrow Z\longrightarrow H\stackrel{p}{\longrightarrow} G\longrightarrow 1,
    \end{equation}
    where $Z$ is a finite product of copies of $R_{F'/F}\mathbb{G}_m$, and $H$ is an unramified reductive group over $F$ such that $H_{\textnormal{der}}=H_{\textnormal{sc}}$. Recall that $Z$ is central in $H$, and $p$ is surjective on $F$-points and $F_r$-points by the triviality of higher Galois cohomology of $Z$. As in loc. cit., the norm homomorphism induces a surjective map $N:{^0Z(F_r)}\rightarrow {^0Z(F)}$.

    Let $\lambda:Z(F)\rightarrow \mathbb{C}^{\times}$ denote a smooth character, and for $f\in C_c^{\infty}(H(F))$, set 
    $$f_{\lambda}(h)=\int_{Z(F)}\; f(hz)\lambda^{-1}(z)\; dz,$$
    where $dz$ is the Haar measure on $Z(F)$ giving $^0Z(F)$ measure 1. Set $\lambda N=\lambda \circ N$ be a character on $Z(F_r)$. Similarly we can define $\phi_{\lambda N}$ for all $\phi\in C_c^{\infty}(H(F_r))$. Note that (twisted) orbital integrals exist at all ($\theta$)-semisimple elements for all such functions. Write $\lambda=1$ for the trivial character, and $\bar{f}$ (resp. $\bar{\phi}$) for the function $f_1$ (resp. $\phi_{1N}$) when it is viewed as an element in $C_c^{\infty}(G(F))$ (resp. $C_c^{\infty}(G(F_r))$).

    Let $T_H=p^{-1}(T)$, a maximal torus in $H$. For a fixed character $^0\chi$ on $^0T$, we define a character $^0\chi_H$ on $^0T_H$ by $^0\chi_H={^0\chi}\circ p$. Therefore, this character is trivial on $^0Z(F)$. Since $H$ and $G$ have the same adjoint group, they have the same root datum. Let $(J,\rho)$ and $(J_H,\rho_H)$ be the types corresponding to the Bernstein block $\mathfrak{s}=({^0T},{^0\chi})$ and $\mathfrak{s}_H=({^0T}_H,{^0\chi}_H)$, respectively. Then we have the following two equations:
    \begin{equation}\label{z extension property 1}
        p(J_H)=J
    \end{equation}
    \begin{equation}\label{z extension property 2}
        \rho_H=\rho\circ p
    \end{equation}
     Similarly, we have the same statements for $(J_r,\rho_r)$ and $(J_{H_r},\rho_{H_r})$.

    We have the following lemma:
    \begin{lemma}\label{derived 1}[\cite{haines2012base}, Lemma 7.2.1]
        Let $\phi\in C_c^{\infty}(H(F_r))$ and $f\in C_c^{\infty}(H(F))$. We have the following statements:
        \begin{itemize}
            \item[(1)] The functions $\phi,f$ are associated if and only if $\phi_{\lambda N},f_{\lambda}$ are associated for every $\lambda$.
            \item[(2)] Suppose that $\phi\in \mathcal{H}(G_r,\tilde{K},\tilde{\rho})$ (resp. $f\in \mathcal{H}(G,K,\rho)$ for a compact open subgroups $\tilde{K}\subset G_r$ and character $\tilde{\rho}:\tilde{K}\rightarrow \mathbb{C}^{\times}$ (resp. $K\subset G$ and $\rho:K\rightarrow \mathbb{C}^{\times}$ such that
            \begin{itemize}
                \item[-] $N(\tilde{K}\cap Z(F_r))=K\cap Z(F),$
                \item[-] $\tilde{K} \supset (1-\theta)(Z(F_r)),$ and
                \item[-] $\tilde{\rho}|_{\tilde{K}\cap Z(F_r)}=\rho\circ N|_{\tilde{K}\cap Z(F_r)}.$
            \end{itemize}
            Then in (1) we only need to consider characters $\lambda$ such that 
            $$\lambda|_{K\cap Z(F)}=\rho^{-1}|_{K\cap Z(F)}.$$
            \item[(3)] If $\phi\in \mathcal{H}(H_r,\rho_{H_r})$ and $f\in \mathcal{H}(H,\rho_H)$, then in (1) we only need to consider the characters $\lambda$ with $\lambda|_{^0Z(F)}={^0\chi^{-1}}|_{^0Z(F)}$.
            \item[(4)] The pair $\phi_1,f_1$ are associated if and only if $\bar{\phi},\bar{f}$ are associated.
        \end{itemize}
    \end{lemma}
    \begin{proof}
        See \cite[Lemma 5.3.1]{haines2009base} for details. Note that in loc. cit., only depth zero characters are considered, but the same argument can be extended to positive depth characters without any change.
    \end{proof}

    We need one more lemma:
    \begin{lemma}[{\cite[Lemma 7.2.2]{haines2012base}}]\label{derived 2}
        The map $\phi\rightarrow \bar{\phi}$ determines a surjective homomorphism from $\mathcal{Z}(H_r, \rho_{H_r})$ to $\mathcal{Z}(G_r,\rho_r)$. Moreover, this map is compatible with the base change homomorphism in the sense that
        \begin{equation}
            b_r(\bar{\phi})=\overline{b_r(\phi)}
        \end{equation}
    \end{lemma}
    \begin{proof}
        First, from \Cref{z extension property 1} and \Cref{z extension property 2}, we can see that $\bar{\phi}\in \mathcal{H}(G_r,\rho_r)$ and $\bar{f}\in \mathcal{H}(G,\rho)$. The rest of the proof is basically the same as [\cite{haines2009base}, Lemma 5.3.2]. Instead of the Bernstein isomorphism, we use the isomorphism $\beta:\mathbb{C}[\mathfrak{X}_{\mathfrak{s}}]\xrightarrow{\sim}\mathcal{Z}(G,\rho)$. The map $p$ induces a morphism $p^*:\mathfrak{X}_s\rightarrow\mathfrak{X}_{\mathfrak{s}_H}$ as follows:
        \begin{align*}
        p^*:\mathfrak{X}_s & \rightarrow\mathfrak{X}_{\mathfrak{s}_H}\\
        (T,\chi)_G & \rightarrow (T_H,\chi\circ p)_H
        \end{align*}
        where $\chi$ is a smooth character extending some $W(F)$-conjugate of $^0\chi$. Note that $p^*$ is injective. Then the map $\phi\rightarrow \bar{\phi}$ is the surjective morphism from $\mathbb{C}[\mathfrak{X}_{\mathfrak{s}_H}]$ to $\mathbb{C}[\mathfrak{X}_{\mathfrak{s}}]$ induced by $p^*$, which is also denoted by $p$. And clearly, the morphism $p^*$ is compatible with the map $N_r^*$ defined in \Cref{base change homomorphism}. Therefore, the map $\phi\rightarrow \bar{\phi}$ determines a surjective homomorphism from $\mathcal{Z}(H_r, \rho_{H_r})$ to $\mathcal{Z}(G_r,\rho_r)$ which is compatible with the base change homomorphism.
        
    \end{proof}

    From the two lemmas above, we can see that the fundamental lemma for $H$ implies the fundamental lemma for $G$.

    \item[(b)] We may assume $\gamma$ is elliptic from \Cref{descent formula section}.

    \item[(c)] We may assume $\gamma$ is regular. This is proved in [\cite{clozel1990fundamental}, Prop 7.2].

    \item[(d)] We may assume that $G$ is such that $G_{\textnormal{der}}=G_{\textnormal{sc}}$ and $Z(G)$ is an induced torus. We work under assumptions (a)-(c). The proof is generalized as follows:

    Construct the following exact sequence 
    $$1\rightarrow G\rightarrow G'\rightarrow Q\rightarrow1$$
    as in \cite[\S 6.1]{clozel1990fundamental}. Here $G'$ and $G$ are unramified groups with the properties that $G_{\textnormal{der}}=G_{\textnormal{sc}}= G'_{\textnormal{der}}=G'_{\textnormal{sc}}$ and $Z(G')$ is an induced torus. Note that this exact sequence is not exact on $F$-points, but the first map is injective on $F$-points. Let $T'$ be the maximal torus in $G'$ corresponding to $T\subset G$. Since the finite Weyl group of $G$ and $G'$ are isomorphic, we can embed the extended affine Weyl group $\widetilde{W}$ of $G$ into the extended affine Weyl group $\widetilde{W}'$ of $G'$. Consequently, we can embed $\widetilde{W}_{^0\chi}\backslash {^0T}$ into $\widetilde{W}_{^0\chi}\backslash {^0T}'$. Therefore, we can choose an arbitrary $\widetilde{W}_{^0\chi}$-invariant character $^0\chi'$ of $^0T'$ extending $^0\chi$. Denote the Bernstein block corresponding to $^0\chi'$ by $\mathfrak{s}'=({^0T'},{^0\chi'})$, and let $(J',\rho')$ denote the type corresponding to $\mathfrak{s}'$. From the construction of the exact sequence, we can see that $J'=J\cdot {^0T'}$ and $\rho'|J=\rho$. Choose a set of representatives $\{n_w\}_{w\in \widetilde{W}_{^0\chi}}$ of $\widetilde{W}_{^0\chi}$ in $G$. Since the character $^0\chi'$ is $\widetilde{W}_{^0\chi}$-invariant, we can see that $\widetilde{W}_{^0\chi}\subset \widetilde{W}_{^0\chi'}'$. Let $1_{n_w}$ (resp. $1_{n_w}')$ denote the function in $\mathcal{H}(G,\rho)$ (resp. $\mathcal{H}(G',\rho')$) supported on $Jn_wJ$ (resp. $J'n_wJ'$) and taking value 1 at the point $n_w$. Note that the set $\{1_{n_w}|\;w\in \widetilde{W}_{^0\chi}\}$ form a basis of $\mathcal{H}(G,\rho)$. We define a linear transformation $L:\mathcal{H}(G,\rho)\rightarrow \mathcal{H}(G',\rho')$ as follows:
    \begin{align*}
        L:\mathcal{H}(G,\rho)&\rightarrow \mathcal{H}(G',\rho')\\
        1_{n_w} & \rightarrow 1_{n_w}'
    \end{align*}
    Note that this is a linear transformation, and for any function $f\in \mathcal{H}(G,\rho)$, we have $L(f)|_G=f$. Moreover, $\textnormal{Supp}(L(f))\subset G(F)\cdot {^0T'}$. Similarly we can define a linear transformation from $\mathcal{H}(G_r,\rho_r)$ to $\mathcal{H}(G_r',\rho_r')$, which is also denoted by $L$. We have the following lemma:
    \begin{lemma}\label{center induced torus}
        The linear transformation $L$ determines an algebraic homomorphism from $\mathcal{Z}(G_r,\rho_r)$ to $\mathcal{Z}(G_r',\rho_r')$, which is compatible with the base change homomorphism in the sense that 
        $$b_r(L(\phi))=L(b_r(\phi))$$
        for any $\phi \in \mathcal{Z}(G_r,\rho_r)$.
    \end{lemma}
    \begin{proof}
        We are still going to use the isomorphism $\beta:\mathbb{C}[\mathfrak{X}_{\mathfrak{s}}]\xrightarrow{\sim}\mathcal{Z}(G,\rho)$, to view an element in the Bernstein center as a regular function on the Bernstein variety. The proof will be carried out on the $F$-level, and the same argument works for the $F_r$-level. 
        
        Choose an arbitrary character $\chi'$ on $T'(F)$ extending $^0\chi'$, and let $\chi=\chi'|_{T(F)}$ be a character on $T(F)$ extending $^0\chi$. Pick an arbitrary function $h\in (i_{B'}^{G'}\chi')^{\rho'}$, and for any $f\in \mathcal{Z}(G,\rho)$ we want to evaluate the action of $L(f)$ on $h$.

        By \Cref{G=G'T}, we have $G'(F)=T'(F)\cdot G'_{\textnormal{der}}(F)=T'(F)\cdot G_{\textnormal{der}}(F)=T'(F)\cdot G(F)$. Therefore, we only need to compute the value of $(L(f)\bullet h)(g_0)$ for $g_0\in G(F)$. We have
        \begin{align}
            (L(f)\bullet h)(g_0) & = \int_{G'(F)} h(g')(L(f))^{\vee}(g'^{-1}g_0) \; \mathrm{d}g'  \label{equation 1}\\
            & = \int_{G'(F)} h(g')(L(f))(g_0^{-1}g') \; \mathrm{d}g' \nonumber\\
            & = \int_{G'(F)} (L(f))(g')h(g_0g') \; \mathrm{d}g' \nonumber\\
            & = \int_{G(F)\cdot {^0T'}} (L(f))(g')h(g_0g') \; \mathrm{d}g' \nonumber\\
            & = \int_{G(F)} \int_{^0T'}(L(f))(gt)h(g_0gt) \; \mathrm{d}t\mathrm{d}g\label{equation 2}\\
            & = \int_{G(F)} (L(f))(g)h(g_0g) \; \mathrm{d}g \nonumber\\
            & = \int_{G(F)} f(g)h(g_0g) \; \mathrm{d}g \nonumber\\
            & = (f\bullet h|_{G(F)})(g_0) \label{equation 3}
        \end{align}
        Here, in \Cref{equation 1}, the Haar measure we are using is $\textnormal{vol}(J')=1$ and in \Cref{equation 2}, we are using Fubini's theorem and the Haar measures we are using are $\textnormal{vol}({^0T'})=1$ and $\textnormal{vol}({J})=1$. For \Cref{equation 3}, we can see that $h|_{G(F)}\in (i_B^G\chi)^{\rho}$. With this observation, we get the conclusion that $L(f)$ acts on $(i_{B'}^{G'}\chi')^{\rho'}$ by the same scalar as $f$ acting on $(i_B^G\chi)^{\rho}$. Therefore, the action induced by $L$ corresponds to the morphism $R:\mathfrak{X}_{s'}\rightarrow \mathfrak{X}_s$ defined by taking the restriction to $T(F)$. Consequently, $L$ is compatible with the base change homomorphism, and the proof is completed.
    \end{proof}

    We need one more lemma:
    \begin{lemma}\label{induced center lemma 2}
        Assume that $\delta$ is $\theta$-regular. Let $T_r$ denote the $\theta$-centralizer of $\delta$. There exists a constant $C_{T_r}$, which only depends on the elliptic torus $T_r$, such that for every $\delta\in G_r$ with elliptic regular norm in $T$, and for any $w\in \widetilde{W}_{r,{^0\chi}_r}$, we have
        \begin{equation}
            \textnormal{SO}_{\delta\theta}^{G_r}(1_{n_w})=C_{T_r}\textnormal{SO}_{\delta\theta}^{G_r'}(1_{n_w}').
        \end{equation}
        
    \end{lemma}
    \begin{proof}
        See \cite[\S 7.4]{haines2012base} for details. The proof are exactly the same as loc. cit.
    \end{proof}

    According to \cite[\S7.4]{haines2012base}, the reason why $C_{T_r}$ might not be $1$ is that we change the definition of twisted orbital integrals. For the twisted case, the orbital integral from the new definition will be $\textnormal{vol}(A_G(F)\backslash G_{\delta\theta}^{\circ})=\textnormal{vol}(A_G(F)\backslash T_r^{\circ})$ multiplied with the orbital integral from the original definition. Similarly, for the untwisted case, the constant will be $\textnormal{vol}(A_G(F)\backslash G_{\gamma}^{\circ})=\textnormal{vol}(A_G(F)\backslash T^{\circ})$. Note that in our original definition, we are choosing compatible measures $T_r^{\circ}$ and $T^{\circ}$, which are $F$-isomorphic (since tori are commutative). Therefore, the two volumes must be the same. This forces the constant $C_{T_r}$ must be the same as the constant $C_T$. Therefore, by \Cref{center induced torus} and \Cref{induced center lemma 2}, the fundamental lemma on $G'$ implies the fundamental lemma on $G$.

    \item[(e)] We may assume $\gamma$ is strongly regular elliptic. This is explained in [\cite{clozel1990fundamental}, p.292]. 

    \item[(f)] By \Cref{derived 1}, we may only consider the function of the form $\phi_{\lambda N}$ and $(b\phi)_{\lambda}$, where $\lambda$ is a character on $Z(F)$ such that $\lambda|_{^0Z(F)}={^0\chi^{-1}}|_{^0Z(F)}$.
\end{itemize}
    Conclusion: we may assume $G_{\textnormal{der}}=G_{\textnormal{sc}}$, and the center of $G$ is an induced torus, and $\gamma$ is a strongly regular elliptic semisimple element which is a norm. Furthermore, we consider the matching of orbital integrals of the pair $(\phi_{\lambda N},(b\phi)_{\lambda})$, where $\lambda$ is a character on $Z(F)$ such that $\lambda|_{^0Z(F)}={^0\chi^{-1}}|_{^0Z(F)}$. 
    \begin{rmk}
        Contrary to an assertion in \cite{haines2012base}, we cannot reduce to the adjoint case, and instead the reduction process stops at the present situation.
    \end{rmk}

\section{Labesse's elementary functions}\label{section 8}

\subsection{Definition}

    We will construct the Labesse's (twisted) elementary function adapted to our cases, mainly following \cite[\S 8]{haines2012base}. We keep the assumption that $G$ is an unramified group over $F$.

    Recall that $A$ denotes a maximal split torus in $G(F)$, $T=C_G(A)$ a maximal torus, and $B=TU$ a Borel subgroup of $G$. Let $A_r$ denote the maximal $F_r$-split torus in $T_r$. Therefore, we have defined the dominant Weyl chambers in $X_*(A_r)_{\mathbb{R}}$ and $X_*(A)_{\mathbb{R}}$. Let $\rho$ denote the half-sum of the $B$-positive absolute roots of $G$. Recall that $q$ denotes the cardinality of the residue field of $F$.

    Fix a uniformizer $\varpi$ for the field F. Consider a regular dominant cocharacter $\nu\in X_*(A^r)$ and set $u=\nu(\varpi)\in T(F_r)$. Let $\tau=N(\nu)$, a regular dominant cocharacter in $X_*(A)$, and set $t=\tau(\varpi)\in T(F)$. Thus $t=N(u)$.

    Following \cite[\S8.1]{haines2012base}, we may characterize the Levi subgroup $M_{u\theta}$ (resp. parabolic subgroup $P_{u\theta}$) as the set of elements $g\in G$ such that $(u\theta)^ng(u\theta)^{-n}$ remains bounded as n ranges over all integers (resp. all positive integers). Since $\nu$ and $t\tau$ are dominant, $M_{u\theta}=M_t=T$ and $P_{u\theta}=P_t=B$.

    Fix any $F$-Levi subgroup $M$, and let $G_r$ (resp. $M_r$) denote the $F_r$-points of $G$ (resp. $M$). Define the left $M_r$ action on $G_r\times M_r$ by $m_1(g,m)=(m_1g, m_1m\theta(m_1)^{-1})$. Let $[g,m]$ denote the equivalence class of $(g,m)$ in the quotient space $M_r \backslash G_r\times M_r$. As in loc. cit., there is natural morphism of $p$-adic analytic manifolds
    \begin{align}
        M_r \backslash G_r\times M_r & \rightarrow G_r \\
    [g,m] & \rightarrow g^{-1}m\theta(g) \nonumber
    \end{align}
    with the property that it is generically one to one and the normalized absolute value of its Jacobian at the point $[g,m]$ is $|D_{G(F)/M(F)}(\mathscr{N}(m))|_F$.

    We apply this result to the Levi subgroup $M_{u\theta}=T$. Here, we only care about the quotient by $J_r\cap T_r={^0T_r}$ of the compact open subset $J_r\times {^0T_ru}\subset G_r\times T_r$. The action of $^0T_r$ on $J_r\times {^0T_r}u$ is defined by restricting the above-defined action of $T_r$ on $G_r\times T_r$. The map
    \begin{align}\label{map}
        ^0T_r\backslash J_r\times {^0T_ru} & \rightarrow G_r \\
        [k,mu] & \rightarrow k^{-1}mu\theta(k) \nonumber
    \end{align}
    is injective and has the absolute value of its Jacobian everywhere equal to the non-zero number
    $$|\textnormal{Jac}_{[k,mu]}|_F=\delta_{B_r}^{-1}(u)=q^{r\langle 2\rho,\nu \rangle}.$$

    Thus its image is a compact open subset, with volume $\delta_{B_r}^{-1}(u)\cdot \textnormal{vol}(J_r)$. We denote this set by $\mathscr{J}_u$.

    \begin{defn}
        We define the elementary function $\phi^{u,\chi_r}$ on $G_r$ to vanish off of $\mathscr{J}_u$, and to have value on $\mathscr{J}_u$ given by
        $$\phi^{u,\chi_r}(k^{-1}mu\theta(k))=\chi_r^{-1}(m)$$
        for $k\in J_r$ and $m\in {^0T_r}$.
    \end{defn}

    To prove that $\phi^{u,\chi_r}$ is well defined, suppose
    $$k_1^{-1}mu\theta(k_1)=k_2^{-1}mu\theta(k_2),$$
    for $k_1,k_2\in J_r$ and $m_1,m_2\in {^0T_r}$. We need the following lemma as in \cite[Lemma 8.1.2]{haines2012base}:
    \begin{lemma}\label{lemma for well defined}
        Suppose $g_1, g_2\in G(\bar{F})$ and $m_1,m_2\in {^0T_r}$. Extend $\theta$  to an element of $\textnormal{Gal}(E/F)$ and use the same symbol $\theta$ to denote the induced automorphism of $G(\bar{F})$. We have the following statements:
        \begin{itemize}
            \item[(1)] If $g_1^{-1}m_1u\theta(g_1)=g_2^{-1}m_2u\theta(g_2)$, then $g_1\in T(\bar{F})g_2$.
            \item[(2)] The $\theta$-centralizer of $m_1u$ is $T(F)$; hence $m_1u$ is strongly $\theta$-regular.
            \item[(3)] If $g_1\in G_r$ and $g_1^{-1}mu\theta(g_1)$ lies in the support of $\phi^{u,\chi_r}$, then $g_1\in T_rJ_r$.
        \end{itemize}
        \end{lemma}
    \begin{proof}
         The proof given in loc. cit. for depth zero characters works for positive depth character with no changes.
    \end{proof}

    From the above lemma we can deduce that, $k=k_2k_1^{-1}\in T_r$. Thus it belongs to ${^0T_r}$ and commutes with $u$, so we have
    $$m_1k\theta(k^{-1})=m_2$$
    which implies $\chi_r(m_1)=\chi_r(m_2)$. This shows that the function $\phi^{u,\chi_r}$ is well defined.

    Let $r=1$ (hence $\theta=$id and $u=t$) we can get the analogous smooth compactly-supported function $f^{t,\chi}$ on $G(F)$. 

    \begin{lemma}
        The functions $\phi^{u,\chi_r}$ belong to $\mathcal{H}(G_r, \rho_r)$.
    \end{lemma}
    \begin{proof}
        First we need to show that the support of $\phi^{u,\chi_r}$ is $\mathscr{J}_u=J_ruJ_r$: consider the map $h:{^0T_r}\backslash J_r\times {^0T_r}u\rightarrow G_r$ defined above. We want to show that the image is exactly $J_ruJ_r$. From the discussion around \Cref{map} we can see that this map is injective. To show this map is surjective, we only need to show that the volume of the image is equal to the volume of $J_ruJ_r$. Since the Jacobian of the map is $\delta_{B_r}^{-1}(u)$, the volume of the image is just $\delta_{B_r}^{-1}(u)\cdot \textnormal{vol}(J_r)$. On the other hand, the volume of $J_ruJ_r$ is $\textnormal{vol}(J_r)\cdot \textnormal{Card}(J_{r,u}\backslash J_r)$, where $J_{r,u}=J_r\cap uJ_ru^{-1}$. Since $J_r$ is a group with Iwahori decomposition and $u$ is an element in the torus $T_r$, we can see that $J_{r,u}$ is an open compact subgroup of $J_r$ with Iwahori decomposition
        $$J_{r,u}=u(\bar{U}\cap J_r)u^{-1}\cdot {^0T_r}\cdot (J_r\cap U).$$
        The volume is just $\delta_{B_r}^{-1}(u)$, by the definition of the modulus character. Thus we prove that the support is exactly $\mathscr{J}_u=J_ruJ_r$.

        Next we need to show that for any $j\in J_r$, $\phi^{u,\chi_r}(jg)=\rho_r(j)\phi^{u,\chi_r}(g)$ (the analogue for right multiplication can be proven similarly). We split this part into two steps: the root groups part and the torus part. We start with the root groups part.

        Recall that $U_{f_{\mathfrak{s}_r}}$ denotes the subgroup of $J_r$ generated by all root groups inside it (see \Cref{main result}). Here we will omit the subindex $\mathfrak{s}_r$ and just denote it by $U_{f_r}$. It is easy to see that $U_{f_r}$ is normal in $J_r$ and $\rho_r$ is trivial on $U_{f_r}$. Also, $U_{f_r}$ has Iwahori decompositions (see \cite[Lemma 3.2]{roche1998types} for details). Therefore, we only need to check the left $U_{f_r}$-invariance of $\phi^{u,\chi_r}$. Using the normality of $U_{f_r}$ in $J_r$, it is sufficient to show that an identity of the form 
        \begin{equation}\label{elementary functions 1}
            k_1^{-1}u'm_1u\theta(k_1)=k_2^{-1}m_2u\theta(k_2)
        \end{equation}
        with $u'\in U_{f_r}$, $k_1,k_2\in J_r$, and $m_1,m_2\in {^0T_r}$, implies that
        \begin{equation}
            \chi_r(m_1)=\chi_r(m_2).
        \end{equation}
        Let $k=k_2k_1^{-1}$. Then \Cref{elementary functions 1} gives us
        \begin{equation}\label{elementary functions 2}
            m_2^{-1}ku'm_1=u\theta (k)u^{-1}
        \end{equation}
        Now write $k\in J_r$ according to the Iwahori decomposition as
        $$k=k_-\cdot k_0\cdot k_+\in (J_r\cap \bar{U})\cdot {^0T_r}\cdot (J_r\cap U)$$
        Since $J_r\cap \bar{U}=U_{f_r}\cap \bar{U}$, $J_r\cap U=U_{f_r}\cap U$, and $U_{f_r}\cap {^0T_r}$ normalizes both $U_{f_r}\cap \bar{U}$ and $U_{f_r}\cap U$, the torus part of $k$ and $k_u'$ will just differ by an element in $U_{f_r}\cap {^0T_r}$. In other words, we can find suitable elements $u_-'\in U_{f_r}\cap \bar{U}$, $u_0'\in U_{f_r}\cap {^0T_r}$ and $u_+'\in U_{f_r}\cap U$ such that
        $$ku'=k_-u_-'\cdot k_0u_0'\cdot k_+u_+'.$$
        Thus \Cref{elementary functions 2} can be written in the form
        $$m_2^{-1}(k_-u_-')m_2\cdot m_2^{-1}(k_0u_0')m_1\cdot m_1^{-1}(k_+u_+')m_1=u\theta(k_-)u^{-1}\cdot \theta(k_0)\cdot u\theta(k_+)u^{-1}$$
        By the uniqueness of the Iwahori decomposition, we can see that 
        $$m_2^{-1}(k_0u_0')m_1=\theta(k_0).$$
        Therefore, $\chi_r(m_1)=\chi_r(m_2)$ by the fact that $u_0'$ is in the kernel of $\chi_r$, and this proves the $U_{f_r}$-bi-invariance.

        The second step is to consider the action of the torus part. In other words, we need to prove that, for $k_1, k_2, m_1, m_2$ as above, and any $m_0\in {^0T_r}$, if they satisfy
        \begin{equation}\label{elementary functions 3}
            m_0k_1^{-1}m_1u\theta(k_1)=k_2^{-1}m_2u\theta(k_2),
        \end{equation}
        then $\chi_r(m_2)=\chi_r(m_0m_1)$. Notice that ${^0T_r}$ normalizes all root groups, so we can write $k_1m_0k_1^{-1}=u'm_0$ for some $u'\in U_{f_r}$. Therefore, we can rewrite \Cref{elementary functions 3} as
        $$k_1^{-1}u'm_0m_1u\theta(k_1)=k_2^{-1}m_2u\theta(k_2)$$
        Now the argument above applied to this yields $\chi_r(m_2)=\chi_r(m_0m_1)$. The proof is completed.
    \end{proof}

\subsection{Orbital integrals}

    In this part we will compute the orbital integral of $\phi^{u,\chi_r}$, which will finally show that $\phi^{u,\chi_r}$ and $f^{t,\chi}$ are associated for all $u$.
    
    Normalize the Haar measure $dg$ on $G_r$ (resp. $ds$ on $T_r$) so that $J_r$ (resp. ${^0T_r}$) has measure $1$. Use this to define the quotient measures $d\bar{g}$ on $T(F)\backslash G(E)$ (resp. $d\bar{s}$ on $T(F)\backslash T(E)$). We have the following important proposition about the orbital integral of $\phi^{u,\chi_r}$:
    \begin{prop}[{\cite[Proposition 8.3.1]{haines2012base}}]\label{orbital integral value}

        The following statements hold:\\
        (a) For $\delta\in G_r$, $\textnormal{TO}_{\delta\theta}(\phi^{u,\chi_r})$ is non-zero if and only if $\delta$ is $\theta$-conjugate in $G_r$ to an element of the form $mu$, $m\in {^0T_r}$. For such $m$ we have
        \begin{equation}
            \textnormal{TO}_{mu\theta}(\phi^{u,\chi_r})=\chi_r^{-1}(m).
        \end{equation}
        (b) For $\gamma\in G(F)$, $\textnormal{O}_{\gamma}(f^{t,\chi})$ is non-zero if and only if $\gamma$ is conjugate in $G(F)$ to an element of the form $m_0t$, $m_0\in {^0T}$. For such $m_0$ we have
        \begin{equation}
            \textnormal{O}_{m_0t}(f_{t,\chi})=\chi^{-1}(m_0).
        \end{equation}
    \end{prop}
    \begin{proof}
        It is enough to just prove (1). If $\textnormal{TO}_{\delta\theta}(\phi^{u,\chi_r})\neq 0$, then $\delta$ is $\theta$-conjugate under $G_r$ to an element of the form $mu$, where $m\in {^0T_r}$. Therefore, we can just replace $\delta$ by $mu$. According to \Cref{lemma for well defined}, we have that $G_{\delta\theta}^{\circ}(F)=G_{mu\theta}^{\circ}(F)=T(F)$. Now if $g\in G_r$ is such that $g^{-1}mu\theta(g)\in \textnormal{supp}(\phi^{u,\chi_r})$, then again by \Cref{lemma for well defined}, we get that $g\in T_rJ_r$. Therefore, using our choice of Haar measure, we have the following equation
        \begin{equation}
            \begin{aligned}
                \int _{T(F)\backslash G_r}\; \phi^{u,\chi_r}(g^{-1}mu\theta(g))d\bar{g} & = \int _{T(F)\backslash T_r}\; \phi^{u,\chi_r}(s^{-1}mu\theta(s))d\bar{s} \\ &= \int _{T(F)\backslash T_r}\; \phi^{1,\chi_r}_{T_r}(s^{-1}m\theta(s))d\bar{s}
            \end{aligned}
        \end{equation}
        where $\phi^{1,\chi_r}_{T_r}$ denotes the elementary function on $T_r$ associated to $u=1$. This is exactly the characteristic function $1_{{^0T_r}}$ of ${^0T_r}$ multiplied by the character $\chi_r^{-1}$. Therefore we have proved that
        $$\textnormal{TO}_{mu\theta}^{G_r}(\phi^{u,\chi_r})=\chi_r^{-1}(m)\textnormal{TO}_{m\theta}^{T_r}(1_{{^0T_r}}).$$
        Using the fact that
        \begin{equation}
            H^1(F_r/F,{^0T_r}))=0,
        \end{equation}
        (see \cite[Lemma 8.2.1]{haines2012base} for details) we can see that $\textnormal{TO}_{m\theta}^{T_r}(1_{{^0T_r}})=\textnormal{vol}_{d\bar{s}}(T(F)\backslash T(F){^0T_r})$. By our choice of measure, this volume equals 1. The proof is completed.
    \end{proof}

    \begin{lemma}\label{conjugation lemma}
        We have the following statements.
        \begin{itemize}
            \item[(1)] Let $m_0,m_0'\in {^0T}$. Then $m_0t$ and $m_0't$ are stably-conjugate in $G$ if and only if $m_0=m_0'$.
            \item[(2)] Let $m,m'\in {^0T_r}$. Then $mu$ and $m'u$ are stably $\theta$-conjugate in $G_r$ if and only if they are $\theta$-conjugate by an element of ${^0T_r}$.
        \end{itemize}
    \end{lemma}
    \begin{proof}
        See [\cite{haines2012base}, Lemma 8.3.2].
    \end{proof}

    \begin{cor}\label{stable orbital integral value}
        For every $\delta\in G_r$, $\textnormal{SO}_{\delta\theta}(\phi^{u,\chi_r})=\textnormal{TO}_{mu\theta}(\phi^{u,\chi_r})$, where $mu$, $m\in {^0T_r}$, is any element of this form which is stably $\theta$-conjugate to $\delta$.
    \end{cor}
    \begin{proof}
        If $\delta$ is not stably $\theta$-conjugate to any element of the form $mu$, $m\in {^0T_r}$, then both sides are zero. Otherwise, just let $\delta=mu$. Then this will be the only term appear in the stable orbital integrals by \Cref{conjugation lemma}.
    \end{proof}

    \begin{thm}\label{associated}
        The functions $\phi^{u,\chi_r}$ and $f^{t,\chi}$ are associated.
    \end{thm}
    \begin{proof}
        First, we need to show that for any $\gamma\in G$ not a norm, the stable orbital integral $\textnormal{SO}_{\gamma}(f^{t,\chi})=0$. That is because, if the stable orbital integral is not zero, then by \Cref{conjugation lemma}, we can assume that $\gamma$ is stably conjugate to an element of the form $m_0t$, with $m\in {^0T}$. We already know that $t=N(u)$ is a norm, and the norm map $N:{^0T_r}\rightarrow {^0T}$ is surjective, so there exists $m\in {^0T_r}$ such that $N(m)=m_0$. So $\gamma$ is a norm, contrary to our assumption.

        Next, we may assume that $\gamma=N(\delta)$ is a norm of an element $\delta\in G_r$. If $\textnormal{SO}_{\delta\theta}(\phi^{u,\chi_r})\neq 0$, then $\delta$ must be $\theta$-conjugate to an element of the form $mu$, with $m\in {^0T_r}$, and by \Cref{orbital integral value} and \Cref{stable orbital integral value}, we get the following equation:
        \begin{equation}\label{orbital integral of elementary function}
            \textnormal{SO}_{\delta\theta}(\phi^{u,\chi_r})=\textnormal{TO}_{mu\theta}(\phi^{u,\chi_r})=\chi_r^{-1}(m).
        \end{equation}
        Let $m_0=N(m)$. We can see that $\gamma$ is stably conjugate to the element $m_0t$, and by the same reason, we get that
        \begin{equation}
            \textnormal{SO}_{\gamma}(f^{t,\chi})=\textnormal{O}_{m_0t}(f^{t,\chi})=\chi^{-1}(m_0),
        \end{equation}
        which coincides with \Cref{orbital integral of elementary function}.

        If $\textnormal{SO}_{\delta\theta}(\phi^{u,\chi_r})=0$, then $\delta$ is not stably $\theta$-conjugate to any element of the form $mu$, with $m\in {^0T_r}$. But that will make $\gamma=N(\delta)$ not stably conjugate to any element of the form $m_0t$, with $m_0\in {^0T}$. That is because, if $\gamma$ is stably conjugate to $m_0t$, then by using the same reasoning in the first step, we can see that $\textnormal{SO}_{\delta\theta}(\phi^{u,\chi_r})\neq 0$, contrary to our assumption. Thus, $\textnormal{SO}_{\gamma}(\phi_{t,\chi})=0$, which finishes the proof.
    \end{proof}

\subsection{An important way to identify the set $\mathfrak{R}_{^0\chi_r}(G_r)$}
    Let $\mathfrak{R}_{^0\chi_r}(G_r)$ denote the subcategory of $\mathfrak{R}(G_r)$ corresponding to the Bernstein block $\mathfrak{s}_r=({^0T_r},{^0\chi_r})$. Let $\Pi$ be a $\theta$-stable admissible representation of $G_r$, and we may fix an intertwiner $I_{\theta}:\Pi
    \;\widetilde{\rightarrow}\;\Pi^{\theta}$. Let $\Theta_{\Pi\theta}$ denote the locally integrable function of Harish-Chandra representing the functional $\phi\mapsto\langle \text{trace}\;\Pi I_{\theta},\phi\rangle$, for $\phi\in C_c^{\infty}(G_r)$. This means that
    $$\langle \textnormal{trace}\;\Pi I_{\theta},\phi\rangle=\int_{G_r}\Theta_{\Pi\theta}(g)\phi(g)\; dg.$$
    We want to compute this trace for $\phi=\phi^{u,{^0\chi_r}}$. Let $d\bar{k}$ denote the quotient measure on ${^0T_r}\backslash J_r$ denoted by $d\bar{g}$ earlier. The change of variable formula yields
    \begin{align}\label{characters of a function}
        \langle \textnormal{trace}\;\Pi I_{\theta},\phi^{u,{^0\chi_r}}\rangle & = \int_{{^0T_r}\backslash J_r}\int_{{^0T_r}}|\textnormal{Jac}_{[k,mu]}|_F\cdot \Theta_{\Pi\theta}(k^{-1}mu\theta(k))\cdot \phi^{u,{^0\chi_r}}(k^{-1}mu\theta(k))\; dmd\bar{k} \notag\\ & = q^{r\langle 2\rho,\nu \rangle}\int_{{^0T_r}} \Theta_{\Pi\theta}(mu){^0\chi_r^{-1}}(m)\; dm \notag\\ & = q^{r\langle 2\rho,\nu \rangle}\int_{{^0T_r}} \Theta_{\Pi_U\theta}(mu){^0\chi_r^{-1}}(m)\; dm.
    \end{align}
    Here $\Pi_U$ is the Jacquet module of $\Pi$ corresponding to the Borel subgroup $B=TU$, and the equality $\Theta_{\Pi\theta}(mu)=\Theta_{\Pi_N\theta}(mu)$ we used is the twisted version due to \cite{rogawski1988trace} of a theorem of \cite{casselman1977characters}.
    \begin{lemma}\label{vanishing trace}
        Suppose $\Pi$ is an irreducible $\theta$-stable representation of $G_r$. If $\langle \textnormal{trace}\;\Pi I_{\theta},\phi^{u,{^0\chi_r}}\rangle\neq 0$, then $\Pi$ belongs to the subcategory $\mathfrak{R}_{^0\chi_r}(G_r)$.
    \end{lemma}
    \begin{proof}
        The non-vanishing of the last formula of \Cref{characters of a function} shows that $\Pi_U^{^0\chi_r}\neq 0$. Recall that the type $(J_r,\rho_r)$ we construct is a $G_r$-cover of $({^0T_r},{^0\chi_r})$ (see \Cref{concrete construction of types}). Then by the isomorphism 
        $$\Pi^{\rho_r}\;\widetilde{\rightarrow} \;\;\Pi_U^{^0\chi_r},$$
        (see \cite[Remark 7.8]{roche1998types} for details) we know that $\Pi^{\rho_r} \neq 0$, which means that $\Pi\in \mathfrak{R}_{^0\chi_r}(G_r)$, by the definition of types.
    \end{proof}

\subsection{Computation of traces of elementary functions}\label{all computation}
    In this subsection, we write $W$ (resp. $W_r$) for the relative Weyl group associated to the maximal $F$-split torus $A$ (resp. $F_r$-split torus $A_r$) in $G$.

    For any irreducible $\theta$-stable representation $\Pi$ of $G_r$ with intertwiner $I_{\theta}$, we will compute the trace
    $$\langle \textnormal{trace}\; \Pi I_{\theta},\phi^{u,{^0\chi_r}}\rangle.$$
    By \Cref{vanishing trace}, we may assume $\Pi$ belongs to the subcategory $\mathfrak{R}_{^0\chi_r}(G_r)$. Suppose the supercuspidal support of $\Pi$ is $(T_r,\xi')_{G_r}$, where $\xi'$ is an extension of a $W_r$-conjugate of $^0\chi_r$.

    Let $\Xi$ denote the set of characters on $T_r$ which extend some $W_r$-conjugate of $^0\chi_r$. Thus $\Xi$ contains $W_r\xi'$. Let $\Xi({^0\chi_r})\subset \Xi$ consists of those whose restrictions to ${^0T_r}$ is exactly $^0\chi_r$. Let $\Xi^{\theta}$ (resp. $\Xi({^0\chi_r})^{\theta}$) denote the subset of $\theta$-fixed elements in $\Xi$ (resp. $\Xi({^0\chi_r})$). From now on, $\xi'$ will denote an arbitrary element of $\Xi$, not just the one associated to $\Pi$.

    Next we fix further notation related to $\Pi$. Let $\Pi_U$ denote the Jacquet module of $\Pi$ relative to $U_r$. Since $(\Pi^{\theta})_U=(\Pi_U)^{\theta}$, the intertwiner induces an intertwiner on the Jacquet module, for which we will still use $\theta$ to represent. Recall that $\Pi_U$ is a subquotient of 
    $$(i_{B_r}^{G_r}(\xi'))_U=\delta_{B_r}^{1/2}\bigoplus_{w\in W_r} \mathbb{C}_{^w\xi'},$$
    where $\mathbb{C}_{^w\xi'}$ is the $1$-dimensional representation of $T_r$ corresponding to the character $^w\xi'$. Thus there is a well-defined subset $\Xi(\Pi)\subset \Xi$ and positive multiplicities $a_{\xi',\Pi}=:a_{\xi'}$ for $\xi'\in \Xi(\Pi)$ such that, as $T_r$-representations,
    \begin{equation}\label{decomposition of jacquet module}
        \Pi_U=\delta_{B_r}^{1/2}\bigoplus_{\xi'\in \Xi (\Pi)} \mathbb{C}_{\xi'}^{a_{\xi'}}.
    \end{equation}
    Since $\Pi_U^{\theta}$ is $\theta$-stable, we can see that the set $\Xi(\Pi)$ is stabilized by $\theta$.

    Furthermore, we define $\Xi(\Pi)^{\theta}:=\Xi^{\theta}\cap \Xi(\Pi)$ and $\Xi(\Pi,{^0\chi_r})^{\theta}:=\Xi({^0\chi_r})^{\theta}\cap \Xi(\Pi)$. For $\xi'\in \Xi^{\theta}$, we define
    $$\textnormal{tr}(I_{\theta},\Pi,\xi'):=\langle \textnormal{trace}\; I_{\theta};\delta_{B_r}^{1/2}\mathbb{C}_{\xi'}^{a_{\xi'}}\rangle,$$
    the trace of the linear transformation $I_{\theta}$ on the subspace $\mathbb{C}_{\xi'}^{a_{\xi'}}$ (since $\xi'$ is $\theta$-fixed).

    Finally, for any $\xi'\in \Xi({^0\chi_r})$, we define $\eta'$ to be the unramified part of $\xi'$. That is, since we already fix a uniformizer $\varpi\in F_r$, we can write every element $t\in T_r$ uniquely as the form $t=t_0\cdot t'$, where $t_0\in {^0T_r}$ and $t'=\tau(\varpi)$ for a cocharacter $\tau\in X_*(A_r)$. And we define $\eta'(t)=\xi'(t)/{^0\chi_r}(t_0)$ for $t,t_0$ as above. 
    \begin{prop}\label{split characters}
        In the notation above, we have
        \begin{equation}
            \langle \textnormal{trace}\; \Pi I_{\theta},\phi^{u,{^0\chi_r}}\rangle= \delta_{B_r}^{-1/2}(u)\; \sum_{\xi'\in \Xi(\Pi,{^0\chi_r})^{\theta}}\; \eta'(u)\;\textnormal{tr}(I_{\theta},\Pi,\xi').
        \end{equation}
    \end{prop}
    \begin{proof}
        We rewrite the integral in terms of \Cref{characters of a function}. Combining it with \Cref{decomposition of jacquet module}, we have the following form:
        \begin{align*}
            \langle \textnormal{trace}\; \Pi I_{\theta},\phi^{u,{^0\chi_r}}\rangle & =\delta_{B_r}^{-1}(u)\int_{{^0T_r}} \Theta_{\Pi_U\theta}(mu){^0\chi_r}^{-1}(m)\; dm \\ & =\delta_{B_r}^{-1}(u)\int_{{^0T_r}} {^0\chi_r}^{-1}(m)\langle \textnormal{trace}\;\Pi_U(mu)I_{\theta}\; ; \;\delta_{B_r}^{1/2}\bigoplus_{\xi'\in \Xi (\Pi)} \mathbb{C}_{\xi'}^{a_{\xi'}}\rangle\; dm
        \end{align*}
        The component $\delta_{B_r}^{1/2}\mathbb{C}_{\xi'}^{a_{\xi'}}$ can contribute to the trace only if $\xi'\in \Xi(\Pi)^{\theta}$, and in that case $\Pi_UI_{\theta}$ preserves that component. Therefore, we can write the integral in the following form:
        \begin{equation}\label{computing the traces}
            \delta_{B_r}^{-1}(u)\sum_{\xi'\in \Xi(\Pi)^{\theta}}\int_{{^0T_r}} {^0\chi_r}^{-1}(m)\langle \textnormal{trace}\;\Pi_U(mu)I_{\theta}\; ; \;\delta_{B_r}^{1/2} \mathbb{C}_{\xi'}^{a_{\xi'}}\rangle\; dm
        \end{equation}
        For each component $\delta_{B_r}^{1/2} \mathbb{C}_{\xi'}^{a_{\xi'}}$, $\Pi_U(mu)$ acts by the scalar
        $$\delta_{B_r}^{1/2}(mu)\;\xi'(mu)=\delta_{B_r}^{1/2}(u)\xi'(m)\eta'(u).$$
        Thus \Cref{computing the traces} can be written as
        $$\delta_{B_r}^{-1/2}(u)\sum_{\xi'\in \Xi(\Pi)^{\theta}}\; \eta'(u)\;\textnormal{tr}(I_{\theta},\Pi,\xi')\;\int_{{^0T_r}}\; {^0\chi_r}^{-1}(m)\xi'(m)\;dm$$
        And the integral on the right side in non vanishing (and equal to $1$) if and only if the restriction of $\xi'$ on ${^0T_r}$ is precisely $^0\chi_r$, which means that $\xi'\in \Xi(\Pi,{^0\chi_r})^{\theta}$. The proof is completed.
    \end{proof}

    We shall also need the following result. Recall the idempotent $e_{\rho_r}\in \mathcal{H}(G_r,\rho_r)$ is defined to have support $J$ and to take value $\rho_r(k)^{-1}$ at $k\in J_r$.

    \begin{lemma}\label{idempotent trace F_r}
        In the notation above, we have
        \begin{equation}\label{computing trace 2}
            \sum_{\xi'\in \Xi(\Pi,{^0\chi_r})^{\theta}}\;\textnormal{tr}(I_{\theta},\Pi,\xi')=\langle\textnormal{trace}\;\Pi I_{\theta},e_{\rho_r}\rangle.
        \end{equation}
    \end{lemma}
    \begin{proof}
        The right hand side can be written as
        $$\int_{J_r}\langle\textnormal{trace}\; \Pi(g)I_{\theta},\;\Pi^{\rho_r}\rangle \;e_{\rho_r}(g)\; dg=\langle\textnormal{trace}\; I_{\theta},\Pi^{\rho_r}\rangle.$$
        By the isomorphism $\Pi^{\rho_r}\;\widetilde{\rightarrow}\;\Pi_U^{^0\chi_r}$ and \Cref{decomposition of jacquet module}, this can be written as
        $$\langle\textnormal{trace}\; I_{\theta},\delta_{B_r}^{1/2}\bigoplus_{\xi'\in \Xi(\Pi,{^0\chi_r})^{\theta}} \mathbb{C}_{\xi'}^{a_{\xi'}}\rangle,$$
        which is the left hand side of \Cref{computing trace 2}.
    \end{proof}
    \begin{cor}\label{idempotent trace}
        We have
        \begin{equation}
            \langle \textnormal{trace}\; \Pi I_{\theta},\phi^{u,{^0\chi_r}}\rangle|_{u=1}=\langle\textnormal{trace}\;\Pi I_{\theta},e_{\rho_r}\rangle.
        \end{equation}
        In particular, if $r=1$, we have 
        \begin{equation}
            \sum_{\xi\in \Xi(\pi,{^0\chi})}\;a_{\xi,\pi}=\langle \textnormal{trace} \; \pi,e_{\rho}\rangle.
        \end{equation}
    \end{cor}

\subsection{Labesse's elementary functions adapted to $\mathfrak{R}_{^0\chi_r,\lambda N}(G_r)$}

    In this subsection, we are going to construct Labesse's elementary functions adapted to the category $\mathfrak{R}_{^0\chi_r,\lambda N}(G_r)$.

    First, we should define the category $\mathfrak{R}_{^0\chi_r,\lambda N}(G_r)$. Let $\lambda:Z(F)\rightarrow \mathbb{C}^{\times}$ denote a character on the center $Z(F)$ with the property that $\lambda|_{^0Z(F)}={^0\chi^{-1}}|_{^0Z(F)}$, as in \Cref{derived 1}. We then define $\mathfrak{R}_{^0\chi_r,\lambda N}(G_r)$ to be the subcategory of $\mathfrak{R}_{^0\chi_r}(G_r)$ consisting of all representations $(\pi,V)\in \mathfrak{R}_{^0\chi_r}(G_r)$ such that $\pi(z)v=\lambda^{-1} N(z)v$ for any $z\in Z(F_r)$ and $v\in V$. Similarly, we could define the Hecke algebra $\mathcal{H}_{\lambda N}(G_r,\rho_r)$ to be the set consisting of functions $\phi\in C^{\infty}(G_r)$ satisfying the following conditions:
    \begin{itemize}
        \item[(1)] $\phi(j_1gj_2)=\rho_r^{-1}(j_1)\phi(g)\rho_r^{-1}(j_2)$ for any $j_1,j_2\in J_r;$
        \item[(2)] $\phi(zg)=\lambda N(z)\phi(g)$ for any $z\in Z(F_r);$
        \item[(3)] $\textnormal{supp}(\phi)$ is compact modulo $Z(F_r).$ 
    \end{itemize}
    It is easy to see that for any functions $\phi\in \mathcal{H}(G_r,\rho_r)$, we have $\phi_{\lambda N}\in \mathcal{H}_{\lambda N}(G_r,\rho_r)$.

    For any functions $\phi\in \mathcal{H}_{\lambda N}(G_r,\rho_r)$ and any representation $(\pi,V)\in \mathfrak{R}_{^0\chi_r,\lambda N}(G_r)$, we may define the action $\pi(\phi)$ on $V$ as follows:
    \begin{equation}\label{definition of action}
        \pi(\phi)v=\int_{G_r/Z(F_r)} \pi(g)\phi(g)v\;\mathrm{d}g.
    \end{equation}
    This is well defined and the integral is convergent. We have the following important lemma:
    \begin{lemma}\label{two actions are the same}
        Let $\phi\in \mathcal{H}(G_r,\rho_r)$. For any $(\pi,V)\in \mathfrak{R}_{^0\chi_r,\lambda N}(G_r)$, we have
        $$\pi(\phi)v=\pi(\phi_{\lambda N})v$$ for any $v\in V$.
    \end{lemma}
    \begin{proof}
        \begin{align*}
            \pi(\phi_{\lambda N})v & = \int_{G_r/Z(F_r)} \phi_{\lambda N}(g)\pi(g)v\; \mathrm{d}g \\
            & = \int_{G_r/Z(F_r)} \int_{Z(F_r)}\phi(gz)\lambda^{-1}(z)\pi(g)v\; \mathrm{d}z\mathrm{d}g \\
            & = \int_{G_r/Z(F_r)} \int_{Z(F_r)}\phi(gz)\pi(gz)v\; \mathrm{d}z\mathrm{d}g \\
            & = \int_{G_r} \phi(g)\pi(g)v \; \mathrm{d}g \\
            & = \pi(\phi)v.
        \end{align*}
    \end{proof}

    We have an instant corollary:
    \begin{cor}
        For any $\theta$-stable admissible representation $(\Pi,V)\in \mathfrak{R}_{^0\chi_r,\lambda N}(G_r)$ and any $\phi\in \mathcal{H}(G_r,\rho_r)$, we have
        \begin{equation}\label{trace invariant formula}
            \langle \textnormal{trace}\;\Pi I_{\theta},\phi\rangle=\langle \textnormal{trace}\;\Pi I_{\theta},\phi_{\lambda N}\rangle.
        \end{equation}
    \end{cor}

    Now we could define our Labesse's elementary functions adapted to $\mathfrak{R}_{^0\chi_r,\lambda N}(G_r)$ to be all functions of the form $\phi^{u,{^0\chi_r}}_{\lambda N}$ (and when $r=1$, $f^{t,{^0\chi}}_{\lambda}$). Define $\Xi_{\lambda N}\subset \Xi$ to be those whose restriction to the center $Z(F_r)$ is $\lambda^{-1} N$, and $\Xi_{\lambda N}({^0\chi_r})=\Xi({^0\chi_r})\cap \Xi_{\lambda N}$. Note that for any irreducible representation $\Pi\in \mathfrak{R}_{^0\chi_r,\lambda N}(G_r)$, its supercuspidal support must be of the form $(T_r,\xi')_{G_r}$ where $\xi'\in \Xi_{\lambda N}$. Therefore, $^w\xi'\in \Xi_{\lambda N}$ for any $w\in W_r$. Combine this with \Cref{decomposition of jacquet module}, we can see that $\Xi(\Pi)\subset \Xi_{\lambda N}$. Thus all discussions in \Cref{all computation} are valid if we substitute $\phi^{u,{^0\chi_r}}$ by $\phi^{u,{^0\chi_r}}_{\lambda N}$.
    \begin{lemma}
        The functions $\phi^{u,{^0\chi_r}}_{\lambda N}$ and $f_{\lambda}^{t,{^0\chi}}$ are associated.
    \end{lemma}
    \begin{proof}\label{associated 2}
        This is directly from \Cref{derived 1} and \Cref{associated}.
    \end{proof}

\section{Proof in the strongly regular elliptic case}\label{section 9}
\subsection{Local data adapted to $\mathfrak{R}_{^0\chi_r,\lambda N}(G_r)$}

    In this subsection, we assume $G$ is any unramified group over $F$ with center an induced torus and $G_{\textnormal{sc}}=G_{\textnormal{der}}$. Let $\textnormal{Irr}_{^0\chi,\lambda}(G)$ (resp. $\textnormal{Irr}_{^0\chi_r,\lambda N}^{\theta}(G_r)$) denote the set of irreducible (resp. irreducible $\theta$-stable) admissible representations in $\mathfrak{R}_{^0\chi,\lambda}(G)$ (resp. $\mathfrak{R}_{^0\chi_r,\lambda N}(G_r)$).

    Following \cite[\S11]{Weimin}, we define the local data adapted to $\mathfrak{R}_{^0\chi,\lambda}(G)$ to consist of the data (a), (b) and (c), subject to conditions (1) and (2) as follows:

    \begin{itemize}
    \item[(a)] An indexing set $\mathscr{I}$ (possibly infinite);

    \item[(b)] A collection of complex numbers $a_i(\pi)$ for $i\in \mathscr{I}$ and $\pi\in \textnormal{Irr}_{^0\chi,\lambda}(G))$;

    \item[(c)] A collection of complex numbers $b_i(\Pi)$ for $i\in \mathscr{I}$ and $\Pi\in \textnormal{Irr}_{^0\chi_r,\lambda N}^{\theta}(G_r).$
    \end{itemize}

    $\hspace{\fill}$

    \begin{itemize}
    \item[(1)] For $i$ fixed, the constants $a_i(\pi)$ and $b_i(\Pi)$ are zero for all but infinitely many $\pi$ and $\Pi$.

    \item[(2)] For $\phi\in \mathcal{H}_{\lambda N}(G_r,\rho_r)$ and $f\in \mathcal{H}_{\lambda}(G,\rho)$, the following two statements are equivalent:

            \begin{itemize}
            \item[(A)] For all $i$, we have $$\sum_{\pi\in \textnormal{Irr}_{^0\chi,\lambda}(G))}\; a_i(\pi)\; \langle \textnormal{trace}\;\pi,f\rangle=\sum_{\Pi\in \textnormal{Irr}_{^0\chi_r,\lambda N}^{\theta}(G_r)}\; b_i(\Pi)\; \langle \textnormal{trace}\;\Pi I_{\theta},\phi\rangle;$$

            \item[(B)] For all strongly regular elliptic semisimple norms $\gamma=\mathscr{N}(\delta)$, we have
            $$\textnormal{SO}_{\gamma}(f)=\textnormal{SO}_{\delta\theta}(\phi).$$
            \end{itemize}
    \end{itemize}
    Thus by using local data, we can the change problem from orbital integrals to computing traces.
    \begin{rmk}
        A word about the intertwiners $I_{\theta}$: these are local intertwiners coming from canonically defined intertwiners at the global (adelic) level. But the global-to-local process involves some choices and hence the $I_{\theta}$ here are not canonically defined. However, since $\Pi$ is irreducible, $I_{\theta}$ is defined up to a scalar, and such scalar can be absorbed into the coefficients $b(\Pi)$. Hence we are free to normalize the $I_{\theta}$ however we like.
    \end{rmk}

\subsection{End of proof}

    We follow all the assumptions in \Cref{reduction steps}. Fix $\phi\in \mathcal{Z}(G_r,\rho_r)$ and $f=b\phi\in \mathcal{Z}(G,\rho)$. We prove that for all characters $\lambda:Z(F)\rightarrow\mathbb{C}^{\times}$ satisfying $\lambda|_{^0Z(F)}={^0\chi^{-1}}|_{^0Z(F)}$, the functions $\phi_{\lambda N}$ and $f_{\lambda}$ are associated.

    By the discussion of previous sections, we may assume $\gamma$ is a strongly regular elliptic semisimple norm, say $\gamma=\mathscr{N}(\delta)$. To check the identity $\textnormal{SO}_{\delta\theta}(\phi_{\lambda N})=\textnormal{SO}_{\gamma}(f_{\lambda})$, we follow the method of Labesse: we use the associated elementary functions $\phi^{u,{^0\chi_r}}_{\lambda N}$ and $f^{t,{^0\chi}}_{\lambda}$ to give, as $u$ ranges, sufficiently many character identities to establish (2)(A) in the local data for the pair $\phi_{\lambda N},f_{\lambda}$.

    We need to show that for each $i$, the identity in (2)(A) above holds. Fixing $i$ and dropping it from our notation, we need to prove that
    \begin{equation}\label{final6}
        \sum_{\pi}\;a(\pi)\; \langle\textnormal{trace}\;\pi,f_{\lambda}\rangle=\sum_{\Pi}\;b(\Pi)\; \langle\textnormal{trace}\;\Pi I_{\theta},\phi_{\lambda N}\rangle.
    \end{equation}
    Since we already know the Labesse's elementary functions we constructed in the last section are associated, by ignoring the character $^0\chi$ (resp. $^0\chi_r$), we get the following equation
    \begin{equation}
        \sum_{\pi}\;a(\pi)\; \langle\textnormal{trace}\;\pi,f^t_{\lambda}\rangle=\sum_{\Pi}\;b(\Pi)\; \langle\textnormal{trace}\;\Pi I_{\theta},\phi^u_{\lambda N}\rangle.
    \end{equation}

    Let us write the left hand side using \Cref{split characters} in the case where $r=1$:
    \begin{equation}\label{final1}
        \sum_{\pi}\; a(\pi)\langle \textnormal{trace}\;\pi,f^t_{\lambda}\rangle=\delta_B^{-1/2}(t)\;\sum_{\xi_0\in \Xi_{\lambda}({^0\chi})/W_{^0\chi}}\;\sum_{\pi\in i_B^G(\xi_0)}\; a(\pi)\;\sum_{\xi\in\Xi(\pi,{^0\chi)}}\; \eta(t)\cdot a_{\xi,\pi}.
    \end{equation}
    Here $\xi_0$ ranges over any set of representatives for the $W_{^0\chi}$-orbits on the set $\Xi_{\lambda}({^0\chi})$ of characters on $T(F)$ extending $^0\chi$. Since $\eta(t)=\eta_r(u)$, \Cref{final1} is a linear combination of characters which are norms evaluated on $u=\nu(\varpi)$.

    For the right hand side, we can also write it as a linear combination of characters, and we also hope to separate the norm part and the non-norm part. We need the following lemma:
    \begin{lemma}[\cite{haines2012base}, Lemma 9.2.1]\label{final lemma}
        If $\xi_1',\xi_2'\in \Xi_{\lambda N}({^0\chi_r})^{\theta}$ are $W_{^0\chi}$-conjugate to each other, then one is a norm if and only if the other one is.
    \end{lemma}
    \begin{proof}
        If $\xi_1'$ and $\xi_2'$ are $W_{^0\chi}$ conjugate, we may assume $\xi_1'={^n\xi_2'}$, for $n\in N_GT(F)$ a representative element of an element $w\in W_{^0\chi}$. We only need to prove one direction. If $\xi_2'$ is a norm, we can assume $\xi_2'=\chi_r$ for a character $\chi\in \Xi_{\lambda}({^0\chi})$. Therefore, for any $a\in T_r$, we have
        $$\xi_1'(a)=\xi_2'(n^{-1}an)=\chi(N(n^{-1}an))=\chi(n^{-1}N(a)n)=(^n\chi)_r(a),$$
        which shows that $\xi_1'$ is also a norm.

    \end{proof}
    From the above lemma, we can rewrite the right hand side in the following way:
    \begin{align}\label{final2}
        \sum_{\Pi}\;b(\Pi)\; \langle\textnormal{trace}\;\Pi I_{\theta},\phi^u_{\lambda N}\rangle & = \notag \\
        \delta_{B_r}^{-1/2}(u) \; & \sum_{\xi_0'\in \Xi_{\lambda N}({^0\chi_r})^{\theta}/W_{^0\chi}}\;\sum_{\substack{\Pi \;\textnormal{s.t.} \\ W_{^0\chi}\xi_0'\cap\Xi(\Pi,{^0\chi_r})^{\theta}\neq \varnothing}}\; b(\Pi)\;\sum_{\substack{\xi'\in \\   W_{^0\chi}\xi_0'\cap\Xi(\Pi,{^0\chi_r})^{\theta}}}\; \eta'(u)\cdot \textnormal{tr}(I_{\theta},\Pi,\xi').
    \end{align}
    This is because, if $\Pi\in i_{B_r}^{G_r}(\xi')$, where $\xi'\in \Xi_{\lambda N}({^0\chi_r})$ is a non $\theta$-invariant character of $T_r$, then all characters in the set $W_{^0\chi}\xi'$ will not be $\theta$-invariant. Therefore, $W_{^0\chi}\xi'\cap\Xi(\Pi,{^0\chi_r})^{\theta}=\varnothing$.

    Let $(9.4)_n$ (resp. $(9.4)_{nn}$) denote the contribution to (9.4) coming from those $\xi_0'$ which are norms (resp. not norms). By \Cref{final lemma}, all (resp. none) of the terms $\eta'(u)$ appearing in $(9.4)_n$ resp. ($(9.4)_{nn}$) are (resp. are not) norms.

    Now we regard each side of the identity 
    \begin{equation}\label{final3}
        (9.3)=(9.4)_n+(9.4)_{nn}
    \end{equation}
    as a linear combination of characters on regular dominant $\nu\in X_*(A^r)$. The linear independence of characters asserts that any non-empty set of distinct characters on $X_*(A^r)$ is linearly independent over $\mathbb{C}$. And the proof shows that the set of their restrictions to the regular dominant cocharcters remains linearly independent. \Cref{final3} holds for all regular dominant $\nu$, thus by this linear independence statement, it must hold for $\nu=0$, that is, for $u=1$. In the self-evident notation, this means that we have two identities:
    \begin{align}
        (9.3)|_{t=1} & = (9.4)_n|_{u=1}; \\
        0 & = (9.4)_{nn}|_{u=1}.
    \end{align}
    In fact the linear independence gives something stronger: the contributions of each $\xi_0'\in \Xi_{\lambda N}({^0\chi_r})^{\theta}/W_{^0\chi}$, to the above equations satisfy corresponding identities. That is to say, the parts consisting of the set of $\eta$ coming from characters whose composition with the norm map have the same value on $T_r$ should be the same. Therefore, combining this with \Cref{idempotent trace}, we have the following lemma:
    \begin{lemma}
    For each $\xi_0'\in \Xi_{\lambda N}({^0\chi_r})^{\theta}/W_{^0\chi}$, we have
        \begin{align}\label{final4}
            \sum_{\substack{\xi_0\in \Xi_{\lambda}({^0\chi})/W_{^0\chi} \\ \textnormal{s.t.}\;\xi_{0_r}=\xi_0'}} \sum_{\pi\in i_B^G(\xi_0)}\; a(\pi)\langle\textnormal{trace}\;\pi,e_{\rho}\rangle & = \nonumber \\
            \sum_{\substack{\Pi \;\textnormal{s.t.} \\ W_{^0\chi}\xi_0'\cap\Xi(\Pi,{^0\chi_r})^{\theta}\neq \varnothing}}\; b(\Pi) & \;\sum_{\substack{\xi'\in \\   W_{^0\chi}\xi_0'\cap\Xi(\Pi,{^0\chi_r})^{\theta}}}\; \textnormal{tr}(I_{\theta},\Pi,\xi').
        \end{align}
    \end{lemma}
    Note that if $\xi_0'$ is not a norm, then the left hand side of \Cref{final4} is zero, and by the linear independence argument, the right hand side should also be zero.

    Next we fix any $\xi_0'\in \Xi_{\lambda N}({^0\chi_r})^{\theta}$, and sum over the $W_{^0\chi}$-orbit inside of the $W_{^0\chi_r}\xi_0'\cap \Xi_{\lambda N}({^0\chi_r})^{\theta}$. Note that by \Cref{idempotent trace F_r}, the right hand side will be simplified (all the not norm $W_{^0\chi}$-orbit will contribute nothing to \Cref{final4}). Therefore, we get 
    \begin{align}\label{final5}
        \sum_{\substack{\xi_0\in\Xi_{\lambda}(\chi)/W_{^0\chi}\\ \textnormal{s.t.}\; \xi_{0_r}\in W_{^0\chi_r}\xi_0'}}\; \sum_{\pi\in i_B^G(\xi_0)} & \; a(\pi)\langle \textnormal{trace}\; \pi, e_{\rho} \rangle \nonumber \\ & =\sum_{\Pi\in i_{B_r}^{G_r}(\xi_0')}\; b(\Pi)\langle\textnormal{trace}\; \Pi I_{\theta}, e_{\rho_r} \rangle. 
    \end{align}
    If the $W_{^0\chi_r}$-orbit of $\xi_0'$ contains no norm, then the left hand side and right hand side are both zero.

    By the definition of the base change homomorphism, we can see that $f$ acts on $i_B^G(\xi_0)^{\rho}$ by the same scalar by which $\phi$ acts on $i_{B_r}^{G_r}(\xi_0')^{\rho_r}$. By \Cref{two actions are the same}, we can see that the scalars are still the same for the actions $f_{\lambda}$ and $\phi_{\lambda N}$. Multiply this scalar on both sides of \Cref{final5}, we get a similar equation but with $e_{\rho_r}$ (resp. $e_{\rho}$) replaced by $\phi_{\lambda N}$ (resp. $f_{\lambda}$). Summing such equations over all $\xi_0'\in \Xi_{\lambda N}({^0\chi_r})/W_{^0\chi_r}$, we get the desired \Cref{final6}. This completes the proof.

    We have the following direct corollary:
    \begin{cor}
        Suppose $G$ satisfies LLC+ (see \cite[\S 5.2]{haines2014stable}). Let $(\tau,V)$ denote a finite dimensional algebraic representation on $^LG$ and $Z_V$ be the corresponding element in $\mathfrak{Z}^{\textnormal{st}}(G_r)$ defined in \Cref{geometric stable Bernstein center}. Then the two functions $Z_V*e_{\rho_r}$ and $b'(Z_V)*e_{\rho}$ are associated. 
    \end{cor}

\printbibliography

@inproceedings{roche1998types,
  title={Types and Hecke algebras for principal series representations of split reductive $p$-adic groups},
  author={Roche, Alan},
  booktitle={Annales scientifiques de l’Ecole Normale Sup{\'e}rieure},
  volume={31},
  number={3},
  pages={361--413},
  year={1998},
  organization={Elsevier}
}

@article{adler2000intertwining,
  title={An intertwining result for $p$-adic groups},
  author={Adler, Jeffrey D and Roche, Alan},
  journal={Canadian Journal of Mathematics},
  volume={52},
  number={3},
  pages={449--467},
  year={2000},
  publisher={Cambridge University Press}
}

@article{kottwitz1982rational,
author = {Robert E. Kottwitz},
title = {{Rational conjugacy classes in reductive groups}},
volume = {49},
journal = {Duke Mathematical Journal},
number = {4},
publisher = {Duke University Press},
pages = {785 -- 806},
year = {1982},
doi = {10.1215/S0012-7094-82-04939-0},
URL = {https://doi.org/10.1215/S0012-7094-82-04939-0}
}

@article{haines2009base,
author = {Thomas J. Haines},
title = {{The base change fundamental lemma for central elements in parahoric Hecke algebras}},
volume = {149},
journal = {Duke Mathematical Journal},
number = {3},
publisher = {Duke University Press},
pages = {569 -- 643},
year = {2009},
doi = {10.1215/00127094-2009-045},
URL = {https://doi.org/10.1215/00127094-2009-045}
}

@inproceedings{haines2012base,
  title={Base change for Bernstein centers of depth zero principal series blocks},
  author={Haines, Thomas J},
  booktitle={Annales scientifiques de l'Ecole normale sup{\'e}rieure},
  volume={45},
  number={5},
  pages={681--718},
  year={2012}
}

@article{yu2001construction,
  title={Construction of tame supercuspidal representations},
  author={Yu, Jiu-Kang},
  journal={Journal of the American Mathematical Society},
  volume={14},
  number={3},
  pages={579--622},
  year={2001}
}

@article{kim2017construction,
  title={Construction of tame types},
  author={Kim, Ju-Lee and Yu, Jiu-Kang},
  journal={Representation Theory, Number Theory, and Invariant Theory: In Honor of Roger Howe on the Occasion of His 70th Birthday},
  pages={337--357},
  year={2017},
  publisher={Springer}
}

@article{fintzen2021tame,
  title={Tame tori in $p$-adic groups and good semisimple elements},
  author={Fintzen, Jessica},
  journal={International Mathematics Research Notices},
  volume={2021},
  number={19},
  pages={14882--14904},
  year={2021},
  publisher={Oxford University Press}
}

@article{fintzen2021types,
  title={Types for tame $p$-adic groups},
  author={Fintzen, Jessica},
  journal={Annals of Mathematics},
  volume={193},
  number={1},
  pages={303--346},
  year={2021},
  publisher={Department of Mathematics, Princeton University Princeton, New Jersey, USA}
}

@article{fintzen2022tame,
  title={Tame cuspidal representations in non-defining characteristics},
  author={Fintzen, Jessica},
  journal={Michigan Mathematical Journal},
  volume={72},
  pages={331--342},
  year={2022},
  publisher={University of Michigan, Department of Mathematics}
}

@misc{fintzensupercuspidal,
      title={Supercuspidal representations: construction, classification, and characters}, 
      author={Jessica Fintzen},
      year={2025},
      eprint={2510.12883},
      archivePrefix={arXiv},
      primaryClass={math.RT},
      url={https://arxiv.org/abs/2510.12883}, 
}

@article{haines2014stable,
  title={The stable Bernstein center and test functions for Shimura varieties},
  author={Haines, Thomas J},
  journal={Automorphic forms and Galois representations},
  volume={2},
  pages={118--186},
  year={2014},
  publisher={Cambridge University Press}
}

@article{clozel1990fundamental,
author = {Laurent Clozel},
title = {{The fundamental lemma for stable base change}},
volume = {61},
journal = {Duke Mathematical Journal},
number = {1},
publisher = {Duke University Press},
pages = {255 -- 302},
year = {1990},
doi = {10.1215/S0012-7094-90-06112-5},
URL = {https://doi.org/10.1215/S0012-7094-90-06112-5}
}

@article{rogawski1988trace,
  title={Trace Paley-Wiener theorem in the twisted case},
  author={Rogawski, JD},
  journal={Transactions of the American Mathematical Society},
  volume={309},
  number={1},
  pages={215--229},
  year={1988}
}

@article{casselman1977characters,
  title={Characters and Jacquet modules},
  author={Casselman, William},
  journal={Mathematische Annalen},
  volume={230},
  pages={101--105},
  year={1977},
  publisher={Springer}
}

@inproceedings{borel1979automorphic,
  title={Automorphic $L$-functions},
  author={Borel, Armand},
  booktitle={Proc. Symp. Pure Math},
  volume={33},
  number={2},
  pages={27--61},
  year={1979}
}

@book{laumon1996cohomology,
  title={Cohomology of Drinfeld modular varieties, part 1, geometry, counting of points and local harmonic analysis},
  author={Laumon, G{\'e}rard},
  year={1996},
  publisher={Cambridge University Press}
}

@article{van1972computation,
  title={Computation of certain induced characters of $p$-adic groups},
  author={Van Dijk, G},
  journal={Mathematische Annalen},
  volume={199},
  number={4},
  pages={229--240},
  year={1972},
  publisher={Springer}
}

@article{haineshecke,
  title={On Hecke algebra isomorphisms and types for depth zero principal series},
  author={Haines, Thomas J},
  journal={available at www. math. umd. edu/tjh}
}

@article{kottwitz1986base,
  title={Base change for unit elements of Hecke algebras},
  author={Kottwitz, Robert E},
  journal={Compositio Mathematica},
  volume={60},
  number={2},
  pages={237--250},
  year={1986}
}

@article{labesse1990fonctions,
author = {J.-P. Labesse},
title = {{Fonctions élémentaires et lemme fondamental pour le changement de base stable}},
volume = {61},
journal = {Duke Mathematical Journal},
number = {2},
publisher = {Duke University Press},
pages = {519 -- 530},
year = {1990},
doi = {10.1215/S0012-7094-90-06120-4},
URL = {https://doi.org/10.1215/S0012-7094-90-06120-4}
}

@article{scholze2013local,
  title={The Langlands-Kottwitz approach for some simple Shimura varieties},
  author={Scholze, Peter},
  journal={Inventiones mathematicae},
  volume={192},
  number={3},
  pages={627--661},
  year={2013},
  publisher={Springer}
}

@article{kottwitz2000distributions,
  title={The distributions in the invariant trace formula are supported on characters},
  author={Kottwitz, Robert E and Rogawski, Jonathan D},
  journal={Canadian Journal of Mathematics},
  volume={52},
  number={4},
  pages={804--814},
  year={2000},
  publisher={Cambridge University Press}
}

@article{kazhdan2012endoscopic,
author = {David Kazhdan and Yakov Varshavsky},
title = {{On endoscopic transfer of Deligne–Lusztig functions}},
volume = {161},
journal = {Duke Mathematical Journal},
number = {4},
publisher = {Duke University Press},
pages = {675 -- 732},
year = {2012},
doi = {10.1215/00127094-1548371},
URL = {https://doi.org/10.1215/00127094-1548371}
}

@book{labesse1999cohomologie,
  title={Cohomologie, stabilisation et changement de base},
  author={Labesse, Jean-Pierre and Breen, Lawrence and Clozel, Laurent},
  volume={257},
  year={1999},
  publisher={Soci{\'e}t{\'e} math{\'e}matique de France}
}

@inproceedings{zelevinsky1977induced,
  title={Induced representations of reductive p-adic groups},
  author={J. N. Bernstein and A. V. Zelevinsky},
  booktitle={Annales Scient. {\'E}c. Norm. Sup.},
  volume={10},
  pages={441--472},
  year={1977}
}

@article{kottwitz1992lambda,
  title={On the $\lambda$-adic representations associated to some simple Shimura varieties},
  author={Kottwitz, Robert E},
  journal={Inventiones mathematicae},
  volume={108},
  number={1},
  pages={653--665},
  year={1992},
  publisher={Springer}
}

@article{Weimin,
    author = {Weimin Jiang},
    title = {Base change fundamental lemma for central elements in depth-zero Hecke algebras over local function fields},
    journal = {in preparation}
}


\end{document}